\documentclass[12pt]{article}

\usepackage{graphicx,hyperref,enumitem}
\usepackage{amsmath,amsthm,amssymb,color}
\usepackage[noadjust,sort]{cite}
\usepackage{tocloft,accents}
\usepackage[title]{appendix}

\makeatletter
\g@addto@macro\normalsize{%
  \setlength\abovedisplayskip{8pt plus 3pt minus 3pt}
  \setlength\belowdisplayskip{8pt plus 3pt minus 3pt}
  \setlength\abovedisplayshortskip{6pt plus 3pt minus 2pt}
  \setlength\belowdisplayshortskip{6pt plus 3pt minus 2pt}
}
\makeatother

\date{\today}

\normalsize

\numberwithin{equation}{section}

\def\({\bigl(}
\def\){\bigr)}
\def\st{\mathrel{:}}
\def\St{\mathrel{:}}
\let\leq=\leqslant   
\let\geq=\geqslant   
\newcommand{\medtilde}{\protect\accentset{\sim}}
\newcommand{\medring}{\protect\accentset{\,\circ}}

\def\ringX{\mathchoice{\medring\X}{\medring\X}{\,\ring{\!\X}}{\,\ring{\!\X}}}

\newtheorem{thm}{Theorem}[section]
\newtheorem{cor}[thm]{Corollary}
\newtheorem{lemma}[thm]{Lemma}
\newtheorem{conj}[thm]{Conjecture}
\newtheorem{assume}[thm]{Assumptions}
\def\abs#1{\lvert#1\rvert} 
\def\Abs#1{\bigl\lvert#1\bigr\rvert} \let\Card=\Abs
\def\norm#1{\lVert#1\rVert}
\def\maxnorm#1{\norm{#1}_{\mathrm{max}}}
\def\infnorm#1{\norm{#1}_\infty}
\def\onenorm#1{\norm{#1}_1}
\def\twonorm#1{\norm{#1}_2}
\def\Norm#1{\bigl\lVert#1\bigr\rVert}

\def\infNorm#1{\Norm{#1}_\infty}

\def\Babs{B_{\mathrm{abs}}}
\def\dmax{d_{\mathrm{max}}}
\def\dmin{d_{\mathrm{min}}}
\def\lmax{\ell_{\mathrm{max}}}
\def\mumin{\mu_{\mathrm{min}}}

\def\betamin{\beta_{\mathrm{min}}}
\def\betamax{\beta_{\mathrm{max}}}
\def\gammamax{\gamma_{\mathrm{max}}}

\def\dfrac#1#2{\lower0.12ex\hbox{\large$\textstyle\frac{#1}{#2}$}}
\def\Dfrac#1#2{\raise0.05ex\hbox{\small$\displaystyle\frac{#1}{#2}$}}
\def\({\bigl(}
\def\){\bigr)}
\newcommand{\stirlingii}{\genfrac{\{}{\}}{0pt}{}}
\def\Zero{\boldsymbol{0}}
\def\One{\boldsymbol{1}}
\def\ljk{\lambda_{jk}}
\def\tauQ{\tau^{}_Q}

\let\eps=\varepsilon

\def\S{\boldsymbol{S}}

\def\X{\boldsymbol{X}}
\def\Y{\boldsymbol{Y}}

\def\calC{\mathcal{C}}

\def\calL{\mathcal{L}}

\def\calB{\mathcal{B}}
\def\gammavec{\boldsymbol{\gamma}}
\def\thetavec{\boldsymbol{\theta}}
\def\deltavec{\boldsymbol{\delta}}
\def\phivec{\boldsymbol{\phi}}
\def\betavec{\boldsymbol{\beta}}
\def\betavecstar{\boldsymbol{\beta^*}}

\def\tphi{\medtilde\phi}
\def\tphivec{\boldsymbol{\tphi}}
\def\tlambda{\medtilde\lambda}
\def\avec{\boldsymbol{a}}
\def\evec{\boldsymbol{e}}
\def\fvec{\boldsymbol{f}}

\def\dvec{\boldsymbol{d}}
\def\gvec{\boldsymbol{g}}
\def\hvec{\boldsymbol{h}}

\def\vvec{\boldsymbol{v}}
\def\wvec{\boldsymbol{w}}
\def\xvec{\boldsymbol{x}}
\def\yvec{\boldsymbol{y}}
\def\zvec{\boldsymbol{z}}
\def\alphavec{\boldsymbol{\alpha}}
\def\xivec{\boldsymbol{\xi}}

\def\trans{^{\mathsf{T}}}
\def\Trans#1{\trans\mkern-#1mu }

\def\E{\operatorname{\mathbb{E}}}

\def\RG{\operatorname{RG}}
\def\Span{\operatorname{span}}
\def\Prob{\operatorname{Prob}}

\def\Cov{\operatorname{Cov}}
\def\diag{\operatorname{diag}}
\def\diam{\operatorname{diam}}

\def\Reals{{\mathbb{R}}}
\def\Complexes{{\mathbb{C}}}
\def\tr{\operatorname{tr}}

\def\nicebreak{\vskip 0pt plus 50pt\penalty-300\vskip 0pt plus -50pt }

\begin{document}

\title{Asymptotic enumeration of graph factors by cumulant expansion}

\author{
Mikhail Isaev\\
\small Department of Mathematics and Statistics\\[-0.8ex]
\small UNSW Sydney\\[-0.8ex]
\small Sydney, NSW, Australia\\
\small \tt m.isaev@unsw.edu.au
  \and
Brendan D. McKay\\
\small School of Computing\\[-0.8ex]
\small Australian National University\\[-0.8ex]
\small Canberra, ACT, Australia\\
\small\tt brendan.mckay@anu.edu.au
}

\date{}

\maketitle

\begin{abstract}
Let $G$ be a dense graph with good expansion properties and
not too close to being bipartite.
Let~$\dvec$ be a graphical degree sequence.
Under very weak conditions, we find the number of subgraphs
of~$G$ with degree sequence~$\dvec$ to arbitrary precision. 
The average degree can be any power of $n$ and the variation
in degrees can be very large.
The method uses an explicit bound on the tail
of the cumulant generating function found by the first author. 
As a first application, we prove that there is an asymptotic expansion
for the number of regular graphs and find several terms explicitly.
We believe that this is the first combinatorial application of the Fourier inversion
method for which the integral outside the dominant regions cannot
be bounded by the integral of the absolute value, and we give
a general method for dealing with that situation.
\end{abstract}

\nicebreak
\section{Introduction}\label{S:intro}

Let $G$ be a graph with $n$ vertices, assumed to be $[n]=\{1,\ldots,n\}$,
and let $\dvec=(d_1,\ldots,d_n)$ be a graphical degree sequence.
A \textit{$\dvec$-factor} of $G$ is a subgraph of~$G$ whose degree
sequence is~$\dvec$.
Define $N(G,\dvec)$ to be the number of $\dvec$-factors of~$G$.
Here, and throughout the paper, all graphs are labelled and simple.

The number $N(K_n,\dvec)$ is of course just the number of graphs
with degree sequence~$\dvec$.
This has been the object of intensive study since Read counted
cubic graphs in his 1958 PhD thesis~\cite{ReadThesis}.
Bender and Canfield~\cite{BenCan} and Wormald~\cite{WormLow}
independently handled bounded degrees, while Bollobas~\cite{BollLog}
allowed maximum degree as high as $\sqrt{2\log n}-1$.

Let $d=\frac12\sum_{j=1}^n d_j$ be the average degree,
and say that a set of degree sequences is \textit{$d$-near-regular} if it
includes all degree sequences in $[d',d'']^n$ where $[d',d'']$
contains the degrees of a random graph of average degree~$d$
with probability $1-o(1)$.
McKay allowed maximum degree $o(d^{1/4}n^{1/4})$ \cite{McKaySparse},
which was extended to $o(d^{1/3}n^{1/3})$ by McKay and
Wormald~\cite{MWsparse}, and to sparse heavy-tailed degree sequences
by Gao and Wormald~\cite{PuNick}.
Meanwhile McKay and Wormald asymptotically enumerated
$N(K_n,\dvec)$ for all $d$-near-regular degree sequences when
$\min\{d,n-d-1\} > cn/\log n$ for constant $c>\frac23$~\cite{MWreg}.
Two breakthrough results complete the $N(K_n,\dvec)$ picture up to
the present.
Barvinok and Hartigan counted graphs with average degree
$\Theta(n)$ and a range of degrees far broader than 
$d$-near-regular~\cite{BarvHart1}.
Liebanau and Wormald solved the $d$-near-regular case for
a range of average degrees wide enough to overlap previous
knowledge at both the sparse and dense ends~\cite{Liebenau}.

The main reason for generalizing $N(K_n,\dvec)$ to 
$N(G,\dvec)$ is to study the structure of random graphs.
For example, if $H$ is a subgraph of~$G$ with degree sequence~$\hvec$,
then the probability that a random $\dvec$-factor of~$G$
contains~$H$, or is edge-disjoint from~$H$, is respectively
\[
   \frac{N(G-H,\dvec-\hvec)}{N(G,\dvec)}
   \text{~~~or~~~}
   \frac{N(G-H,\dvec)}{N(G,\dvec)}.
\]
This generalization has been pursued in the sparse domain by
\cite{BollMcK,WormLow,BenCan}, with the best result being~\cite{McKaySparse}.
In the dense regime, McKay~\cite{ranx} allowed $K_n-G$ to have maximum
degree $O(n^{1/2+\eps})$ and $O(n^{1+2\eps})$ edges, for sufficiently small
$\eps>0$, provided the degree sequence is $d$-near-regular.
For average degree $\Theta(n)$, and the same wide variation of degrees
as covered by~\cite{BarvHart1}, Isaev and McKay~\cite{mother} allowed
$K_n-G$ to have maximum degree $O(n^{1/6})$ and sum of squared
degrees~$O(n)$.

Missing from the above recitation are any precise calculations for
large subgraphs and degree sequences with average degree between
$o(n^{1/3})$ and $O(n)$. We will fill that gap in this paper.
Our main results are stated explicitly in the following section.
Here we will summarize the differences between this work and
previous work.

Our estimate of $N(G,\dvec)$ allows $G$ to be any graph with
minimum degree $\Omega(n)$ that is sufficiently far from bipartite
and has moderately good expansion properties.
This is much more general than any previous work.
Our average degree~$d$ need only satisfy
$\min\{d,n-d-1\}=\Omega(n^\sigma)$ for arbitrary $\sigma>0$, which
is more general than previous work except for~\cite{Liebenau}.
However, \cite{Liebenau} only considers $G=K_n$
and $d$-near-regular degree sequences.
Our degree sequences can have widely varying degrees, matching
\cite{BarvHart1} in the case of average degree $\Theta(n)$ and
otherwise new except in the very sparse range.
Finally, our formula can provide $N(G,\dvec)$ to relative precision
$O(n^{-p})$ for any~$p$, whereas no previous work provided
even~$O(n^{-1})$.
This is significant for the computation of variances and other properties
that involve near-cancellation of large values.

From the technical point of view, our work is related to a
generalisation of the so-called
\textit{$\beta$-equations}, which appeared in disguise in~\cite{MWreg}
and explicitly in~\cite{BarvHart1}, see also~\cite{CDS}.
However, we do not require an exact solution but only a sufficient
approximation.
Our method is that of Fourier inversion, introduced into combinatorics
in~\cite{Mtourn,MWreg}.
As in the other problems solved by that method (see~\cite{mother} for
a long list), the value of the integral is concentrated in some dominant 
regions and negligible outside them.
However, unlike all previous applications we are aware of, the
contribution from outside the dominant  regions cannot be sufficiently
bounded by the integral of the absolute value of the integrand.
This requires us to invent a new technique for bounding it.
 
\nicebreak
\subsection{Statement of the main results}\label{s:results}

In this section we will present our main theorems.

 We will write ``$jk\in G$'' as a less cluttered alternative to ``$\{j,k\}\in E(G)$''.
 Let $\betavec=(\beta_1,\ldots,\beta_n)$ be a vector of
 real numbers, and for $1\leq j,k\leq n$ define
 \[
   \ljk  = \ljk (\betavec) :=
      \frac{e^{\beta_j+\beta_k}}{1+e^{\beta_j+\beta_k}}.
 \]
 Under very general conditions, which we investigate in a future paper,
there is a unique solution~$\betavecstar$ to the equations
 \begin{equation}\label{betaeqn}
    \sum_{k\st jk\in G} \ljk  (\betavecstar)= d_j \quad (1\leq j\leq n).
 \end{equation}
However, although the existence of $\betavecstar$ is implied by our
assumptions, our results are not phrased in terms of $\betavecstar$
or, explicitly, in terms of the distance of $\betavec$ from $\betavecstar$ .
Instead, we require an approximate solution of~\eqref{betaeqn} in the
sense that $\deltavec$ is small, where
$\deltavec= \deltavec(\betavec) =(\delta_1,\ldots,\delta_n)$ is defined by
\begin{equation}\label{approxbeta}
   \delta_j := \sum_{k:jk\in G} \ljk  - d_j
   \quad (1\leq j\leq n).
\end{equation}
 Define 
\[
  \lambda:=\frac{\sum_{j=1}^n d_j} {\sum_{j=1}^n g_j}  
  \quad\text{ and }\quad \varLambda :=\lambda(1-\lambda),
\]
and note that $\lambda$ is the average of $\ljk (\betavecstar)$ 
over $jk\in G$, if $\betavecstar$ exists.

The notations $\onenorm{\cdot}$, $\twonorm{\cdot}$  and $\infnorm{\cdot}$
have their usual meanings as vector norms and the corresponding
induced matrix norms.  For a matrix $M=(m_{jk})$ we will also use
$\maxnorm{M}=\max_{jk}\,\abs{m_{jk}}$ but note that it is
not submultaplicative.
 
 The \textit{signless Laplacian} of $G$ is the matrix
 $Q(G)=(q_{jk})$ where
 \[
    q_{jk} = \begin{cases}
               \,g_j, & \text{~~if $j=k$}; \\
               \,1,   & \text{~~if $jk\in G$}; \\
               \,0,   & \text{~~otherwise}.
             \end{cases}
 \]
 The least eigenvalue $q(G)$ of $Q(G)$ is called the
 \textit{algebraic bipartiteness} of~$G$ and is zero
 iff $G$ has a bipartite component~\cite{Fallat,HePan}.
 A large value of $q(G)$ implies that $G$ is far from bipartite.
 For example, at least $q(G)$ vertices or edges must be deleted
from $G$ to obtain a bipartite graph~\cite{Fallat}.
 
 For $U\subseteq [n]$,
 $\partial_G \,U$ is the set of edges of $G$ with exactly one
 end in $U$.
 The \textit{Cheeger constant}, also called the isoperimetric
 constant, of $G$ is
 \[
    h(G ) := \min_{1\leq\abs U\leq n/2} \frac{\abs{\partial_G\, U}}{\abs U}.
 \]
 A large Cheeger constant implies that $G$ has strong expansion
 properties.

 We can now define the assumptions on $(\dvec,G)$ under
 which we will be able to estimate~$N(G,\dvec)$.

 \begin{assume}\label{mainassumptions}
   There are positive constants $\sigma,B,C,\tauQ$,
   a vector $\betavec\in\Reals^n$, and a positive
   integer sequence $\dvec=(d_1,\ldots,d_n)$ with
   even sum  such that the following hold for sufficiently
   small $\eps>0$.
   \begin{itemize}[noitemsep,topsep=0pt]
      \item[(a)] $d_j\leq g_j$ for $j\in[n]$.
        \item[(b)]  $\abs{\beta_j-\beta_k} \leq C$ for $j,k\in[n]$.
      \item[(c)] $\varLambda \geq B n^{\sigma-1}$.
      \item[(d)] $h(G )\geq \varLambda^{-1}\log^2 n$ and 
          $q(G)\geq \tauQ n$ (so $G$ has
          minimum degree at least $\tauQ n$).
      \item[(e)] $\infnorm\deltavec=O(\varLambda^{1/2}n^{1/2-\sigma/2})$
         and $\onenorm\deltavec=O(n^{1+\eps})$.
   \end{itemize}   
 \end{assume}

  The second part of Assumption~(d) is equivalent to that,
 for any $(\theta_1,\ldots,\theta_n)\in\Reals^n$,
 $\sum_{jk\in G}\, (\theta_j+\theta_k)^2 \geq \tauQ n
 \sum_{j=1}^n \theta_j^2$.  Taking $\theta_j=1$ and
 $\theta_k=0$ for $k\ne j$, we see that $g_j\geq\tauQ n$.
 
 We will use $O(n^{-p})$ as a precision target, for some $p>0$.

\smallskip
Throughout the paper, asymptotic claims will be with
respect to $n\to\infty$ and sometimes involve a constant
$\eps>0$ which must be small enough.
Precisely, we will observe the
following discipline regarding the $O(\,)$ and 
similar notations.  There are values
$n_0=n_0(p,B,C,\sigma,\tauQ,\eps)$,
$\eps_0=\eps_0(p,B,C,\sigma,\tauQ)>0$,
$\check c=\check c(p,B,C,\sigma,\tauQ,\eps)$
and $\hat c=\hat c(p,B,C,\sigma,\tauQ,\eps)>0$,
\textit{depending only on the arguments listed},
such that
$O(g(n,\eps))$, respectively $\Omega(g(n,\eps))$,
represents a quantity of absolute value at most
$\check c\, g(n,\eps)$, respectively a positive
quantity of value at least $\hat c\, g(n,\eps)$,
provided $n\geq n_0$ and $0<\eps\leq \eps_0$.
The notation $\Theta(g(n,\eps))$ represents the combination
of $O(g(n,\eps))$ and $\Omega(g(n,\eps))$. 

Define $f_\lambda(z)=\ln(1+\lambda(e^{iz}-1))$ and denote
its Taylor coefficients by $c_\ell=c_\ell(\lambda)$ where
$f_\lambda(z)=\sum_{\ell=1}^\infty c_\ell z^\ell$.
The first few values are
\begin{align*}
  c_1 &= i\lambda, &
  c_2 &= -\dfrac12 \varLambda,  \\
  c_3 &= -\dfrac16 i (1-2\lambda)\varLambda,&
  c_4 &= \dfrac1{24} (1-6\varLambda)\varLambda .
\end{align*}

Define
\[
    M(G,\dvec) := \frac{\prod_{jk\in G}(1+e^{\beta_j+\beta_k})}
                                 {e^{\sum_{j=1}^n d_j\beta_j}}
                            = \frac{e^{\sum_{j=1}^n \delta_j\beta_j}}
                                      {\prod_{jk\in G} \ljk^{\ljk}(1-\ljk)_{\vphantom{k}}^{1-\ljk}}.
\]

\begin{thm}\label{mainthm}
Adopt Assumptions~\ref{mainassumptions}.  Define
\begin{align*}
      \ell_0 &:= 2\biggl\lceil \frac{1+p}{\sigma} \biggr\rceil, \quad
      r_0 := 2\biggl\lceil \frac{1+p}{\sigma} \biggr\rceil - 2,\\
      R_{\ell_0}(\thetavec) &:= i\sum_{j\in[n]}\delta_j\theta_j
       + \sum_{\ell=3}^{\ell_0}
        \sum_{jk\in G} c_\ell(\ljk ) (\theta_j+\theta_k)^\ell.
\end{align*}
Define the real symmetric matrix $A=A(G,\betavec)$ by 
$\sum_{jk\in G} \,c_2(\ljk ) (\theta_j+\theta_k)^2
= -\frac12 \varLambda n\,\thetavec\Trans6 A\thetavec$.
Then
  \begin{equation*}
    N(G,\dvec) = 
      \frac{2\,M(G,\dvec)}
                     {(2\pi \varLambda n)^{n/2}\abs{A}^{1/2}}
     \exp\biggl(\, \sum_{r=1}^{r_0} \dfrac{1}{r!} \kappa_r(R_{\ell_0}(\Y)) +O(n^{-p})\biggr),
  \end{equation*}
where $\Y$ is a Gaussian  random variable with density
$(2\pi)^{-n/2} \abs{A}^{1/2} (\varLambda n)^{n/2}
   e^{-\frac12\varLambda n\sum_{j=1}^n \yvec\Trans5 A\yvec}$
  and $\kappa_r$ denotes the $r$-th cumulant.
\end{thm}

As we have noted, we don't require~\eqref{betaeqn} to have a solution,
but instead are satisfied with a near solution in the sense of
Assumptions~\ref{mainassumptions}(e).
Explicit expressions for near solutions are not easy to come by, but
the following gives such an expression under moderately strong conditions.

\begin{thm}\label{goodbeta}
    Suppose $G$ satisfies Assumptions~\ref{mainassumptions}(a,c,d).
    Define $\betavec$ by
    \[
        \beta_j := \dfrac12\log\frac{\lambda}{1-\lambda} + \gamma_j,
    \]
    where $\gammavec:=(\gamma_1,\ldots,\gamma_n)\trans 
    = \varLambda^{-1} Q(G)^{-1}(\dvec-\lambda\gvec)$.
    Suppose
    \[
         \infnorm{\dvec-\lambda\gvec}=O(\varLambda^{3/4}n^{3/4-\sigma/4})
         \text{~~~and~~~}
         \twonorm{\dvec-\lambda\gvec}=O(\varLambda^{1/2}n^{1+\eps/2}).
    \]
   Then $\betavec$ satisfies Assumptions~\ref{mainassumptions}(b,e) with
   \begin{equation}\label{nearreg}
    \infnorm\deltavec=O(\varLambda^{-1}n^{-1})\,\infnorm{\dvec-\lambda\gvec}^2
    \text{~~~and~~~}
    \onenorm\deltavec=O\(\varLambda^{-1}n^{-1})\,\twonorm{\dvec-\lambda\gvec}^2.
   \end{equation}
\end{thm}
\begin{proof}
  By Assumption~\ref{mainassumptions}(d) and Theorem~\ref{matrixthm1}(a),
  $\infnorm{Q(G)^{-1}}=O(n^{-1})$.
    Assumption~\ref{mainassumptions}(b) is therefore satisfied
   since $\abs{\beta_j-\beta_k}\leq 2\infnorm\gammavec$ for all $j,k$.
  
If $z=O(1)$, then $(e^z-1)/(1+\lambda(e^z-1))=z+O(z^2)$
uniformly over $\lambda\in(0,1)$.
Using this and the definition of $Q$, we have for $j\in[n]$ that
\begin{align*}
    \sum_{k:jk\in G} \ljk(\betavec) &= \lambda g_j +
              \varLambda\sum_{k:jk\in G} \frac{e^{\gamma_j+\gamma_k}-1}
                                                                    {1+\lambda(e^{\gamma_j+\gamma_k}-1)} \\
     &= \lambda g_j + \varLambda \sum_{k:jk\in G}
                       \( \gamma_j+\gamma_k + O((\gamma_j+\gamma_k)^2)\) \\
     &= d_j + O(\varLambda^{-1} n)\, \infnorm{Q^{-1}(\dvec-\lambda \gvec)}^2 \\[0.5ex]
     &= d_j + O(\varLambda^{-1}n^{-1})\,\infnorm{\dvec-\lambda\gvec}^2,
\end{align*}
which proves the first part of~\eqref{nearreg}.
For the second part, note that
\begin{align*}
   \varLambda \sum_{jk\in G} O((\gamma_j+\gamma_k)^2)
       &= O(\varLambda n) \,\twonorm{\gammavec}^2 \\[-2ex]
       &= O(\varLambda^{-1}n^{-1})\,\twonorm{\dvec-\lambda \gvec}^2,
\end{align*}
since $\twonorm{Q(G)^{-1}}=O(n^{-1})$ by Assumption~\ref{mainassumptions}(d).
\end{proof}

Consider a random factor of $G$ formed by taking each edge independently
with probability~$\lambda$.  Then, provided $\varLambda\geq n^{-1+\sigma}\log^2 n$,
both the conditions on $\dvec-\lambda\gvec$ are satisfied with high probability.
Thus, Theorem~\ref{goodbeta} provides an adequate value of $\betavec$ for
studying random factors.
Note that for the case $G=K_n$, we have
\[
    \gamma_j = \frac{d_j-\lambda(n-1)}{\varLambda(n-2)}.
\]

\smallskip
A weak corollary of Theorem~\ref{mainthm} is the following,
whose discussion and proof is postponed until Section~\ref{s:observations}.

\begin{thm}\label{edgeprob}
 Suppose $G$,  $\dvec$ and $\betavec$
 satisfy Assumptions~\ref{mainassumptions} with
 $\onenorm{\deltavec}=O(\varLambda n)$.
 Let $uv$ be an edge of~$G$.
 Then the probability that a uniform random $\dvec$-factor $H$ of $G$
 contains $uv$~is
 \[
      \(1+O(\varLambda^{-1}n^{-1}(\onenorm\deltavec+1))\)\,\lambda_{uv}(\betavec).
 \]
 In particular, if~\eqref{betaeqn} has a solution~$\betavecstar$,
 then the probability is $\(1+O(\varLambda^{-1}n^{-1})\)\,\lambda_{uv}(\betavecstar)$.
\end{thm}

\medskip

The most detailed application we will present in this paper concerns the
number of $R(n,d)$ of regular graphs of order~$n$ and degree~$d$.
The main term of the asymptotic value is already known as a consequence
of the overlapping domains proved in~\cite{Liebenau,MWreg,MWsparse}.
Here we show that there is an asymptotic expansion.

\begin{thm}\label{regularthm}
There are
polynomials $p_j(x)$ for $j\geq 1$, with $p_j$ having degree
$j$ for each~$j$, such that
\begin{equation}\label{regularexpansion}
   \RG (n,d) = \sqrt2 \,
     \(\lambda^\lambda (1{-}\lambda)^{1{-}\lambda}\)^{\binom n2}
      \binom{n-1}{d}^{\!n}
   \exp\biggl( \,
        \sum_{j=1}^k \frac{p_j(\varLambda)}{\varLambda^j n^{j-1}} 
          + O(\varLambda^{-k-1}n^{-k}) \biggr)
\end{equation}
for any fixed~$k\geq 1$,
provided $\min\{d,n-d-1\}\geq n^\sigma$ for some $\sigma>0$.
\end{thm}
Theorem~\ref{regularthm} will be proved in Section~\ref{s:regular},
where we give $p_1,\ldots,p_7$ and present empirical evidence that
the condition $\min\{d,n-d-1\}\geq n^\sigma$  is stronger than necessary.

\subsection{Outline of the method}

Our overall approach follows previous work~\cite{BarvHart1,MWreg,ranx}
but the greater generality poses new challenges.
Since
$N(G,\dvec) = [x_1^{d_1}\cdots x_n^{d_n}] \,\prod_{jk\in G} (1+x_jx_k)$,
we have by Cauchy's Theorem that
\[
   N(G,\dvec) = \frac{1}{(2\pi i)^n}
     \oint\!\cdots\!\oint \,
     \frac{\prod_{jk\in G} (1+x_jx_j)}
          {\prod_{j=1}^n x_j^{d_j+1}}\, dx_1\cdots dx_n,
\]
where the contours circle the origin once anticlockwise.
Now change variables as $x_j = e^{\beta_j+i\theta_j}$.
This gives
\[
    N(G,\dvec) = (2\pi)^{-n} M(G,\dvec)\, J(G,\dvec),
\]
where
\begin{align}
  J(G,\dvec) &:= \int_{-\pi}^\pi\!\cdots\!\int_{-\pi}^\pi
                  F_{\betavec}(\thetavec)\,d\thetavec, \label{Fdefn} \\
  F_{\betavec}(\thetavec) &:= \frac{\prod_{jk\in G}\,
                    \(1+\ljk (e^{i(\theta_j+\theta_k)}-1)\)}
               {e^{i\sum_{j=1}^n d_j\theta_j}}. \notag
\end{align}
The integrand $F_{\betavec}(\thetavec)$ takes its maximum absolute
value of~1 at $\thetavec\in\{ (0,\ldots,0), (\pi,\ldots,\pi)\}$.
In Section~\ref{s:inbox} we evaluate the integral inside
small regions $\calB_0,\calB_\pi$ surrounding those two points,
using the theory from~\cite{daughter} to obtain much more
precision than previously.  In Section~\ref{s:outofbox}
we show that the integral outside $\calB_0\cup\calB_\pi$
is negligible. Contrary to all previous work, the integral
of the absolute value $\abs{F_{\betavec}(\thetavec)}$ outside
$\calB_0\cup\calB_\pi$ is \textit{not} small compared to the integral of
$F_{\betavec}(\thetavec)$ inside $\calB_0\cup\calB_\pi$, which
requires us to invent new techniques.

\nicebreak
\section{Cumulants}\label{s:cumulants}

For the basic theory of cumulants, we suggest the paper
of Speed~\cite{Speed}.  We will make much use of the
mixed-cumulant formulation, in which
$\kappa_r(Z) = \kappa(\underbrace{Z,Z,\ldots,Z}_{r\text{ copies}})$.

If $S$ is a set of even size, a \textit{pairing} of $S$ is a
partition of $S$ into $\frac12\,\abs S$ disjoint pairs.  We
will write the pairs like $(i_1,i_2)$, but note that each pair
is an unordered set.
Recall the following result of Isserlis~\cite{Isserlis}.
\begin{thm}\label{isserlis}
  Let $A$ be a positive-definite real symmetric matrix of
  order~$n$ and
  let  $\X=(X_1,\ldots,X_n)$ be a random vector with the
  Gaussian  density $\pi^{-n/2}\abs{A}^{1/2} e^{-\xvec\Trans6 A\xvec}$.
  Let $\varSigma=(\sigma_{jk})=(2A)^{-1}$ be the corresponding
  covariance matrix.
  Consider a product $Z=X_{j_1}X_{j_2}\cdots X_{j_k}$, where the
  subscripts do not need to be distinct.  If $k$ is odd, then
  $\E Z=0$.  If $k$ is even, then
  \[
     \E Z = \sum_{\{(i_1,i_2),(i_3,i_4),\ldots,(i_{k{-}1},i_k)\}}
        \sigma_{j_{i_1}j_{i_2}}\cdots\sigma_{j_{i_{k{-}1}}j_{i_k}},
  \]
  where the sum is over all partitions of\/ $\{1,\ldots,k\}$
  into disjoint pairs.
  The number of terms in the sum is $(k-1)(k-3)\cdots3\cdot 1$.
\end{thm}

In quantum field theory, pairings are known as (complete)
Feynman graphs and Theorem~\ref{isserlis} is known as
Wick's formula after a later discoverer.  The following
generalization is
well-known in that field, but for completeness we will give
a simple proof here.

\begin{thm}\label{cums}
Assume the conditions of Theorem~\ref{isserlis} and
let $\{K_1,\ldots,K_r\}$ be a partition of $\{1,\ldots,k\}$.
Define
\[
 \kappa =
  \kappa\biggl(\,\prod_{i\in K_1} X_{j_i},\ldots,\prod_{i\in K_r} X_{j_i}\biggr),
\]
where $1\leq j_1,\ldots,j_k\leq n$ are not necessarily distinct.
If $k$ is odd then $\kappa=0$, so assume that $k$ is even.

If $\pi$ is a pairing of $\{1,\ldots,k\}$, define the graph $G_\pi$
as follows: $V(G_\pi)=\{1,\ldots,r                                                                                                                                                                                                                                                                                                                                                                                                                                                                                                                                                                                                                                                                                                                                                                                                                                                                                                                                                                                                                                                                                                                                                                                                                                                                                                                                                                                                                                                                                                                                                                                                                                                               \}$, and for $\ell\ne m$, 
$\{\ell,m\}\in E(G_\pi)$ iff $\pi$ has a pair $(i_1,i_2)$ such that
$i_1\in K_\ell$ and $i_2\in K_m$.
Let $\varPi$ be the set of all pairings $\pi$ such that $G_\pi$ is connected.
Then
 \[
     \kappa
      = \sum_{\{(i_1,i_2),(i_3,i_4),\ldots,(i_{k{-}1},i_k)\}\in\varPi}
        \sigma_{j_{i_1}j_{i_2}}\cdots\sigma_{j_{i_{k{-}1}}j_{i_k}}.
  \]
\end{thm}
\begin{proof}
  If $k$ is odd then $\kappa$ is an odd function of its arguments,
  so by symmetry $\kappa=0$. Now assume that $k$ is even.

 For $B\subseteq\{1,\ldots,r\}$, define $K(B)=\bigcup_{j\in B} K_j$.
 By the standard formula for cumulants in terms of moments~\cite{Speed},
 \begin{equation}\label{cumsum}
    \kappa
    =
   \sum_{B_1\cup\cdots\cup B_t=\{1,\ldots,r\}}
      (-1)^{t+1} (t-1)! \, \prod_{u=1}^t \,
        \E\Bigl( \prod_{i\in K(B_u)} X_{j_i} \Bigr),
 \end{equation}
 where the sum is over all partitions of $\{1,\ldots,r\}$.
 By Theorem~\ref{isserlis}, we have
 \[
    \prod_{u=1}^t \,\E\Bigl(\, \!\prod_{i\in K(B_u)} X_{j_i} \Bigr)
    =
    \sum_{\{(i_1,i_2),(i_3,i_4),\ldots,(i_{k{-}1},i_k)\}\in\varPi_B}
        \sigma_{j_{i_1}j_{i_2}}\cdots\sigma_{j_{i_{k{-}1}}j_{i_k}},
 \]
 where $\varPi_B$ is the set of pairings such that no pair
  spans two of the sets $K(B_u)$.
 
 Now consider any particular product
 $\sigma_{j_{i_1}j_{i_2}}\cdots\sigma_{j_{i_{k{-}1}}j_{i_k}}$,
 corresponding to pairing~$\pi$.
 The total weight with which this occurs in~\eqref{cumsum} is
 \begin{equation}\label{stirsum}
     \sum_{k=1}^{m} (-1)^{k+1} (k-1)! \,\stirlingii mk,
 \end{equation}
 where $m$ is the number of components of $G_\pi$
 and $\stirlingii mk$ denotes the Stirling number of the
 second kind (number of partitions of an $m$-set into
 $k$ parts).
 A standard identity for the Stirling numbers is that~\eqref{stirsum}
 equals~1 for $m=1$ and~0 for $m\geq 2$, which completes the proof.
\end{proof}

We will apply this theory to a family of random variables of a 
particular structure.  For $1\leq j\leq k\leq n$ let $Y_{jk}$ be a
real random variable, such that $\{ Y_{jk} \}$ have a joint
Gaussian  distribution with mean 0 and
\begin{equation}\label{covarbounds}
  \Cov(Y_{jk},Y_{uv})=
  \begin{cases} O(1), 
    & \text{~~if $\{j,k\}\cap\{u,v\}\ne\emptyset$}; \\
    O(n^{-1}), 
        & \text{~~if $\{j,k\}\cap\{u,v\}=\emptyset$}.
  \end{cases}
\end{equation}
Note that there are diagonal variables $Y_{jj}$.
In applying the $O(\,)$ notation, we will assume that the
cumulant order, and the powers of $Y_{jk}$ variables, are
bounded.

For any graph $K$ with weighted edges, let $n_1,n_2$ be the
number of components of total odd weight or even weight,
respectively.   Define
\[
     y(K) := \begin{cases}
             \,0, & \text{~~if~~} n_1=0, n_2=1 \\
           \,n_2, & \text{~~if~~} n_1=0, n_2>1 \\
           \,n_1+n_2-1, & \text{~~if~~} n_1>0.
          \end{cases} 
\]

For a sequence $Y_{j_1k_1}^{\ell_1},\ldots,Y_{j_rk_r}^{\ell_r}$, define
the edge-weighted graph $K=K(Y_{j_1k_1}^{\ell_1},\ldots,Y_{j_rk_r}^{\ell_r})$
as follows: $V(K)=\{j_1,k_1,\ldots,j_r,k_r\}$ (a set with duplicates removed),
and edges $\{j_1k_1,\allowbreak
\ldots,j_rk_r\}$ (a multiset with duplicates retained)
where edge $j_sk_s$ has weight $\ell_s$ for $1\leq s\leq r$.
In Figure~\ref{fig1}, we show $K=K(Y_{11}^1,Y_{55}^2,Y_{12}^1,Y_{12}^2,Y_{34}^3,Y_{45}^4)$.
Note that $K(Y_{j_1k_1}^{\ell_1},\ldots,Y_{j_rk_r}^{\ell_r})$ does not\
encode the order of its arguments, but that will not matter since
cumulants are symmetric functions.

\begin{figure}[ht]
\centering
\includegraphics[scale=0.8]{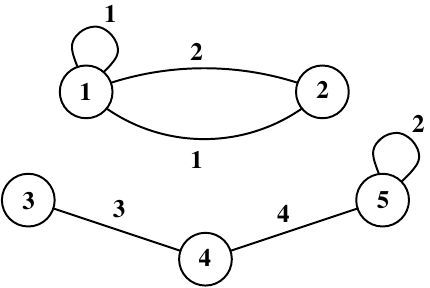}
\caption{The graph $K=K(Y_{11}^1,Y_{55}^2,Y_{12}^1,Y_{12}^2,Y_{34}^3,Y_{45}^4)$,
  which has $y(K)=1$.}
\label{fig1} 
\end{figure} 

\begin{lemma}\label{exactcum}
 Consider the cumulant 
 $\kappa=\kappa\( Y_{j_1k_1}^{\ell_1},\ldots,Y_{j_rk_r}^{\ell_r}\)$.
 Define $\ell:=\sum_{i=1}^r \ell_i$.  If~$\ell$ is odd then
 $\kappa=0$, so assume that $\ell$ is even.
 Define $K=K\(Y_{j_1k_1}^{\ell_1},\ldots,Y_{j_rk_r}^{\ell_r}\)$.
 Then
 \[
     \kappa = O( n^{-y(K)} ).
 \]
\end{lemma}
\begin{proof}
 We can apply Theorem~\ref{cums}. By~\eqref{covarbounds}.
 Each pairing $\pi$ for which $G_\pi$ is connected contains
 at least $y(K)$ pairs that correspond to edges in different
 components of~$K$ and which therefore give contribution
 $O(n^{-1})$ by~\eqref{covarbounds}.
\end{proof}

Recall the multilinearity property of cumulants~\cite{Speed}:
If $Z=\sum_j a_j Z_j$, then
\begin{equation}\label{multilinear}
      \kappa_r(Z) = \sum_{j_1,j_2,\ldots,j_r} 
         a_{j_1}\cdots a_{j_r}\,\kappa(Z_{j_1},Z_{j_2},\ldots,Z_{j_r}),
\end{equation}
with no assumptions other than that the cumulants exist.

\begin{thm}\label{maxcumcor}
Consider the sum
\[
    Y :=  \sum_{j=1}^n c^{(1)}_{jj} Y_{jj} 
    + \!\!\sum_{1\leq j\leq k\leq n} c_{jk}^{(2)} \,Y_{jk}^2
    + \sum_{\ell=3}^{\hat\ell} \, \sum_{1\leq j< k\leq n} c_{jk}^{(\ell)} \,Y_{jk}^\ell,
\]
where the coefficients may be complex. 
Assume that $r,\hat\ell=O(1)$ and
\begin{align*}
   \abs{c^{(1)}_{jj}} &\leq nC_1, \text{~~for $1\leq j\leq n$}; \\
    \abs{c^{(2)}_{jk}} &\leq \begin{cases}
               nC_2, & \text{~~for $1\leq j=k\leq n$}, \\
                 C_2 &  \text{~~for $1\leq j< k\leq n$},
                 \end{cases} \\
  \abs{c_{jk}^{(\ell)}} & \leq C_\ell,  \text{~~for $1\leq j< k\leq n$, $\ell\geq 3$}.
\end{align*}
Then
\[
   \kappa_r(Y)=O(n^{r+1})\max_{\ell_1,\ldots,\ell_r} \(C_{\ell_1}\cdots C_{\ell_r}\),
\]
where the maximum is over $\ell_1,\ldots,\ell_r$ with
$1\leq \ell_1,\ldots,\ell_r\leq \hat\ell$ and
$\ell_1+\cdots+\ell_r$ even.
\end{thm}

\begin{proof}
 Consider the graph
 $K=K\(Y_{j_1k_1}^{\ell_1},\ldots,Y_{j_rk_r}^{\ell_r}\)$.
 If the number of loops is $t$ and the number of components is~$c$,
 then it has at most $r+c-t$ vertices.
 Therefore, the number of choices of $j_1,k_1,\ldots,j_r,k_r$
 is less than $n^{r+c-t}$. Also, $y(K)\geq c-1$.
 Thus Lemma~\ref{exactcum} and~\eqref{multilinear} imply that the
 contribution to $\kappa_r(Y)$ of terms of the form
\[
    \kappa\(c_{j_1k_1}^{(\ell_1)}Y_{j_1k_1}^{\ell_1},
    \ldots,c_{j_rk_r}^{(\ell_r)} Y_{j_rk_r}^{\ell_r}\),
\]
when summed over all $j_1,k_1,\ldots,j_r,k_r$, is
zero if $\ell_1+\cdots+\ell_r$ is odd and otherwise is
\[
   O\( n^{r+1} C_{\ell_1}\cdots C_{\ell_r} \).
\]
Since by assumption the number of graphs is bounded, the 
lemma follows.
 \end{proof}

\begin{thm}\label{cumdiff}
 Define the random variable $Y$ as in Theorem~\ref{maxcumcor}.
 Let $\medtilde Y_{jk}$ be another set of jointly Gaussian  random variables
 satisfying~\eqref{covarbounds} and define
 \[
   \medtilde Y :=  \sum_{j=1}^n \tilde c^{(1)}_{jj} \medtilde Y_{jj} 
       + \!\!\sum_{1\leq j\leq k\leq n} \tilde c_{jk}^{(2)} \,\medtilde Y_{jk}^2
       + \sum_{\ell=3}^{\hat\ell} \, \sum_{1\leq j< k\leq n} \tilde c_{jk}^{(\ell)} \,\medtilde Y_{jk}^\ell,
\]
where the coefficients satisfy the same bounds as those of~$Y$.

Let $\sigma_{j_1k_1,j_2k_2}$ denote $\Cov(Y_{j_1k_1},Y_{j_2k_2})$ and
$\tilde\sigma_{j_1k_1,j_2k_2}$ denote $\Cov(\medtilde Y_{j_1k_1},\medtilde Y_{j_2k_2})$.
Suppose there is a value $\mu$ such that
for $1\leq j_1,j_2,k_1,k_2\leq n$,
\[
\Dfrac{1}{n!} \sum_{g\in S_n} 
   \Abs{\sigma_{g(j_1)g(k_1),g(j_2)g(k_2)}-\tilde\sigma_{g(j_1)g(k_1),g(j_2)g(k_2)}}
   \leq \begin{cases} \mu, & \text{~if $\{j_1,k_1\}\cap\{j_2,k_2\}\ne\emptyset$}; \\
                                \mu/n & \text{~otherwise}.
             \end{cases}   
\]
Also suppose that there are $\mu_1,\mu'_1,\mu_2,\ldots,\mu_m$ such that
\begin{align*}
   \sum_{j=1}^n \,\abs{c^{(1)}_{jj} - \tilde c^{(1)}_{jj}} &\leq \mu_1 n^2 C_1, \\
   \abs{c^{(1)}_{jj} - \tilde c^{(1)}_{jj}}\,
    \sum_{j=1}^n \,\(\abs{c^{(1)}_{jj}}+\abs{\tilde c^{(1)}_{jj}}\) &\leq \mu'_1 n^3 C_1^2,
    \quad\text{for $1\leq j\leq n$}; \\
   \sum_{j=1}^n \,\abs{c^{(2)}_{jj} - \tilde c^{(2)}_{jj}} &\leq \mu_2 n^2 C_2, 
       \quad\text{for $1\leq j\leq n$};\\
  \sum_{1\leq j<k\leq n} \abs{c^{(\ell)}_{jk} - \tilde c^{(\ell)}_{jk}}
       &\leq \mu_\ell n^2 C_\ell, \quad\text{for $2\leq \ell\leq \hat\ell$}.
\end{align*}
Then for $r,\hat\ell=O(1)$,
\[
  \abs{\kappa_r(Y)-\kappa_r(\medtilde Y)} =
     O(n^{r+1}) \max_{\ell_1,\ldots,\ell_r} \biggl(C_{\ell_1}\cdots C_{\ell_r}
        \Bigl(\mu + \!\!\sum_{j\in \{\ell_1,\ldots,\ell_r\}} \mu_j\Bigr)\biggr),
\]
 where the maximum is over $\ell_1,\ldots,\ell_r$ with
$1\leq \ell_1,\ldots,\ell_r\leq \hat\ell$ and $\ell_1+\cdots+\ell_r$ even.
In the case that $\ell_j=1$ for two or more of $\ell_1,\ldots,\ell_r$, the
term $\mu_1C_{\ell_1}\cdots C_{\ell_r}$ can optionally be replaced by
$\mu'_1 C_{\ell_1}\cdots C_{\ell_r}$.
\end{thm}
\begin{proof}
It will be convenient to consider an intermediate function with the
same coefficients as $Y$ but the same variables as $\medtilde Y$.
Define
 \[
   \hat Y :=  \sum_{j=1}^n c^{(1)}_{jj} \medtilde Y_{jj} 
       + \!\!\sum_{1\leq j\leq k\leq n} c_{jk}^{(2)} \,\medtilde Y_{jk}^2
       + \sum_{\ell=3}^{\hat\ell} \, \sum_{1\leq j< k\leq n} c_{jk}^{(\ell)} \,\medtilde Y_{jk}^\ell,
\]
Let $\hat\sigma_{j_1k_1,j_2k_2}$ denote $\Cov(\hat Y_{j_1k_1},\hat Y_{j_2k_2})$.
Recall from Theorem~\ref{cums} that $\kappa_r(Y)$ is the sum of products
of covariances of $\{Y_{jk}\}$, with terms like
 $\prod_{h=1}^{\hat h} \sigma_{j_hk_h,j'_hk'_h}$.
 To compare this term with the corresponding term in the expansion of
 $\kappa_r(\hat Y)$, we can use a telescoping sum:
 \begin{align}
      \prod_{h=1}^H \sigma_{j_hk_h,j'_hk'_h} &- 
          \prod_{h=1}^H \hat\sigma_{j_hk_h,j'_hk'_h} \notag\\
      &=
      \sum_{q=1}^{\hat h}\, \Bigl(
        (\sigma_{j_qk_q,j'_qk'_q}-\hat\sigma_{j_qk_q,j'_qk'_q})
        \prod_{h=1}^{q-1} \sigma_{j_hk_h,j'_hk'_h}
       \prod_{h=q+1}^{\hat h}\hat\sigma_{j_hk_h,j'_hk'_h} \Bigr).\label{e:cumdiff}
 \end{align}
 In computing the bounds of Theorem~\ref{maxcumcor}, we used the
 sum of the absolute values of the terms, ignoring the signs.  Also, we
 used uniform bounds on the coefficients for each degree.
  
 Now consider a parallel bounding calculation for the difference of
 cumulants, ignoring the sign of terms as before.
 For any term like those on the right side of~\eqref{e:cumdiff}
  for particular~$q$, applying a permutation to each of the
  indices involved produces a term which also appears.
  Since the bound we use for
  $\prod_{h=1}^{q-1} \sigma_{j_hk_h,j'_hk'_h}
       \prod_{h=q+1}^{\hat h} \hat\sigma_{j_hk_h,j'_hk'_h}$ is
   unchanged by the permutation, the average effect of the
   factor 
 $\sigma_{j_qk_q,j'_qk'_q}-\hat\sigma_{j_qk_q,j'_qk'_q}$
  is at most $\mu$ times the bound on $\sigma_{j_qk_q,j'_qk'_q}$.
  This implies
  \[
  \abs{\kappa_r(Y)-\kappa_r(\hat Y)} =
     O(n^{r+1}) \,\mu \max_{\ell_1,\ldots,\ell_r} (C_{\ell_1}\cdots C_{\ell_r}).
  \]
  
  Next we compare $\kappa_r(\hat Y)$ to $\kappa_r(\medtilde Y)$.
  In this case the variables are the same but the coefficients vary.
  Consider a coefficient $c^{(\ell_1)}_{j_1k_1}\cdots c^{(\ell_r)}_{j_rk_r}$.
  We can express $c^{(\ell_1)}_{j_1k_1}\cdots c^{(\ell_r)}_{j_rk_r}
  - \tilde c^{(\ell_1)}_{j_1k_1}\cdots \tilde c^{(\ell_r)}_{j_rk_r}$ as
  a telescoping sum similarly to~\eqref{e:cumdiff}.
  Averaging again over permutations of the indices, the net effect of
  a factor like $c^{(\ell)}_{jk}-\tilde c^{(\ell)}_{jk}$ is to replace the bound
  on $c^{(\ell)}_{jk}$ by $\mu_\ell$ times that bound.
  This implies
  \[
  \abs{\kappa_r(\hat Y)-\kappa_r(\medtilde Y)} =
     O(n^{r+1}) \,\max_{\ell_1,\ldots,\ell_r}
     \biggl(C_{\ell_1}\cdots C_{\ell_r}\!\!\!\sum_{j\in \{\ell_1,\ldots,\ell_r\}} \mu_j\biggr).
  \]
  In the case that $\ell_1=\ell_2=1$, we have expressions like
  $(c_{j_1j_1}^{(1)}-\tilde c_{j_1j_1}^{(1)})\,\hat c_{j_2j_2}^{(1)}
    \hat c^{(\ell_3)}_{j_3k_3}\cdots \hat c^{(\ell_r)}_{j_rk_r}$, where
    each $c^{(\ell_r)}_{j_rk_r}$ is either $\tilde c^{(\ell_r)}_{j_rk_r}$
    or $\hat c^{(\ell_r)}_{j_rk_r}$.
  Bounding  $c_{j_1j_1}^{(1)}-\tilde c_{j_1j_1}^{(1)}$ and
  $\hat c^{(\ell_3)}_{j_3k_3}\cdots \hat c^{(\ell_r)}_{j_rk_r}$ by their
  uniform bounds and averaging $\hat c_{j_2j_2}^{(1)}$ gives the
  final version.
  This completes the proof.
\end{proof}

\subsection{Estimation via truncated
                 cumulant expansion}\label{s:daughter}

The arbitrary precision in Theorem~\ref{mainthm} is obtained by
use of the following theorem of the first author~\cite{daughter}.

Consider a product domain $\S =S_1\times\cdots\times S_n$
and a bounded measurable function $g:\S \to \Complexes$.
For vectors $\xvec=(x_1,\ldots,x_n),\yvec=(y_1,\ldots,y_n)\in\S$
and $V\subseteq[n]$, define $\xvec\triangleright_V\yvec = (u_1,\ldots,u_n)$
where, $u_j=x_j$ if $j\notin V$ and $u_j=y_j$ if $j\in V$.
Define the constants
\[
   \Delta_V(f) = \sup_{\xvec,\yvec\in\S}\;
   \biggl| \sum_{W\subseteq V} (-1)^{\abs{W}} f(\xvec\triangleright_V\yvec)
   \biggr|.
\]
Further define, for integer $t>0$, the quantity
\[
   \widehat \Delta_t(g) = \max_{j \in [n]}
   \,\sum_{V\in \binom{[n]}{t} \st j \in V } \!\Delta_V(g).
\]

\begin{thm}[\cite{daughter}] \label{daughterthm}
	Let $\X=(X_1,\ldots,X_n)$ be a random vector with independent components
 taking values in $\S =S_1\times\cdots\times S_n$ and let
  $g:\S \to \Complexes$ be measurable. 
Suppose that for some $s\in[n]$ and $0<\alpha\leq\frac{1}{100}$,
and all $u\in[s]$,
\begin{equation}\label{alphaneed}
      \sum_{t=u}^n \binom{t-1}{u-1} e^{3\alpha t/2} 2^u \widehat \Delta_t(g)
      \leq\alpha.
\end{equation}
Then
\[
    \E e^{g(\X)} = (1+K)^n \exp\biggl(\sum_{r=1}^s \frac{\kappa_r (g(\X))}{r!} \biggr),
\]
where $K=K(g,s,\X)\in\Complexes$ satisfies $\abs{K(g,s,\X)} \leq e^{(100\alpha)^{s+1}}-1$.
\end{thm}

We will work with functions of random variables with linear dependencies,
in which case the following lemma adapted from~\cite{eulerian} will assist in
applying Theorem~\ref{daughterthm}.
 
 \begin{lemma}[\cite{eulerian}]\label{Deltabound}
 Suppose $g : \S\to\Complexes$ is an 
 infinitely smooth function, where $\S=I_1\times\cdots\times I_n$
 and each $I_j$ is a real interval of length at most~$z$.
 Let $T$ be a real $n\times n$ invertible matrix.
  Let  $f : T(\S) \to \Reals$ be defined by $f(\xvec):=g(T^{-1}\xvec)$. 
  Then,  for all $t\in [n]$,
  \[
      \widehat \Delta_t(g)
      \leq
      \frac{\infnorm{T}^{t-1}\, \onenorm{T}}{(t-1)!}  z^t\,
  \max_{u^{}_1\in[n]} \biggl(\,
        \sum_{u^{}_2,\ldots,u^{}_t\in [n]}\,
          \sup_{\xvec\in T(\S)}\;
        \biggl| \frac {\partial^t f(\xvec)}{\partial x_{u^{}_1}\cdots\partial x_{u^{}_t}}\biggr|
       \, \biggr) .
\]
\end{lemma}

\nicebreak
\section{Integrating in a cuboid}\label{s:inbox}

In this section, we will evaluate an integral that
will later be applied to several similar functions.
Define
\begin{align}
   \medtilde F(\thetavec) &:= \exp\biggl(-\sum_{jk\in G} c_{jk} (\theta_j+\theta_k)^2
   + \medtilde R(\thetavec)\biggr),\quad\text{where~ } \label{Fdef}\\
    \medtilde R(\thetavec)  &:=   \sum_{j\in[n]}\alpha_j\theta_j
      + \sum_{1\leq j\leq k\leq n} \gamma_{jk}(\theta_j+\theta_k)^2
      + \sum_{\ell\geq 3}\sum_{jk\in G} c^{(\ell)}_{jk} (\theta_j+\theta_k)^\ell. \notag
\end{align}
In this function, $\{c_{jk}\}$ are real, but $\{\alpha_j\}$, $\{\gamma_{jk}\}$ and
$\{c^{(\ell)}_{jk}\}$ may be complex.
We can assume without loss of generality that $\{\gamma_{jk}\}$ and
$\{c^{(\ell)}_{jk}\}$ are symmetric.
Define 
\begin{equation}\label{rhodef}
   \rho:=\varLambda^{-1/2}n^{-1/2+\eps}
\end{equation}
 for some sufficiently small $\eps>0$ and, for any $t>0$, 
\[
   U_n(t):=\{\xvec\in\Reals^n\st \infnorm\xvec \leq t\}.
\]
Let $\calC\subseteq\Reals^n$ be a cuboid of the form $I_1\times\cdots\times I_n$,
where $I_1,\ldots,I_n$ are intervals.

We make the following assumptions, uniformly over $j,k,\ell$.
\begin{assume}\label{intassumptions}
  \item[(a)] The graph $G$ satisfies Assumptions~\ref{mainassumptions}(a,c,d).
  \item[(b)] $c_{jk}^{(\ell)}=\Theta(\varLambda)$ for $jk\in G$.
  \item[(c)]
    $\infnorm\alphavec=O(\varLambda^{1/2}n^{1/2-\sigma/2})$,
    where $\alphavec=(\alpha_1,\ldots,\alpha_n)\trans $.
  \item[(d)]  $\abs{\gamma_{jj}}\leq \gammamax\, n$ for $1\leq j\leq n$ and
    $\abs{\gamma_{jk}}\leq \gammamax$ for $j\ne k$,
    where $\gammamax=O(\varLambda n^{-\sigma/2})$.
  \item[(e)] $c^{(\ell)}_{jk}=O(\varLambda)$ uniformly for $jk\in G,\ell\geq 3$.
   \item[(f)] The cuboid $\calC$ satisfies
    $U_n(t_1\rho)\subseteq \calC\subseteq U_n(t_2\rho)$
  for some constants $0<t_1\leq t_2$.
\end{assume}

Our aim is to use Theorem~\ref{daughterthm} to
estimate the integral
\[
      \int_{T\calC} \medtilde F(\thetavec)\,d\thetavec,
\]
where $T$ is a particular matrix such that $T\Trans6 AT=I$,
defined below in Lemma~\ref{quadform}(c).
This will be useful both for the dominant  regions of
integral~\eqref{Fdef} and for bounding the contribution 
outside the dominant  regions.
We can't work directly with a cuboid aligned with the $\thetavec$
axes because Theorem~\ref{daughterthm} requires independent components.

Our approach is to use Theorem~\ref{daughterthm} to write the integral
to our desired precision $O(n^{-p})$ in terms of a bounded number of
cumulants of $\medtilde{R}(\thetavec)$.
Then we truncate $\medtilde{R}(\thetavec)$ to a bounded number of
terms using elementary bounds.
Having reduced the problem to a bounded number of cumulants 
of a polynomial of bounded degree, we use Theorem~\ref{maxcumcor}
to reduce the bounds even further.

\begin{lemma}\label{quadform}
 Assume Assumptions~\ref{intassumptions}.
 Define the symmetric matrix $\medtilde A$ by
 \[
   \sum_{jk\in G} c_{jk}(\theta_j+\theta_k)^2 = 
 \dfrac12 \varLambda n \thetavec\Trans4 \medtilde A\thetavec.
 \]
 Let $D$ be the diagonal matrix with the same diagonal as~$\medtilde A$.
\begin{itemize}[noitemsep,topsep=0pt]
     \item[(a)] $\medtilde A$ is positive definite, with all eigenvalues $\Theta(1)$.
     Its diagonal elements are $\Theta(1)$ and its off-diagonal elements
     are $O(n^{-1})$.
      \item[(b)]  $\infnorm{\medtilde A^{-1}}=O(1)$ and
     $\maxnorm{\medtilde A^{-1}-D^{-1}}= O(n^{-1})$.
   \item[(c)] There is a matrix $T$ such that
     $T\Trans4  \medtilde A T=I$.  Moreover, $T$ can be chosen such that
    $\infnorm{T}, \infnorm{T^{-1}} =O(1)$,
    $\maxnorm{T-D^{-1/2}} =O(n^{-1})$ and $\maxnorm{T^{-1}-D^{1/2}}=O(n^{-1})$.
    \item[(d)] Let $\Y=(Y_1,\ldots,Y_n)$ be a Gaussian  random variable with density
      \[
          (2\pi)^{-n/2} (\varLambda n)^{n/2}
      \abs{\tilde A}^{1/2}e^{-\frac12\varLambda n\sum_{j=1}^n \yvec\Trans5\tilde A\yvec}.
      \]
      For $1\leq j\leq k\leq n$, define $Y_{jk} = (\varLambda n)^{1/2} (Y_j+Y_k)$.
     Then $\{ Y_{jk} \}$ satisfy~\eqref{covarbounds}.
    \end{itemize}
\end{lemma}
\begin{proof}
 The lower bound on the eigenvalues comes from
 Assumption~\ref{mainassumptions}(d) and Assumption~\ref{intassumptions}(b).
 The upper bound comes from the maximum row sum of~$\medtilde A$.
 The bounds on the entries of~$\medtilde A$ come from Lemma~\ref{lambdarange}.
Parts (b) and (c) follow from Theorem~\ref{matrixthm1}.
The final claim (d) is a consequence of parts (a) and~(b).
\end{proof}

Adopt the matrix $T$ from Lemma~\ref{quadform} and define
$\phivec$ by $\thetavec=T\phivec$.
This transformation diagonalizes the quadratic
form: $\thetavec\Trans4 \medtilde A\thetavec= \phivec\Trans3 \phivec$.

Let $\X$ be the Gaussian  random variable whose
 density is
 $(2\pi)^{-n/2} (\varLambda n)^{n/2}e^{-\frac12\varLambda n\sum_{j=1}^n x_j^2}$.
Define $\hat\X$ to be the normalized truncation of $\X$ to the
cuboid~$\calC$.
We can now apply Theorem~\ref{daughterthm}.

\begin{lemma}\label{applydaughter}
Adopt Assumptions~\ref{intassumptions} and the matrix $T$ given
by Lemma~\ref{quadform}(c).  Assume $\eps<\frac1{16}\sigma$ and define
  \begin{equation}\label{sweak}
        r_1 := \biggl\lfloor \frac{4(p+1)}{\sigma} \biggr\rfloor.
  \end{equation}
  Then
  \[
     \E e^{\medtilde R(T\hat\X)} = (1+O(n^{-p}))\exp\biggl(
       \sum_{r=1}^{r_1} \dfrac{1}{r!} \kappa_r(\medtilde R(T\hat\X))\biggr).
  \]
\end{lemma}
\begin{proof}
  We will use Lemma~\ref{Deltabound} to bound the quantities
  $\widehat \Delta_t(\medtilde R)$ needed by Theorem~\ref{daughterthm}.
  Note that the function $\widehat \Delta_t$ is subadditive.
  
  Let $b=\max\{\infnorm T,\onenorm T\}$.
  Then $T(\calC)\subseteq U_n(bt_2\rho)$.
  For $u,v\in[n]$, let $b_u,b_{uv}\in\Complexes$, where the array
  $(b_{uv})$ is symmetric with zero diagonal.
  Consider $r,\ell\geq 1$.  Then we have
  \begin{align}
      \max_{u_1\in[n]} & \biggl(\,\sum_{u_2,\ldots,u_t\in[n]}
      \sup_{\thetavec\in T(\calC)} \biggl|
      \frac{\partial^t}{\partial \theta_{u_1}\cdots\partial\theta_{u_t}}
      \sum_{j=1}^n b_j\theta_j^\ell \biggr| \,\biggr) \notag \\
   &= \max_{u_1\in[n]} \sup_{\thetavec\in T(\calC)}
               \Abs{ b_{u_1} (\ell)_t \theta_{u_1}^{\ell-t} }
          \leq (\ell)_t (b t_2 \rho)^{\ell-t}  \max_{u\in[n]}\, \abs{b_u}. \label{diff1}
 \end{align}
  Similarly,
  \begin{align}
      \max_{u_1\in[n]} & \biggl(\,\sum_{u_2,\ldots,u_t\in[n]}
      \sup_{\thetavec\in T(\calC)} \biggl|
      \frac{\partial^t}{\partial \theta_{u_1}\cdots\partial\theta_{u_t}}
      \sum_{1\leq j<k\leq n} b_{jk}(\theta_j+\theta_k)^\ell \biggr|\, \biggr) \notag\\
   &= \max_{u_1\in[n]}  \biggl(\,\sum_{u_2,\ldots,u_t\in[n]}
      \sup_{\thetavec\in T(\calC)} \biggl|
      \frac{\partial^t}{\partial \theta_{u_1}\cdots\partial\theta_{u_t}}
      \sum_{k=1}^n b_{u_1k}(\theta_{u_1}+\theta_k)^\ell \biggr|\,  \biggr) \notag\\
   &\leq  \max_{u_1\in[n]} \biggl(\, \sum_{k=1}^n \,\abs{b_{u_1k}}
      \sum_{u_2,\ldots,u_t\in[n]}
      \sup_{\thetavec\in T(\calC)} \biggl|
      \frac{\partial^t}{\partial \theta_{u_1}\cdots\partial\theta_{u_t}}
     (\theta_{u_1}+\theta_k)^\ell \biggr|\, \biggr) \notag \\
   &\leq \infnorm{(b_{uv})} \,2^{\ell-1} (\ell)_t (bt_2\rho)^{\ell-t},\label{diff2}
  \end{align}
  where the final step comes from noting that the sum over $u_2,\ldots,u_t$
  only gives nonzero values when $u_2,\ldots,u_t\in\{u_1,k\}$.
  Note that the falling factorial $(\ell)_t$ is 0 if $t>\ell$.

  Now applying~\eqref{diff1} and~\eqref{diff2}, Lemma~\ref{Deltabound},
  the bound $\varLambda^{-1}n^{-1}=O(n^{-\sigma})$ from
  Assumptions~\ref{mainassumptions}(c),
  the identity $\sum_{\ell\geq t} (\ell)_t z^\ell = t!\,z^t/(1-z)^{t+1}$,
  and the bound $\sum_{\ell\geq 3} (\ell)_t z^\ell <28z^3$ for $t=1,2,z\in(0,\frac12)$,
  we obtain
  \begin{align*}
      \widehat\Delta_1(\medtilde R) & = O(n^{-\sigma/2+3\eps}) \\
       \widehat\Delta_2(\medtilde R) & = O(n^{-\sigma/2+3\eps}) \\
        \widehat\Delta_t(\medtilde R) & = O(1)^t n^{-(t/2-1)\sigma+t\eps} \quad(t\geq 3),
   \end{align*}
   where $O(1)$ in the last expression is uniform over~$t$.
   Now, with the aid of the identity $\sum_{t\geq u}\binom{t-1}{u-1}z^t=z^u/(1-z)^u$,
   we can check that~\eqref{alphaneed} is satisfied with $\alpha=n^{-\sigma/2+4\eps}$,
   which, together with the assumption $\eps< \frac{1}{16}\sigma$,
   completes the proof.
\end{proof}

Next we make two simplifications of Lemma~\ref{applydaughter}.
First we show that the series defining $\medtilde R(\thetavec)$ can be 
truncated to a bounded number of terms.  Then we show that
the cumulants of $\medtilde R(T\hat\X)$
can be replaced by the more-manageable cumulants of the
corresponding non-truncated multivariate Gaussian  distribution.
For any $m\geq 3$, define
\[
   \medtilde R_m(\thetavec)
   :=  \sum_{j\in[n]}\alpha_j\theta_j +
        \!\!\sum_{1\leq j\leq k\leq n} \gamma_{jk}(\theta_j+\theta_k)^2 + 
      \sum_{\ell=3}^m
        \sum_{jk\in G} c_\ell(\ljk ) (\theta_j+\theta_k)^\ell.
\]
  
\begin{lemma}\label{truncatedsum}
Adopt Assumptions~\ref{intassumptions}. Assume $\eps<\frac1{16}\sigma$
 and define
  \[
     \ell_1 = \biggl\lfloor \frac{11(p+1)}{\sigma^2}\biggr\rfloor.
  \]
  Then, as $n\to\infty$,
  \[
       \E e^{\medtilde R(T\hat\X)} = (1+O(n^{-p}))\exp\biggl(
       \sum_{r=1}^{r_1} \dfrac{1}{r!} \kappa_r(\medtilde R_{\ell_1}(T\X))\biggr),
  \]
\end{lemma}
\begin{proof}
  From~Lemma~\ref{quadform} we have that the components
  of $T\hat\X$ have absolute value at most $O(\rho)$.
  Since $\rho\to 0$, for sufficiently large $n$ we have
  \[
      \abs{\medtilde R(T\hat\X)} =
        O(\rho n) \infnorm\alphavec
       + O(\rho^2 n^2) \gammamax
       + \varLambda n^2 \sum_{\ell\geq 3} (O(\rho))^\ell
      = O(n^{1-\sigma/2+3\eps})
      = O(n^{1-5\sigma/16}).
  \]
  Also, for any fixed $m\geq 3$,
  \[
  \abs{\medtilde R(T\hat\X)-\medtilde R_m(T\hat\X)}
    = O\(\varLambda n^2(2\rho)^{m+1}\)
    = O(n^{1+\sigma/2-\sigma m/2+(2m+1)\eps})
    = O(n^{1+9\sigma/16-3\sigma m/8}),
  \]
  using Assumptions~\ref{intassumptions}(e).  
   
  Recall that the $r$-th cumulant of a random variable can be
  written as a polynomial in central moments of total order~$r$.
  (For example $\kappa_4 = \mu_4 - 3\mu_2^2$.)
  Consequently,
  \[
       \Abs{\kappa_r(\medtilde R(T\hat\X))-\kappa_r(\medtilde R_{\ell_1}(T\hat\X))} 
      = O( n^{1+9\sigma/16-3\sigma m/8 + (r-1)(1-5\sigma/16)} ).
 \]
 For $1\leq r\leq r_1$ and $m\geq 11(p+1)/\sigma^2-1$, we have
 $1+\frac{9}{16}\sigma-\frac38\sigma m + (r-1)(1-\frac5{16}\sigma)<-p$.
  
  In order to replace $\hat \X$ by $\X$,
 note that the cumulant sum is a polynomial in a bounded
 number of moments with polynomially bounded coefficients.
 Since the moments for the truncated and
 non-truncated distributions only differ by
  $e^{-\Omega(n^{2\eps})}$ (see~\cite[Lemma 4.1]{mother} for
  a proof), the cumulant sums differ by an
  amount much smaller than the existing $O(n^{-p})$ error term.
  \end{proof}

\begin{lemma}\label{better}
Adopt Assumptions~\ref{intassumptions} and let
$\Y$ be a Gaussian  random variable with density
$(2\pi)^{-n/2} (\varLambda n)^{n/2}
\abs{\medtilde A}^{1/2}e^{-\frac12\varLambda n\sum_{j=1}^n
\yvec\Trans5\tilde A\yvec}$.
Then for $1\leq r\leq r_1$ we have the following uniform bounds
on the cumulants of $\medtilde R_{\hat\ell}(\Y)$
for $3\leq \hat\ell\leq \ell_1$.
\[
  \kappa_r(\medtilde R_{\hat\ell}(\Y)) =
    \begin{cases}
        O(n^{1-\lceil r/2\rceil\sigma}), & \text{~~if $\gammamax\leq\varLambda n^{-\sigma}$}; \\
        O(n^{1-r\sigma/2}), & \text{~~in general}.
    \end{cases}
\]
\end{lemma}
\begin{proof}
For $1\leq j\leq k\leq n$, define $Y_{jk}$ as in Lemma~\ref{quadform}(d).
By Assumptions~\ref{intassumptions}, we can apply
Theorem~\ref{maxcumcor} using
\begin{align*}
   C_1 &= O(\varLambda^{-1/2}n^{-3/2}\infnorm\alphavec)=O(n^{-1-\sigma/2}) \\
   C_2 &= O(\varLambda^{-1}n^{-1}\gammamax) = O(n^{-1-\sigma/2}) \\
   C_\ell &= O(\varLambda^{1-\ell/2}n^{-\ell/2}), \text{~~if $\ell\geq 3$}.
\end{align*}
Note that $C_\ell = O(n^{-1-\sigma/2})$ for all~$\ell$.
Using Theorem~\ref{maxcumcor}, we find the uniform bound
$O(n^{1-r\sigma/2})$ for $n^{r+1}C_{\ell_1}\cdots C_{\ell_r}$
for any $\ell_1,\ldots,\ell_r$ and $r\leq r_1$.
Since there are only $O(1)$ such terms for fixed~$r$,
the contribution to $\kappa_r(\medtilde R_{m'}(\Y))$ from
all cumulants of the form
$\kappa(Y_{j_1k_1}^{\ell_1},\ldots,Y_{j_rk_r}^{\ell_r})$
is $O(n^{1-r\sigma/2})$.

Under the stronger assumption
$\gammamax \leq\varLambda n^{-\sigma}$,
we have a better bound $C_\ell=O(n^{-1-\sigma})$ for
even~$\ell$.
Recall that $\kappa(Y_{j_1k_1}^{\ell_1},\ldots,Y_{j_rk_r}^{\ell_r})=0$
if $\sum_{j=1}^r \ell_j$ is odd, so for all $r$ we
can use the improved bound
$n^{r+1}C_{\ell_1}\cdots C_{\ell_r}=O(n^{1-\lceil r/2\rceil\sigma})$.
\end{proof}

\begin{thm}\label{genbox}
Adopt Assumptions~\ref{intassumptions}.
Then
\[
  \int_{T\calC} \medtilde F(\thetavec)\,d\thetavec = (1+O(n^{-p}))
  \frac{(2\pi)^{n/2}}{(\varLambda n)^{n/2}\abs{\medtilde A}^{1/2}}
   \exp\biggl(
       \sum_{r=1}^{r_1} \dfrac{1}{r!} \kappa_r(\medtilde R_{\ell_0}(\Y))\biggr),
\]
where $\Y$ is a Gaussian  random variable with density
$(2\pi)^{-n/2} (\varLambda n)^{n/2}
\abs{\medtilde A}^{1/2}e^{-\frac12\varLambda n\sum_{j=1}^n \yvec\Trans4\medtilde A\yvec}$
and $\ell_0=2\lceil{(p+1)/\sigma}\rceil$.
\end{thm}
\begin{proof}
The theorem with $R_{\ell_1}$ in place of $R_{\ell_0}$ follows from Lemma~\ref{truncatedsum}
and the definition of~$\Y$.
To show that $R_{\ell_0}$ is adequate, we can apply Theorem~\ref{cumdiff}
with $\mu=0$, $\mu_1=\cdots=\mu_{\ell_0}=0$ and
$\mu_{\ell_0+1}=\cdots=\mu_{\ell_1}=1$.

If $\ell_1=\max\{\ell_1,\ldots,\ell_r\}\geq \ell_0+1$, then
\[
   C_{\ell_1} = O(n^{-1-(\ell_1-2)\sigma/2})
   = \begin{cases}
       O(n^{-2-p+\sigma/2}), & \text{ if $\ell_1=\ell_0+1$}; \\
       O(n^{-2-p}), & \text{ if $\ell_1\geq \ell_0+2$}.
     \end{cases}
\]
If $r=1$, the cumulant $\kappa(Y_{j_1k_1}^{\ell_1},\ldots,Y_{j_rk_r}^{\ell_r})=0$
is zero unless $\ell_1\geq \ell_0+2$
since $\ell_0+1$ if odd, so $n^{r+1}C_{\ell_1}=O(n^{-p})$.
If $r\geq 2$, using the uniform bound $C_{\ell_j}=O(n^{-1-\sigma/2})$
for all $j\geq 2$,
we get $n^{r+1}C_{\ell_1}\cdots C_{\ell_r} =O(n^{-p})$ again.
Therefore $\kappa_r(\medtilde R_{\ell_1}(\Y))
=\kappa_r(\medtilde R_{\ell_0}(\Y))+O(n^{-p})$ for $r=O(1)$.
This completes the proof.
\end{proof}

\nicebreak
\section{Proof of Theorem~\ref{mainthm}}\label{s:outofbox}

Define $F_{\betavec}(\thetavec)$ as in~\eqref{Fdefn}, and 
$\rho$ as in~\eqref{rhodef}.
Let $T$ be the matrix such that $T\Trans5 AT=I$ guaranteed by
Lemma~\ref{quadform}.
Define small regions as follows
\[
     \calB_0:=TU_n(\rho), \quad 
     \calB_\pi := \calB_0 + (\pi,\ldots,\pi)\trans  \pmod{2\pi}.
\]
The integrals of $F_{\betavec}(\thetavec)$ in $\calB_0$ and $\calB_\pi$
are the same, since $F_{\betavec}(\thetavec)$ is invariant under the
translation $\thetavec\mapsto\thetavec+(\pi,\ldots,\pi)\trans $.

In this section, we estimate $\int_{U_n(\pi)}F_{\betavec}(\thetavec)$
in four stages. In Section~\ref{s:inside}, we use
the theory from Section~\ref{s:inbox} to integrate $F_{\betavec}(\thetavec)$
inside the parallelepiped~$\calB_0$.
In Section~\ref{s:outside1}, we
identify a portion of $U_n(\pi)$ far from $\calB_0$ and $\calB_\pi$
in which the integral of the absolute value
$\abs{F_{\betavec}(\thetavec)}$ is negligible compared to the integral
inside $\calB_0$.
The remainder of $U_n(\pi)$ is the most challenging,
as the integral of the absolute value is not negligible there.
Section~\ref{s:outside2} bounds the integral in a parallelepiped
similar to $\calB_0$ but shifted away from origin, in
terms of the integral in~$\calB_0$.
Finally, in Section~\ref{s:finale}, the parts are brought together
to prove Theorem~\ref{mainthm} with the help of Theorem~\ref{MishaMagic}.

There are two ways in which our method differs from previous
methods~\cite{BarvHart1,MWreg,ranx}. First, note that the
parallelepiped $\calB_0$ is not axis-aligned, but is instead tilted
so that it becomes an axis-aligned cuboid when the linear transform~$T$
is applied to diagonalise the quadratic form.  The reason
is that Theorem~\ref{daughterthm} requires a random vector with
independent components.

The second difference is the method of working, which is forced
on us by the fact that
$\int_{\calB-\calB_0-\calB_\pi} \, \abs{F_{\betavec}(\thetavec)}\,d\thetavec$ is not
small compared to $\int_{\calB_0} F_{\betavec}(\thetavec)\,d\thetavec$
when $\varLambda\to 0$ quickly. 

\subsection{The integral inside $\calB_0$}\label{s:inside}

The function $F_{\betavec}(\thetavec)$ has a convergent Taylor
expansion in $\calB_0$:
\begin{align*}
   F_{\betavec}(\thetavec) &= -\dfrac12\varLambda n\,\thetavec\Trans6 A\thetavec 
     + R(\thetavec),\qquad\text{where} \\
   R(\thetavec) &= i\sum_{j=1}^n\delta_j\theta_j +
        \sum_{\ell\geq 3} \sum_{jk\in G} c_\ell(\ljk)(\theta_j+\theta_k)^\ell .
\end{align*}
To apply Theorem~\ref{genbox}, we need to verify that
Assumptions~\ref{intassumptions} follow from Assumptions~\ref{mainassumptions}. 

\begin{lemma}\label{lambdarange}
Under Assumptions~\ref{mainassumptions}, we have for all $jk\in G$ that
\begin{align*}
   e^{-3C} \lambda &\leq \ljk  \leq e^{3C}\lambda, \\
   e^{-3C} (1-\lambda) &\leq 1-\ljk  \leq e^{3C}(1-\lambda), \\
   e^{-6C} \varLambda &\leq \ljk (1-\ljk ) \leq e^{6C}\varLambda.
\end{align*}
\end{lemma}
\begin{proof}
Let $\betamin =\min_{j\in [n]}\beta_j$ and $\betamax=\max_{j\in [n]}\beta_j$.
Sum the equation $d_j+\delta_j = \sum_{j:jk\in G}\ljk$ from~\eqref{approxbeta}
over $j\in[n]$ and divide by $2\abs{E(G)}$.  This gives
\[
    \frac{e^{2\betamin }}{1+e^{2\betamin }}
    \leq \lambda + (2\abs{E(G)})^{-1}\sum_{j\in[n]} \delta_j
    \leq \frac{e^{2\betamax}}{1+e^{2\betamax }}.
\]
From Assumptions~\ref{mainassumptions}(c,e) and the fact that $G$ has
$\Omega(n^2)$ edges, we have
\[
  (2\abs{E(G)})^{-1}\sum_{j\in[n]} \delta_j
     =o(\varLambda^{1/2}n^{-1/2-\sigma/2}) = o(\varLambda).
\]
Consequently, for any $jk\in G$, Assumption~\ref{mainassumptions}(b) gives
\begin{align*}
    \frac{\ljk }{\lambda}
    &\leq \frac{e^{2\betamax }(1+e^{2\betamin })}
                  {e^{2\betamin }(1+e^{2\betamax })}(1 + o(1))
    \leq \frac{e^{2\betamin +2C}(1+e^{2\betamin })}
                  {e^{2\betamin }(1+e^{2\betamin +2C})}(1 + o(1)) \\
     &\leq e^{2C}(1+o(1)) \leq e^{3C}
     \quad\text{for large enough~$n$}.
\end{align*}
The lower bound follows in the same way and the
bounds on $1-\ljk $ follow by symmetry.  The final bound follows
from the first two.
\end{proof}

\begin{lemma}\label{ckbound}
  For $\ell\geq 2$ and $jk\in G$, $\abs{c_\ell(\ljk)} \leq 8 e^{6C}\!\varLambda $.
\end{lemma}
\begin{proof}
  Since $f_{1-\ljk}(z) = iz + f_{\ljk}(-z)$, we can
  assume that $\ljk\leq\frac 12$.
  For $\abs{z}\leq 1$ we have $\abs{e^{iz}-1}\leq e-1$, so
  $f_{\ljk}(z)$ is analytic in a disk of radius greater than~1.
  Consequently,
  \[
     c_\ell(\ljk) = \frac{1}{2\pi} \int_{-\pi}^\pi 
       \frac{f_{\ljk}(e^{i\theta})}{e^{i\ell\theta}}\,d\theta,
  \]
  and so $\abs{c_\ell(\ljk)} \leq \max_{\abs z=1} \abs{f_{\ljk}(z)}$.
  Applying the Taylor series of $\ln(1+w)$ for
  $w=\ljk(e^{iz}-1)$, we have
    $\abs{c_\ell(\ljk)} \leq \sum_{t\geq 1} \dfrac{1}{t}\, \abs{w}^t
    \leq \sum_{t\geq 1} \dfrac{1}{t} (e-1)^t\ljk^t
    = -\ln\(1-(e-1)\ljk\)$.
  To complete the proof, note that
  $-\ln\(1-(e-1)\ljk\)\leq 8\ljk(1-\ljk) $ for $1\leq\ljk\leq\frac12$
  and apply Lemma~\ref{lambdarange}.
\end{proof}

\begin{thm}\label{inbox}
  Adopt Assumptions~\ref{mainassumptions} and define $r_0,\ell_0$ as in 
  Theorem~\ref{mainthm}.
  Then
\[
  \int_{\calB_0} F(\thetavec)\,d\thetavec = (1+O(n^{-p}))
  \frac{(2\pi)^{n/2}}{(\varLambda n)^{n/2}\abs{A}^{1/2}}
   \exp\biggl(\,
       \sum_{r=1}^{r_0} \dfrac{1}{r!} \kappa_r(R_{\ell_0}(\Y))\biggr),
\]
where $\Y$ is a Gaussian  random variable with density
$(2\pi)^{-n/2} (\varLambda n)^{n/2}
\abs{A}^{1/2}e^{-\frac12\varLambda n\sum_{j=1}^n \yvec\Trans6 A\yvec}$
and
\[
  R_{\ell_0}(\thetavec) := i\sum_{j=1}^n\delta_j\theta_j +
        \sum_{\ell=3}^{\ell_0} \sum_{jk\in G} c_\ell(\ljk)(\theta_j+\theta_k)^\ell .
\]
\end{thm}
\begin{proof}
  Assumptions~\ref{intassumptions} hold by
  Assumptions~\ref{mainassumptions} and Lemma~\ref{ckbound}.
  The result thus follows from Theorem~\ref{genbox} with $r_1$
  in place of~$r_0$.
  Since $\gammamax=0$, we have by Lemma~\ref{better}
  that $\kappa_r(R_{\ell_0})=O(n^{1-\lceil r/2\rceil\sigma})$
  for $1\leq r\leq r_1$.
  This implies that $\kappa_r(R_{\ell_0})=O(n^{-p})$ for
  $r_0< r\leq r_1$, completing the proof.
\end{proof}

\nicebreak
\subsection{The integral outside $\calB_0$ and $\calB_\pi$; part 1}
\label{s:outside1}

Define 
\[
   J_0 := \int_{\calB_0} F_{\betavec}(\thetavec)\,d\thetavec.
\]

In this section we will show that the integral of $\abs{F_{\betavec}}$
is small compared to $J_0$ in some regions very far from $\calB_0$
and $\calB_\pi$.
First we need an approximation of $J_0$ to compare against, 
which can be deduced from Theorem~\ref{inbox}.
\begin{align}
   J_0 &= (1+O(n^{-p}))\,(2\pi)^{n/2} (\varLambda n)^{-n/2}
    \abs{A}^{-1/2}
     \exp\biggl(\,
           \sum_{r=1}^{r_0} \dfrac{1}{r!} \kappa_r(R_{\ell_0}(G,T\X))\biggr) 
              \notag \\[-0.6ex]
    &= (2\pi)^{n/2} (\varLambda n)^{-n/2}
    \abs{A}^{-1/2} e^{O(n^{1-\sigma})}. \label{J0est-1} \\
    &= \Theta(1)^n  (\varLambda n)^{-n/2}.
       \label{J0est-2}
\end{align}

For $x\in\Reals$, define
$\abs{x}_{2\pi} = \min\{\,\abs y \st y \equiv \pm x \pmod{2\pi} \}$.
Note that this function satisfies
$\abs{x+y}_{2\pi}\leq \abs{x}_{2\pi}+\abs{y}_{2\pi}$.

\begin{lemma}\label{boring}
 For $x\in \Reals$ and $0<\lambda<1$, we have
 \begin{itemize}[noitemsep,topsep=1ex]
   \item[(a)] $\abs{1 + \lambda(e^{ix}-1)} 
                  \leq e^{-\frac15\varLambda \abs{x}_{2\pi}^2 }$.
   \item[(b)] $\abs{1 + \lambda(e^{ix}-1)} 
                  \leq e^{-\frac12\varLambda x^2 
                            +\frac1{24}\varLambda x^4}$.
 \end{itemize}
\end{lemma}
\begin{proof}
  These are routine and we omit the proofs. The second
  part appeared in~\cite{MWreg}.
\end{proof}

\begin{lemma}\label{firstcut}
Let $\eta>0$ be constant. Define
\begin{align*}
 \calB'&:=\bigl\{ \thetavec\in [-\pi,\pi]^n \St \\
  &{\qquad}
  \text{more than $\tfrac12 n^{1-\eps/2}$ components $\theta_j$
    lie in  $[\eta\rho,\pi{-}\eta\rho]\cup[-\pi{+}\eta\rho,-\eta\rho]$} \bigr\}.
\end{align*}
    Then
    \[
    \int_{\calB'} \,\abs{F_{\betavec}(\thetavec)}\,d\thetavec
     = e^{-\Omega(n^{1+\eps})} J_0.
    \]
\end{lemma}
\begin{proof}
 Without loss of generality, at least $\frac14 n^{1-\eps/2}$
 components $\theta_j$ lie in $[\eta\rho,\pi-\eta\rho]$.
 Let $U = \{ j \st \theta_j\in [\eta\rho,\pi-\eta\rho] \}$.
 By Theorem~\ref{oddpaths}, there are constants $\delta,L$ such
 that $G$ has at least $\delta n\abs{U}$ edge-disjoint
 walks of odd length at most $L$ with both ends in~$U$.
 Suppose that $\theta_j=\theta^{(0)},\theta^{(1)},\ldots,\theta^{(t)}
 =\theta_k$ is such a walk, where $t$ is odd.
 Since $\theta_j,\theta_k\in [\eta\rho,\pi-\eta\rho]$, then
 $\abs{\theta_j+\theta_k}_{2\pi} \geq 2\eta\rho$, and
 since
 $\theta_j+\theta_k=(\theta^{(0)}+\theta^{(1)}) - (\theta^{(1)}+\theta^{(2)})
  + \cdots + (\theta^{(t-1)}+\theta^{(t)})$,\
 we have $\abs{\theta^{(i-1)}+\theta^{(i)}}_{2\pi}
 \geq (2\eta\rho)/L$ for some edge $(\theta^{(i-1)},\theta^{(i)})$
 in the walk.
 Consequently, for $\thetavec\in\calB'$, Lemma~\ref{boring}(a)
 implies that
 \[
      \abs{F_{\betavec}(\thetavec)} 
       \leq e^{-\frac15 \varLambda(\frac1{16} \delta n^{2-\eps/2}-1)(2\eta\rho/L)^2}
       = e^{-\Omega(n^{1+\eps})}.
 \]
 Multiplying by the volume of $\calB'$, which is less than $(2\pi)^n$,
 and allowing $2^n$ for the choice of which components of $\thetavec$
 have $\eta\rho \leq \abs{\theta_j}\leq \pi-\eta\rho$ completes the proof
 by comparison with~\eqref{J0est-2}.
\end{proof}

If Lemma~\ref{firstcut} doesn't apply, we have at least $n-\frac12 n^{1-\eps/2}$
components of $\thetavec$ lying in
$[-\eta\rho,\eta\rho]\cup (\pi+[-\eta\rho,\eta\rho])$.
Next we will use a similar argument to show that most of these components
lie in one of those two intervals.

\begin{lemma}\label{secondcut}
Let $\eta>0$ be constant.  Define
\begin{align*}
   \calB'' := \bigl\{ \thetavec\in [-\pi,\pi]^n &\St
  \text{at most $\tfrac12 n^{1-\eps/2}$ components $\theta_j$
    of $\thetavec$ satisfy 
    $\eta\rho\leq\abs{\theta_j}\leq \pi-\eta\rho$} \\[-0.5ex]
   &{~~\quad}\text{and
    each of the intervals $[-\eta\rho,\eta\rho]$ and $(\pi+[-\eta\rho,\eta\rho])$}
      \\[-0.5ex]
    &{~~\quad}\text{contains at least $\varLambda^{-1}\log^2 n$ components} \bigr\}.
\end{align*}    
    Then
    \[
    \int_{\calB''} \,\abs{F_{\betavec}(\thetavec)}\,d\thetavec
      = e^{-\Omega(n\log^2 n)} J_0.
    \]
\end{lemma}
\begin{proof}
  Let $U_1=\{ j \st \theta_j \in [-\eta\rho,\eta\rho]\}$ and
  $U_2=\{ j \st \theta_j \in (\pi+[-\eta\rho,\eta\rho])\}$.
  By Theorem~\ref{shortpaths} and Assumption~\ref{mainassumptions}(d),
  using $m=\varLambda^{-1}\log^2 n$ and $r=\frac14\min_{j=1}^n g_j$,
  we can find $\Omega(\varLambda^{-1}n\log^2 n)$ edge-disjoint paths in $G$
  of length at most~3 from $U_1$ to $U_2$.  Moreover,  since
  $\abs{\theta_j+\theta_k}_{2\pi}\geq \pi-2\eta\rho$ for $j\in U_1,k\in U_2$,
  at least one edge $(\theta^{(1)},\theta^{(2)})$ on  each such path has
  $\abs{\theta^{(1)}+\theta^{(2)}}_{2\pi}\geq \frac13\pi-\frac23\eta\rho$.
  The proof now follows the same line as in the previous lemma.
\end{proof}

Since adding $\pi$ to each component is a symmetry, we can
now assume that at least $n-n^{1-\eps/2}$ components of
$\thetavec$ lie in $[-\eta\rho,\eta\rho]$.
Applying the same logic to $\frac12\eta$, we can in fact assume that 
at least $n-n^{1-\eps/2}$ components of
$\thetavec$ lie in $[-\frac12\eta\rho,\frac12\eta\rho]$.
Next we will establish limits on those components lying
outside the doubled interval $[-\eta\rho,\eta\rho]$.

For $\xvec=(x_1,\ldots,x_n)\in\Reals^n$ and $a\geq 0$, define
  \[
      M(\xvec,a) := \{ j \st 1\leq j\leq n \text{~and~} \abs{x_j}>a \}.
  \]

\begin{lemma}\label{thirdcut}
  Let $\eta>0,b>0$ be constant.
  Define
  \begin{align*}
     \calB''' := \bigl\{ \thetavec \in [-\pi,\pi]^n &\St
  \text{at least $n-n^{1-\eps/2}$ components $\theta_j$
  lie in $[-\tfrac12\eta\rho,\tfrac12\eta\rho]$} \\
  &{~~\quad} \text{and
  $\sum_{j\in M(\thetavec,\eta\rho)} \theta_j^2
  \geq bn^{-\sigma}\varLambda^{-1}$}\bigr\}.
 \end{align*}
 Then there is a function $b(\eta)$ such that for $b\geq b(\eta)$,
  \[
      \int_{\calB'''} \,\abs{F_{\betavec}(\thetavec)}\,d\thetavec
      = O(e^{-\Omega(n^{2\eps})}) J_0.
  \]
\end{lemma}
\begin{proof}
 Define $M_0(\thetavec):=\{ j \st \abs{\theta_j}\leq \frac12\eta\rho \}$,
 $M_1(\thetavec):=\{ j \st \frac12\eta\rho < \abs{\theta_j}\leq\eta\rho \}$
 and $M_2(\thetavec):=M(\thetavec,\eta\rho)$.
 If $M_0(\thetavec)=\emptyset$ there is nothing to do, so suppose
 otherwise.
 For $jk\in G$, we will use bounds as follows.
 \[
    \Abs{1+\ljk (e^{i(\theta_j+\theta_k)}-1} 
    \leq \begin{cases}
        \exp\Bigl( -\frac12\ljk (1{-}\ljk )(\theta_j+\theta_k)^2
                 + \frac1{24}\ljk (1{-}\ljk )(\theta_j+\theta_k)^4\Bigr), \\             
                 &{\kern -10em} \text{if $j,k\in M_0(\thetavec)\cup M_1(\thetavec)$;} \\
        \exp\(-\dfrac1{20}\varLambda \theta_j^2\),
                 &{\kern -10em} \text{if $j\in M_2(\thetavec)$, $k\in M_0(\thetavec)$;} \\
         1, &{\kern -10em} \text{otherwise}.
         \end{cases}
\]
The first bound is from Lemma~\ref{boring}(b).
The second bound comes from Lemma~\ref{boring}(a) and the fact that
$(x-1)^2\geq \frac14 x^2$ for $x\geq 2$.

For fixed $M_2(\thetavec)$, the integral of our bound on $\abs{F_{\betavec}(\thetavec)}$
separates.  Take $M_2(\thetavec)=\{1,\ldots,m\}$ as a representative case.
Note that $m\leq n^{1-\eps/2}$ if $\thetavec\in\calB'''$.
The contribution to the bound is then $I_1(m)\,I_2(m)$ where
\begin{align*}
   I_1(m) &:= \int_{\varOmega_m} e^{-\Omega(1)\varLambda n
      (\theta_1^2+\cdots+\theta_m^2)} \, d\theta_1\cdots d\theta_m, \\
\intertext{where $\varOmega_m$ is $\Reals^m$ subject to
  $\sum_{j=1}^m\theta_j^2\geq bn^{-\sigma}\varLambda^{-1}$, and}
   I_2(m) &:= \int_{U_{n-m}(\eta\rho)}
      \exp\Bigl(-\dfrac12\varLambda n \check\thetavec\Trans6  A_m\check\thetavec
      + \negthickspace
       \sum_{j,k>m, jk\in G} \dfrac1{24}\ljk (1{-}\ljk )
          (\theta_j+\theta_k)^4)\Bigr) \,d\check\thetavec,
\end{align*}
where $\check\thetavec=(\theta_{m+1},\ldots,\theta_n)$ and
\[
    \dfrac12 \varLambda n \check\thetavec\Trans6  A_m\check\thetavec
    = \sum_{j,k>m, jk\in G} \dfrac12 \ljk (1{-}\ljk )
          (\theta_j+\theta_k)^2.
\]

Let us start with $I_1(m)$.  The $n$ that appears in the definition comes
from the fact that each vertex in $M_2(\theta)$ is adjacent in $G$ to
$\Omega(n)$ vertices of $M_0(\theta)$.  In addition to the assumption
that $\sum_{j\in M_2(\thetavec)} \theta_j^2
  \geq bn^{-\sigma}\varLambda^{-1}$, we also have
$\sum_{j\in M_2(\thetavec)} \theta_j^2
  \geq m\eta^2n^{-1+2\eps}\varLambda^{-1}$ since
$\abs{\theta_1},\ldots,\abs{\theta_m} \geq \eta\rho$.
Consequently, for some constant $b'>0$,
$\Omega(1)\varLambda n(\theta_1^2+\cdots+\theta_m^2)
\geq  b'(n^{1-\sigma}+mn^{2\eps})$, where $b'$ can be
made as large as we please by increasing~$b$.
Since the region of integration has volume less than $(2\pi)^m$,
we have
\[
   I_1(m) \leq (2\pi)^m
       e^{-b' (n^{1-\sigma}+mn^{2\eps})}.
\]

Now consider $I_2(m)$.
Let $G_m$ be the subgraph of $G$
induced by $\{m+1,\ldots,n\}$.  Since deleting a vertex reduces
the algebraic bipartiteness of a graph by at most~1~ \cite[Thm.~1.1]{HePan},
we have $q(G_m)\geq \tauQ n-m$.
Now take the matrix $T_m$ such that $T_m\trans  A_mT_m=I_{n-m}$
as provided by Theorem~\ref{matrixthm1} applied to $A_m$ and
define $\check\phivec=(\check\phi_{m+1},\ldots,\check\phi_n) =T_m\check\thetavec$.
By  Theorem~\ref{matrixthm1}(c), $\infnorm{T_m^{-1}}$ is bounded
independently of~$m$ (for
$m\leq n^{1-\eps/2}$), so there is some constant $\check\eta$ such that
$T_m U_{n-m}(\eta\rho) \subseteq U_{n-m}(\check\eta\rho)$.
Note that our integrand is positive, so we can expand the region of
integration if we only want an upper bound.

Next, using the inequality $(u+v)^4\leq 8(u^4+v^4)$ twice, together with
Lemma~\ref{lambdarange}, the power-norm inequality
$\(\sum_{j=m+1}^n u_j\)^4 \leq (n-m)^3 \sum_{j=m+1}^n u_j^4$
 and
the structure of $T_m$ as given by Theorem~\ref{matrixthm1}(e),
there is a constant $\bar\eta$ independent of~$m$ such that
\begin{align*}
 \sum_{j,k>m, jk\in G} \dfrac1{24}\ljk (1{-}\ljk )
          (\theta_j+\theta_k)^4 
         &\leq
         \dfrac23 e^{6C} \varLambda n \sum_{j=m+1}^{n} \theta_j^4
         \\
         &\leq \dfrac12\bar\eta \varLambda n 
         \sum_{j=m+1}^n  \check\phi_j^4  
                   + \dfrac12 \bar\eta\varLambda n^2 
                           \biggl(\Dfrac 1n\sum_{j=m+1}^n \abs{\check\phi_j} \biggr)^4
         \\ 
         &\leq \bar\eta\, \varLambda n \sum_{j=m+1}^n \check\phi_j^4.
\end{align*}

Putting these bounds together and using the larger domain
$U_{n-m}(\check\eta\rho)$, we find that the integral separates:
\[
  I_2(m) \leq \abs{T_m}\, I'_2(m)^{n-m},
  \text{ where }
  I'_2(m) = \int_{-\check\eta\rho}^{\check\eta\rho} \exp\(
     -\tfrac12 \varLambda n x^2 + \bar\eta \varLambda n x^4 \)\, dx.
\]
For $\abs x \leq\check\eta\rho$ we have
$e^{\bar\eta\varLambda n x^4}\leq 1+2\bar\eta\varLambda n x^4$, so
\[
  I'_2(m) \leq \int_{-\infty}^\infty e^{-\frac12\varLambda nx^2}(1+2\bar\eta\varLambda n x^4)\, dx
  = \sqrt{\frac{2\pi}{\varLambda n}}\,\biggl(1 + \frac{6\bar\eta}{\varLambda n}\biggr).
\]
Therefore,
\[
   I_2(m) \leq (2\pi)^{(n-m)/2} (\varLambda n)^{-(n-m)/2}
    \abs{A_m}^{-1/2} e^{6\bar\eta\varLambda^{-1}}.
\]
It remains to bound $\abs{A_m}^{-1/2}$.  We can write
\[
    A = \begin{pmatrix} \hat A & \bar A\trans  \\
                                   \bar A & A_m+\diag(\bar A)
          \end{pmatrix},
\]
where $\diag(\bar A)$ is a diagonal matrix whose diagonal entries
are the same as the corresponding row sums of~$\bar A$.
By the Schur Complement theorem, we have
\[
   \abs{A} = \abs{\hat A}\,\abs{A_m}\,\abs{I+\check A},
   \text{ where } \check A = A_m^{-1}\(\diag(\bar A) - \bar A\hat A^{-1}\!\bar A\trans \).
\]
In the following, $O(\cdot)$ and $\Omega(\cdot)$ are uniform over~$m$
for $m\leq n^{1-\eps/2}$.
By Lemma~\ref{lambdarange}, the diagonal entries of $\hat A$ are $\Omega(1)$
and the off-diagonal entries are $O(n^{-1})$.  This
means that the same statements hold for the entries of $\hat A^{-1}$ and
that $\abs{\hat A}=\Omega(1)^m$.
Using Theorem~\ref{matrixthm1} for the entries of $A_m^{-1}$, we find that
that $\check A$ has entries $\Omega(m/n)$ on the diagonal and
$O(m/n^2)$ off the diagonal, which also implies that $\infnorm{\check A}=o(1)$ and
the powers of $\check A$ are also diagonally dominant.
This implies $\abs{I+\check A} = \exp\( \tr \check A + O(\tr \check A^2)\) = e^{\Omega(m)}$.
Therefore $\abs{A_m} = e^{O(m)}\abs{A}$.
Putting these bounds together, for large enough $b'$,
and using~\eqref{J0est-1},
\[
   I_1(m) I_2(m) \leq \(2\pi e^{O(1)}\varLambda n e^{-2b'n^{2\eps}}\)^{m/2}
                    e^{(O(1)-b')n^{1-\sigma}} e^{6\bar\eta\varLambda^{-1}}J_0
          \leq e^{-m\,\Omega(n^{2\eps})} J_0.        
\]
Finally we sum over all the possible choices of $M_0(\thetavec)$.
\[
  \int_{\calB'''} \,\abs{F_{\betavec}(\thetavec)}\,d\thetavec
   \leq \sum_{m=1}^{n^{1-\eps/2}} 
         \binom nm e^{-m\,\Omega(n^{2\eps})} J_0
    \leq \(1 + e^{-\Omega(n^{2\eps})}\)^n J_0
   = O( e^{-\Omega(n^{2\eps})} ) J_0.  \qedhere
\]
\end{proof}

\begin{cor}\label{bigLambda}
  If $\sigma>1-2\eps$, then
  \[
     \int_{[-\pi,\pi]^n} F_{\betavec}(\thetavec)\,d\thetavec
     = \(2 + O(e^{-\Omega(n^{2\eps})})\)
     \int_{\calB_0} F_{\betavec}(\thetavec)\,d\thetavec .
  \]
\end{cor}
\begin{proof}
  If any component $\theta_j$ of $\thetavec$ lies outside $[-\eta\rho,\eta\rho]$,
  then $\theta_j^2 \geq \eta^2\rho^2  = \omega(n^{-\sigma}\varLambda^{-1})$.
  So Lemma~\ref{thirdcut} applies with $j\in M(\thetavec,\eta\rho)$.
  If $M(\thetavec,\eta\rho)=\emptyset$, then
  \[
       \calB'\cup\calB''\cup\calB'''\cup 
          U_n(\eta\rho)\cup\(U_n(\eta\rho)+(\pi,\ldots,\pi)\trans \)
       = U_n(\pi) \pmod{2\pi}.
  \]
  Choose $\eta$ so that $U_n(\eta\rho)\subseteq\calB_0$. Since the
  integrands in~Lemmas~\ref{firstcut}--\ref{thirdcut} are
  nonnegative, we have that
  \[
       \int_\calB \,\abs{F_{\betavec}(\thetavec)}\,d\thetavec
       = O\(e^{-\Omega (n^{2\eps})})\)
       \int_{\calB_0} F_{\betavec}(\thetavec)\,d\thetavec,
  \]
  and this implies the corollary.
\end{proof}

At this point we will summarize the regions in which we have shown
the integral of the absolute value of the integrand to be negligible.
Since we will need to work with $\phivec$ as well as
$\thetavec=T\phivec$, we will give some consequences in
those coordinates too.

Let $T=(t_{jk})$ and for any constant $\eta>0$ define
\begin{equation}\label{bdefs}
  \begin{aligned}
   \rho_0 &:= \frac{\min_j t_{jj}}{9\max t_{jj}}\rho = \Theta(\rho) \\
        \calB_{\thetavec}(\eta) &:= \Bigl\{ \thetavec\in [-\pi,\pi]^n \St
            \sum_{j\in M(\thetavec,\eta\rho)} \theta_j^2
              \leq b(\eta) n^{-\sigma}\varLambda^{-1} \Bigr\} \\
       \calB_{\thetavec}^\pi(\eta) &:= \calB(\eta) + (\pi,\ldots,\pi)\trans 
          \pmod{2\pi} \\
       \calB_{\phivec} &:= \Bigl\{ \phivec\in T^{-1}U_n(\pi) 
          \St    \sum_{j\in M(\phivec,\rho_0)} \phi_j^2
             \leq  n^{-\sigma+\eps}\varLambda^{-1} \Bigr\}.
  \end{aligned}
\end{equation}

\begin{lemma}\label{goodplaces}
We have
\begin{enumerate}[label=(\alph*),itemsep=0pt]
  \item If $\thetavec\in\calB_{\thetavec}(\eta)$, then
    $\abs{M(\thetavec,\eta\rho)}  = O(n^{1-\sigma-2\eps})$ and
    $\sum_{j\in M(\thetavec,\eta\rho)} \abs{\theta_j} =O(\rho n^{1-\sigma-2\eps})$.
    If~$\phivec\in\calB_{\phivec}$, then
    $\abs{M(\phivec,\rho_0)}  = O(n^{1-\sigma-\eps})$ and
     $\sum_{j\in M(\phivec,\rho_0)} \abs{\phi_j} =O(\rho n^{1-\sigma-\eps})$.
\item $\displaystyle
  \int_{[-\pi,\pi]^n-(\calB_{\thetavec}(\eta)\cup\calB _{\thetavec}^\pi(\eta))}
    \abs{F_{\betavec}(\thetavec)}\,d\thetavec
      = O(e^{-\Omega(\log^2 n)}) J_0$.
 \item There is a constant $\eta>0$  such that
    $T^{-1}\calB _{\thetavec}(\eta) \subseteq \calB_{\phivec}$.
\end{enumerate}
\end{lemma}
\begin{proof}
   For part (a), the definition of $\calB_{\thetavec}(\eta)$ implies that
  \[
  \abs{M(\thetavec,\eta\rho)}\leq \Dfrac{1}{\eta^2\rho^2}
   \sum_{j\in M(\thetavec,\eta\rho)} \theta_j^2  =O(n^{-\sigma}\varLambda^{-1}/\rho^2)
    = O(n^{1-\sigma-2\eps}).
    \]
    Also,
    \[
     \sum_{j\in M(\thetavec,\eta\rho)} \abs{\theta_j} 
      \leq \Dfrac{1}{\eta\rho} \sum_{j\in M(\thetavec,\eta\rho)} \theta_j^2 
         =O(\rho n^{1-\sigma-2\eps}).
    \] 
    The same arguments apply to $\phivec\in\calB_{\phivec}(\alpha)$.
    
  By Lemma~\ref{secondcut}, with a loss of at most $O(e^{-\Omega(\log^2 n)}) J_0$
  we can assume that either at least $n-n^{1-\eps/2}$ components of $\thetavec$
  lie in $[-\frac12\eta\rho,\frac12\eta\rho]$, or at least least $n-n^{1-\eps/2}$ components of $\thetavec$ lie in $[-\pi+\frac12\eta\rho,-\pi]\cup
 [\pi-\frac12\eta\rho,\pi]$.  Since these situations are disjoint and equivalent
 under the symmetries $F_{\betavec}(\thetavec)$, we can continue with the first.
By Lemma~\ref{thirdcut}, we lose at most a further
 $O(e^{-\Omega(n^{2\eps})}) J_0$ if we assume that
 $\sum_{j\in M(\thetavec,\eta\rho)} \theta_j^2 =O(n^{-\sigma}\varLambda^{-1})$,
 which holds since $\thetavec\in\calB_{\thetavec}(\eta)$.  That proves part~(b).
 
By Lemma~\ref{quadform}(c), there is a constant $u$ such that
$\abs{(T^{-1})_{jj}}\leq u$ and $\abs{(T^{-1})_{jk}}\leq u/n$ for $1\leq j\ne k\leq n$.
Consider $\eta = \rho_0/(3u\rho)$ and suppose $\thetavec\in\calB_{\thetavec}(\eta)$.
We wish to show that $\phivec\in\calB_{\phivec}$, where
$\thetavec=T\phivec$. 

Suppose $j\notin M(\thetavec,\eta\rho)$; i.e., $\abs{\theta_j}\leq \eta\rho$.
Then, using part (a),
\[
   \abs{\phi_j} \leq u\abs{\theta_j} + \Dfrac un\sum_{k\ne j}\, \abs{\theta_j}
      \leq u\eta\rho + u\eta\rho + O(\rho n^{-\sigma-2\eps})
      \leq \rho_0.
\]
Therefore $M(\phivec,\rho_0)\subseteq M(\thetavec,\eta\rho)$.
Finally, using $(x+y+z)^2\leq 3(x^2+y^2+z^2)$, we have
\begin{align*}
  \sum_{j\in M(\phivec,\alpha\rho)} \phi_j^2
     &\leq \sum_{j\in M(\thetavec,\eta\rho)} \phi_j^2 \\
     &\leq 3\sum_{j\in M(\thetavec,\eta\rho)} \(u^2\theta_j^2 
          + u^2\eta^2\rho^2 + O(\rho^2 n^{-2\sigma-4\eps})\) \\
     &= (6 +o(1))u^2\sum_{j\in M(\thetavec,\eta\rho)} \theta_j^2
          \leq n^{-\sigma+\eps}\varLambda^{-1},
\end{align*}
and so $\phivec\in\calB_{\phivec}$ as required.
\end{proof}

\nicebreak
\subsection{The integral outside $\calB_0$ and $\calB_\pi$; part 2}
\label{s:outside2}

In this section we will find a bound on the integral of $F_{\betavec}(\thetavec)$
over a parallelepiped similar to $\calB_0$ but shifted away from the origin.
Such regions will be used as tiles in the following section to
complete the proof of Theorem~\ref{mainthm}.

Let $\emptyset\subseteq J\subseteq[n]$.
For notational simplicity, we will use the example $J=\{1,\ldots,m\}$.
Recall that
\[
    J_0 = \abs{A}^{-1/2} \int_{U_n(\rho)} F_{\betavec}(T\phivec)\,d\phivec.
\]
In this section we will bound $(n-m)$-dimensional integrals of the
form
\[
    I(\phi_1,\ldots,\phi_m) 
        := \int_{U(\phi_1,\ldots,\phi_m)} F_{\betavec}(T\phivec)\, d\phi_{m+1}\cdots d\phi_n,
\]
where the region of integration is a product of intervals
\[
   U(\phi_1,\ldots,\phi_m) := I_{m+1}\times\cdots\times I_n.
\]
We will work under the following assumptions.

\begin{assume}\label{assumeoutside}
\leavevmode
\begin{itemize}\itemsep=0pt
 \item[(a)] Assumptions~\ref{mainassumptions} hold.
 \item[(b)] $\phi_1,\ldots,\phi_m$ satisfy
    $\abs{\phi_1},\ldots,\abs{\phi_m}\geq\rho_0$  and
    $\sum_{j=1}^m \phi_j^2 \leq \varLambda^{-1}n^{-\sigma+\eps}$.
 \item[(c)] For $m+1\leq j\leq n$, $I_j=I_j(\phi_1,\ldots,\phi_m)$
 is an interval with\\
 $[-\frac13 \rho,\frac13 \rho]\subseteq I_j\subseteq [-\frac43\rho,\frac43\rho]$.
 \item[(d)] $\sigma \leq 1-2\eps$.
\end{itemize}
\end{assume}

Define $S_1:=\sum_{j=1}^m\abs{\phi_j}$.  Then we have
\begin{equation}\label{S1bound}
 S_1 = O\( \min\{ \rho n^{1-\sigma - \eps}, 
    m^{1/2}\varLambda^{-1/2}n^{-\sigma/2+\eps/2}, m\} \).
\end{equation}
The first part follows from  Lemma~\ref{goodplaces}(a), the second
from Assumption~\ref{assumeoutside}(b) by the Cauchy-Schwartz
inequality, and the third is obvious.

Split $\thetavec$ into $\yvec+\zvec$ where
\begin{align*}
\yvec&=\yvec(\phivec) := T(0,\ldots,0,\phi_{m+1},\ldots,\phi_n)\trans , \\
\zvec&=\zvec(\phivec) := T(\phi_1,\ldots,\phi_m,0,\ldots,0)\trans .
\end{align*}
Using Lemma~\ref{quadform}(c), we have the following bounds.
\begin{align}
 y_j &= \begin{cases}
     \sum_{\ell\geq m+1} O(n^{-1})\phi_\ell,  & \text{~if $j\leq m$,} \\
      \Theta(\phi_j) + \sum_{\ell\geq m+1} O(n^{-1})\phi_\ell,  &
               \text{~if $j\geq m+1$}.
      \end{cases} \label{yjbounds}
\\[0.6ex]
 z_j &= \begin{cases}
      \Theta(\phi_j) + O(S_1 /n),  & \text{~if $j\leq m$,} \\
      O(S_1/n) & \text{~if $j\geq m+1$}.
      \end{cases} \label{zjbounds}
\end{align}
For $jk\in G$, define
\[
    \tlambda_{jk} := \frac{\ljk  e^{i(z_j+z_k)}}
                                       {1 + \ljk  (e^{i(z_j+z_k)}-1)},
\]
which satisfies
\begin{equation}\label{lambdatilde}
 1+\ljk (e^{i(\theta_j+\theta_k)}-1) 
 =
 \( 1+\ljk (e^{i(z_j+z_k)}-1) \)\,
 \(1+\tlambda_{jk}(e^{i(y_j+y_k)}-1) \).
\end{equation}
Note that $\tlambda_{jk}$ is generally not real.
Using $\varLambda=o(1)$ we can calculate that
\[
    c_\ell(\ljk ) - c_\ell(\tlambda_{jk}) = O(\varLambda)(z_j+z_k),
\]
for $1\leq\ell=O(1)$.

Since $\zvec$ is independent of $\phi_{m+1},\ldots,\phi_n$ and
$\sum_{j=1}^n d_j\theta_j=\sum_{j=1}^n d_j y_j + \sum_{j=1}^n d_j z_j$,
we can apply~\eqref{lambdatilde} to obtain
\begin{align}
   \Abs{I(\phi_1,\ldots,\phi_m)} &=
    \biggl| \prod_{jk\in G} \(1 + \ljk (e^{i(z_j+z_k)}-1)\) \biggr|
    \; \abs{\hat I(\phi_1,\ldots,\phi_m)}, \label{Jm}
\intertext{where}
 \hat I(\phi_1,\ldots,\phi_m) &:= \int_{U(\phi_1,\ldots,\phi_m)} 
    \frac{\prod_{jk\in G} \(1 + \tlambda_{jk}(e^{i(y_j+y_k)}-1)\)}
              {e^{i\sum_{j=1}^n d_j y_j}}\, d\phi_{m+1}\cdots d\phi_n. \notag
\end{align}

Next we will raise the $(n-m)$-dimensional
integral  $\hat I$ to $n$ dimensions by introducing
new independent variables $\tphi_1,\ldots,\tphi_m$.
For notational convenience we also define
$\tphi_{m+1},\ldots,\tphi_n$ to be synonyms for $\phi_{m+1},\ldots,\phi_n$.
$\hat I$ is unchanged if we multiply the integrand by
\[
   (1+o(1))\biggl(\frac{2\pi}{\varLambda  n}\biggr)^{\!-m/2}
     e^{-\frac12\varLambda n(\tphi_1^2+\cdots+\tphi_m^2)},
\]
where the $o(1)$ depends only on $m,n$ and is chosen to make this
an $m$-dimensional density.
Also, similarly to Section~\ref{s:inbox},
define $\ringX$ to be the $n$-dimensional random variable whose
density is proportional to $e^{-\frac12\varLambda n\sum_{j=1}^n x_j^2}$,
truncated to the region $U_m(\rho)\times U(\phi_1,\ldots,\phi_m)$.
Then
\begin{align}\label{Jmbound}
\hat I(\phi_1,\ldots,\phi_m) &= 
 O(1) \biggl(\frac{2\pi}{\varLambda  n}\biggr)^{\!(n-m)/2}
 \,
 \E e^{f_{\phi_1,\ldots,\phi_m}(\ringX)},
\intertext{where}
 f_{\phi_1,\ldots,\phi_m}(\tphivec)
   &:= \sum_{jk\in G} \ln\(1+\tlambda_{jk}(e^{i(y_j+y_k)}-1)\)
    + \dfrac12\varLambda n\!\!\sum_{j=m+1}^n \tphi_j^2
    - i\sum_{j=1}^n d_j y_j. \notag
\end{align} 

From \eqref{approxbeta}, we have $i\sum_{j=1}^n (d_j+\delta_j) y_j=\sum_{jk\in G} c_1(\ljk ) (y_j+y_k)$.
By $T\Trans5 AT=I$ and the definition of~$\yvec$, we also have
$\frac12\varLambda n \sum_{j=m+1}^n \tphi_j^2
= \frac12\varLambda n\yvec\Trans6 A\yvec =
-\sum_{jk\in G} c_2(\ljk )(y_j+y_k)^2$.
Therefore,
\[
  f_{\phi_1,\ldots,\phi_m}(\tphivec)
   = i\sum_{j =1}^n \delta_j y_j +
   \sum_{\ell=1}^2 \sum_{jk\in G} \( c_\ell(\tlambda_{jk})-c_\ell(\ljk )\)
                                     (y_j+y_k)^\ell
   + \sum_{\ell=3}^\infty  \sum_{jk\in G} c_\ell(\tlambda_{jk})
                                        (y_j+y_k)^\ell.
\]

\begin{lemma}\label{redistribution}
If\/ $\medtilde\Y:=T(0,\ldots,0,X_{m+1},\ldots,X_n)$, then
there are $\alpha_j=\alpha_j(\phi_1,\ldots,\phi_m)$ and
$\gamma_{jk}=\gamma_{jk}(\phi_1,\ldots,\phi_m)$ such that
\begin{align*}
    \abs{\alpha_j} &= O(\varLambda S_1), \quad\text{~for $1\leq j\leq n$}, \\
    \abs{\gamma_{jk}} &= \begin{cases}
               O(\varLambda S_1), & \text{~for $1\leq j=k\leq n$}; \\
               O(\varLambda n^{-1} S_1), & \text{~for $1\leq j<k\leq n$}.
               \end{cases}
\end{align*}
and
\[
  \sum_{\ell=1}^2 \sum_{jk\in G} \( c_\ell(\tlambda_{jk})-c_\ell(\ljk )\)
                                      (\medtilde Y_j+\medtilde Y_k)^\ell
   = \sum_{j=1}^n {\alpha_j\medtilde Y_j} 
      + \!\!\sum_{1\leq j\leq k\leq n} \gamma_{jk}\medtilde Y_j\medtilde Y_k.
\]  
\end{lemma}
\begin{proof}
Let $\evec_1,\ldots,\evec_n$ be the standard basis and $\fvec_1,\ldots,\fvec_n$
be the columns of~$T$.  
The equation $\yvec=(0,\ldots,0,\phi_{m+1}\ldots,\phi_n)\trans $ implies that
\[ 
   \sum_{j=m+1}^n \phi_j\fvec_j - \sum_{j=1}^m y_j\evec_j=\sum_{j=m+1}^n  y_j\evec_j.
\]
Let $T_J$ be the matrix whose columns are
$\evec_1,\ldots,\evec_m,\fvec_{m+1},\ldots,\fvec_n$.  Then
\[
    (-y_1,\ldots,-y_m,\phi_{m+1},\ldots,\phi_n)\trans  =
   T_J^{-1}(0,\ldots,0,y_{m+1},\ldots,y_n)\trans .
\]
The lemma now follows from Lemma~\ref{quadform} and Lemma~\ref{l:Mtrans}.
\end{proof}

\begin{lemma}\label{fewcumulants}
Define $r_1$ and $\ell_0$ as in Theorem~\ref{mainthm} and~\eqref{sweak},         
and
$\{ \alpha_j\}, \{\gamma_{jk}\}$ as in Lemma~\ref{redistribution}.
Then, under Assumptions~\ref{assumeoutside},
\[
     \E e^{f_{\phi_1,\ldots,\phi_m}(\ringX)} = (1+O(n^{-p}))
       \exp\biggl(\sum_{r=1}^{r_1} 
       \dfrac{1}{r!}\, \kappa_r(\medtilde f_{\phi_1,\ldots,\phi_m}(\X))\biggr),
 \]
 where
 \[
   \medtilde f_{\phi_1,\ldots,\phi_m}(\X)
    :=  \sum_{j =1}^n (i \delta_j +\alpha_j)\medtilde Y_j 
      + \!\!\sum_{1\leq j\leq k\leq n} \gamma_{jk}\medtilde Y_j\medtilde Y_k
    + \sum_{\ell=3}^{\ell_0}  \sum_{jk\in G} c_\ell(\tlambda_{jk})
                                         (\medtilde Y_j+\medtilde Y_k)^\ell
 \]
 with $\medtilde\Y=T(0,\ldots,0,X_{m+1},\ldots,X_n)$.
\end{lemma}
\begin{proof}
Since $S_1=O(\rho n^{1-\sigma-\eps})$, the result
follows immediately from Theorem~\ref{genbox}.
\end{proof}

\begin{lemma}\label{cumdifference}
Under Assumptions~\ref{assumeoutside},
\[
\E e^{f_{\phi_1,\ldots,\phi_m}(\ringX)} 
= e^{O(m n^{\eps})} \E e^{R(G,T\X)}.
\]
\end{lemma}
\begin{proof}
By Lemmas~\ref{truncatedsum} and~\ref{fewcumulants}, we need
to compare the first $r_1$ cumulants of $R_{\ell_0}(T\hat\X)$ and
$\medtilde f_{\phi_1,\ldots,\phi_m}(\X)$.
Our main tool will be Lemma~\ref{cumdiff}.
Define $Y_{jk}=(\varLambda n)^{1/2}(Y_j+Y_k)$ as before,
and similarly $\medtilde Y_{jk}=(\varLambda n)^{1/2}(\medtilde Y_j+\medtilde Y_k)$.
Then both $\{Y_{jk}\}$ and $\{\medtilde Y_{jk}\}$ satisfy~\eqref{covarbounds}.

With $C_\ell$ defined as in Theorem~\ref{maxcumcor}, we can take
\begin{align*}
     C_1 &= O(n^{-1-\sigma/2})\\
     C_2 &= O(n^{-2} S_1) = O(n^{-1-3\sigma/2})\\
     C_\ell &= O(\varLambda^{1-\ell/2} n^{-\ell/2}), \text{~~if $\ell\geq 3$}.
\end{align*}

Write $T=(t_{jk})$.  In the language of Theorem~\ref{cumdiff}, we have
\begin{align*}
   \sigma_{jk,j'k'}-\medtilde\sigma_{jk,j'k'}
      &= \sum_{s=1}^m (t_{js}+t_{ks})(t_{j's}+t_{k's}) \\
      &= \begin{cases}
         O(m/n^2), & \text{ if $j,k,j','k' \geq m+1$}, \\
         O(1), & \text{ if $\{j,k\}\cap\{j',k'\}\cap\{1,\ldots,m\}\ne\emptyset$}, \\
         O(1/n), & \text{ otherwise},
         \end{cases}
\end{align*}
which shows that we can take $\mu=O(m/n)$.
Since $C_\ell=o(n^{-1})$ for all $\ell\geq 1$, we have that
$\mu n^{r+1}C_{\ell_1}\cdots C_{\ell_r}=O(m)$ for all
$\ell_1,\ldots,\ell_r$ and $1\leq r\leq r_1$.

To complete the proof we calculate that
\begin{align*}
     \mu_1 C_1 &= O(\varLambda^{1/2}n^{-3/2} S_1) \\
     \mu'_1 C_1^2 &= O(n^{-3+\eps}S_1) + O(\varLambda n^{-3}S_1^2) \\
     \mu_2 C_2 &= O(n^{-2} S_1) \\
     \mu_\ell C_\ell &= O(\varLambda^{1-\ell/2} n^{-1-\ell/2} S_1),\text{~~~for $\ell\geq 3$}.
\end{align*}
Note that $C_\ell=O(n^{-1-\sigma/2})$ and 
$\mu_\ell C_\ell=O(\varLambda^{1-\ell/2}n^{-1-\ell/2})S_1$ for $\ell\geq 1$.
For $k\geq 2$, the value of $n^{r+1}\mu_k C_{\ell_1}\cdots C_{\ell_r}$ 
is $O(1)n^{r+1}(n^{-1-\sigma/2})^{r-1}\varLambda^{1-k/2}n^{-1-k/2}S_1
=O(S_1)$.
Since $\ell_1+\cdots+\ell_r$ must be even, the appearance of $\ell_j=1$ 
implies that some other odd $\ell_k$ appears. If $\ell_k\geq 3$, the contribution
of $n^{r+1}\mu_1 C_{\ell_1}\cdots C_{\ell_r}$ is also $O(S_1)$,
since $C_\ell=O(\varLambda^{-1/2}n^{-3/2})$ for $\ell\geq 3$.
The remaining case is that $\ell_j=\ell_k=1$ for distinct~$j,k$.
Then we have $n^{r+1}\mu'_1 C_{\ell_1}\cdots C_{\ell_r}
= O(n^\eps S_1)+O(\varLambda S_1^2)$. 
The proof is now complete on applying~\eqref{S1bound}.
\end{proof}

\begin{cor}\label{Ibound}
Under Assumptions~\ref{assumeoutside},
\[
    I(\phi_1,\ldots,\phi_m) = O(e^{-m\Omega(n^\eps)}) \abs{A}^{1/2} J_0.
\]
\end{cor}
\begin{proof}
Consider~\eqref{Jm}.
Since $\Abs{1+\ljk (e^{i(z_j+z_k)-1}-1)}\leq 1$ always,
$\{\ljk \}$ and $\{z_j\}$ being real, we can write
\begin{align*}
   \prod_{jk\in G} \,\Abs{1+\ljk (e^{i(z_j+z_k)}-1)}
     &\leq \prod_{\substack{jk\in G \\ j\leq m, k> m}}
        \Abs{1+\ljk (e^{i(z_j+z_k)}-1)}  \\
     &\leq \prod_{\substack{jk\in G \\ j\leq m, k> m}}
             \Abs{1+\ljk (e^{i(\phi_j+O(S_1/n))}-1)},
\end{align*}
by~\eqref{zjbounds}. 
For $1\leq j\leq m$, \eqref{yjbounds} implies
that $\abs{z_j}\leq \pi+O(\rho)$, since $\zvec=\thetavec-\yvec$,
and also $\abs{z_j}\geq\infnorm{T^{-1}}^{-1}\rho_0=\Omega(\rho)$
by our assumption on $\phi_j$.

Also, for $1\leq5 j\leq m$, Assumption~\ref{mainassumptions}(d)
and $m=o(n)$ imply that there are $\Omega(n)$
values of $k>m$ such that $jk\in G$.
Applying Lemma~\ref{boring}(a) and the fact that 
$S_1/n=o(\rho)$, we find that
\[
    \prod_{jk\in G} \,\Abs{1+\ljk (e^{i(z_j+z_k)}-1)}
    \leq e^{-m\Omega(n^{2\eps})}.
\]

Applying~\eqref{Jm}, \eqref{Jmbound} and Lemma~\ref{cumdifference},
the corollary is proved.
\end{proof}

\nicebreak
\subsection{Putting the parts together}\label{s:finale}

We now have all the pieces required to complete the proof of Theorem~\ref{mainthm}.
Define
\[
     F^*(G,\thetavec) := \begin{cases}
                  F_{\betavec}(\thetavec), &\text{ if }\Card{\{ j \st \abs{\theta_j}_{2\pi}\leq\frac12\pi\}}\leq \frac12 n; \\
                                        0, &\text{ otherwise.}
                                   \end{cases}
\]
Note that $F^*$ inherits the periodicity (period $2\pi$ in each coordinate direction) of~$F$.  Our
reason for introducing $F^*$ is the following.
\begin{lemma}\label{F*lemma}
\[
  \biggl| \int_{[-\pi,\pi]^n-\calB_0-\calB_\pi} F_{\betavec}(\thetavec)\,d\thetavec \biggr|
  = O(e^{-\Omega(\log^2 n)})J_0 +  2\,\biggl| \int_{[-\pi,\pi]^n-\calB_0} F^*(G,\thetavec)\,d\thetavec\biggr|.
\]
\end{lemma}
\begin{proof}
 Except in a domain that lies entirely within the region covered by
  Lemmas~\ref{firstcut} and~\ref{secondcut}, one of $F^*(G,\thetavec)$ and
 $F^*(G,\thetavec+(\pi,\ldots,\pi)\trans )$ is equal to $F_{\betavec}(\thetavec)$
 while the other is~0.
\end{proof}

\begin{thm}\label{boxingOK}
\[
     \int_{[-\pi,\pi]^n} F_{\betavec}(\thetavec)\,d\thetavec 
       = \(2+O(e^{-\Omega(\log^2 n)})\) J_0.
\]
\end{thm}
\begin{proof}
Since Appendix~\ref{AppendixB} is written with independent
notation for future applications, we provide the correspondence.
\[
 \text{\begin{tabular}{cc|cc}
    Theorem~\ref{MishaMagic} & This section & Theorem~\ref{MishaMagic} & This section\\
    \hline
    $M$ & $2\pi T$ &
     $\rho$ & $\rho$  \\
     $B$ & $T^{-1}[-\pi,\pi]^n$  &
     $\xvec\mapsto H(\xvec)$ & $\phivec\mapsto F^*(T\phivec)$ \\
     $B_0$ & $T^{-1}\calB_0$ &
      $\zeta$ &  $n^{-\sigma-\eps}$
 \end{tabular}}
 \]
We first establish that the conditions of Theorem~\ref{MishaMagic} are
satisfied.  $2\pi T$ satisfies~\eqref{M_property} by Lemma~\ref{quadform}(c).

Define $\Babs = T^{-1}U_n(\pi)-\calB_{\phivec}$,
where $\calB_{\phivec}$ is defined in~\eqref{bdefs}.
Since $\rho_0\leq \frac19\rho$, we know that
$\sum_{j:\phi_j\geq\rho_0} \phi_j^2\leq \varLambda^{-1}n^{-\sigma+\eps}$
implies that $\sum_{j:\phi_j\geq\rho/3} \abs{\phi_j}\leq\frac13n^{1-\sigma}\rho$,
as in Lemma~\ref{goodplaces}(a),
so condition~\eqref{S_property} is satisfied.
Next, the conditions (i) and (ii) on $\wvec=(w_1,\ldots,w_n)$
in Theorem~\ref{MishaMagic}
imply that $\abs{w_j}>\rho_0$ for $j\in J$ and
$\sum_{j\in J} w_j^2\leq \varLambda^{-1}n^{-\sigma+\eps}$,
and so $\abs J \leq n^{1-\sigma-\eps}$.

By Corollary~\ref{Ibound} this means we can take
\[
   L(j) = \begin{cases}
                 O(e^{-j\Omega(n^\eps)}) \abs{A}^{1/2} J_0,
                    & \text{~if $1\leq j\leq n^{1-\sigma-\eps}$}; \\
                 0, & \text{~otherwise}.
             \end{cases}
\]
Now applying Theorem~\ref{MishaMagic} using $\infnorm{T}=O(1)$, we have
\begin{align*}
   \biggl| \int_{T^{-1}[-\pi,\pi]^n-U_n(\rho)} \!\!F^*(T\phivec)\,d\phivec \biggr|
   &\leq \int_{\Babs} \!\Abs{F^*(T\phivec)}\,d\phivec 
   + \sum_{j=1}^n O\(n^j e^{-j\Omega(n^\eps)}\) \abs{A}^{1/2} J_0 \\
   &= \int_{\Babs} \!\Abs{F^*(T\phivec)}\,d\phivec +
        O(e^{-\Omega(n^\eps)}) \abs{A}^{1/2} J_0.
\end{align*}
By Lemma~\ref{goodplaces}(b,c), this is equivalent to
\[
   \biggl| \int_{[-\pi,\pi]^n-\calB_0} \!\!F^*(\thetavec)\,d\thetavec \biggr|
   = O(e^{-\Omega(\log^2 n)}) J_0,
\]
which completes the proof by Lemma~\ref{inbox} and Lemma~\ref{F*lemma}.
\end{proof}

%%%%%%%%%%%%%%%%%%%%%%%%%%%%%%%%%%%%%%%%%%%%%%%%%%%%%%%%%%%%%%%%%%%%%%%%%%%%

\nicebreak
\section{Regular graphs}\label{s:regular}

The simplest case of Theorem~\ref{mainthm} is the estimation of the
number $\RG(n,d)$ of regular graphs of order~$n$ and degree~$d$.
Please refer to~\cite{MWsparse} for the early history.
The main milestones for the problem are~\cite{MWsparse} for
$d=o(n^{1/2})$, \cite{MWreg} for $\min\{d,n-d-1\} \geq cn/\log n$ for
constant $c>\frac23$, and \cite{Liebenau} for a large range that overlaps
both the sparse and dense regimes. 
The main term, for the full range $1\leq d\leq n-2$ is
\begin{equation}\label{regconjecture}
  \RG (n,d) \sim \sqrt2\, e^{1/4}
     \(\lambda^\lambda (1{-}\lambda)^{1{-}\lambda}\)^{\binom n2}
      \binom{n-1}{d}^{\!n} .
\end{equation}
We will show that there is an asymptotic expansion whenever
$\min\{d,n-d-1\} \geq n^\sigma$ for some $\sigma>0$.

In terms of the earlier terminology, $G$ is a complete graph,
$\lambda=d/(n-1)$ and $d_1=\cdots=d_n=d$.
The unique solution to~\eqref{betaeqn} is
$\ljk =\lambda$ for all $j,k$, corresponding
to $\beta_1=\cdots=\beta_n=\frac12\ln\( \lambda/(1-\lambda) \)$. 

The matrix $A$ is $(n-1)/n$ on the diagonal and
$1/n$ off the diagonal. Its eigenvalues are $2-2/n$ (once)
and $1-2/n$ ($n-1$ times), so its determinant is
$\abs{A}=(2-2/n)(1-2/n)^{n-1}$.

In order to compare $\RG(n,d)$ to~\eqref{regconjecture}, we cast our
expansion in  similar terms.
Define
\[
   \eps_1 := \log(1-1/n)+(n-1)\log(1-2/n)
\]
and define $\xi(N)=\Dfrac{1}{12N}-\Dfrac{1}{360N^3}+\cdots\,$ by
\[
    N! = \frac{N^N}{e^N}\sqrt{2\pi N}\,e^{\xi(N)}.
\]
and then $\eps_2 = -\Dfrac{1-7\varLambda}{12\varLambda}
+ \Dfrac{1-4\varLambda}{12\varLambda n} + \cdots\,$ by
\[
   \eps_2 := n\xi(n-1) - n\xi(\lambda(n-1)) - n\xi((1-\lambda)(n-1))
     - \dfrac12 n\log(1-1/n),
\]
which satisfies
\[
  \binom{n-1}{d}^{\!n}
  = \frac{\(\lambda^\lambda (1{-}\lambda)^{1{-}\lambda}\)^{-n(n-1)}}
            {2^{n/2} \pi^{n/2} n^{n/2}}
     \, e^{\eps_2}.
\]

Define the random variable $\Y$ and $m,s,R_m(\Y)$ as in
Theorem~\ref{mainthm}.  Then we have
\begin{align*}
   \RG (n,d) = \sqrt2\, &
     \(\lambda^\lambda (1{-}\lambda)^{1{-}\lambda}\)^{\binom n2}
      \binom{n-1}{d}^{\!n} \\
   &\times \exp\biggl( -\dfrac12\eps_1 - \eps_2 
       + \sum_{r=1}^{r_0} \,\dfrac{1}{r!} \kappa_r(R_m(\Y)) + O(n^{-p})\biggr).
\end{align*}

Define the random variables $Y_{jk}=(\varLambda n)^{1/2}(Y_j+Y_k)$
as before.  For explicit computation of the cumulants of
$R_m(\Y)$ we can use Theorem~\ref{maxcumcor} to determine
which isomorphism types of graphs
$K(Y_{j_1k_1}^{\ell_1},\ldots,Y_{j_rk_r}^{\ell_r})$ are required,
then the number of labelled graphs in each isomorphism class
are determined by the automorphism groups.
The necessary covariances are:
\[
  \Cov(Y_{jk},Y_{j'k'}) = \begin{cases}
       \displaystyle\sigma_2 = \frac{2n}{(n-1) },
           & \text{~if~~}\abs{\{j,k\}\cap\{j',k'\}} = 2; \\[2ex]
       \displaystyle\sigma_1 = \frac{(n-3)n}{(n-1)(n-2) },
           & \text{~if~~}\abs{\{j,k\}\cap\{j',k'\}} = 1; \\[2ex]
       \displaystyle \sigma_0=-\frac{2n}{(n-1)(n-2) },
           & \text{~if~~}\abs{\{j,k\}\cap\{j'k'\}} = 0.
    \end{cases}
\]

Figure~\ref{fig:cumexample} shows the contribution to the
third cumulant from the form $\kappa(Y_{uv}^3,Y_{vw}^4,Y_{wx}^3)$.
The pairing in the middle corresponds to $\sigma_0\sigma_1^2\sigma_2^2$.
As explained in Theorem~\ref{cums}, the pairing on the right doesn't contribute
at all, because the graph whose vertices are the three dashed groups
(called $G_\pi$ in the theorem) is disconnected.
The total over all pairings with connected $G_\pi$ is
$216\sigma_0^2\sigma_1^2\sigma_2 + 216\sigma_0\sigma_1^2\sigma_2^2
+ 108\sigma_1^2\sigma_2^3 + 216\sigma_0\sigma_1^4+144\sigma_1^4\sigma_2$.
This must be multiplied by the distinct choices of $u,v,w,x$, namely 
$\frac12 n(n-1)(n-2)(n-3)$, where the $\frac12$ comes
from the automorphism group of $K(Y_{uv}^3,Y_{vw}^4,Y_{wx}^3)$.
Finally, the coefficient
$c_3^2c_4=-\frac{1}{864}\varLambda^3(1-4\varLambda)(1-6\varLambda)$
needs to be applied.

\begin{figure}
\[ \includegraphics[scale=0.73]{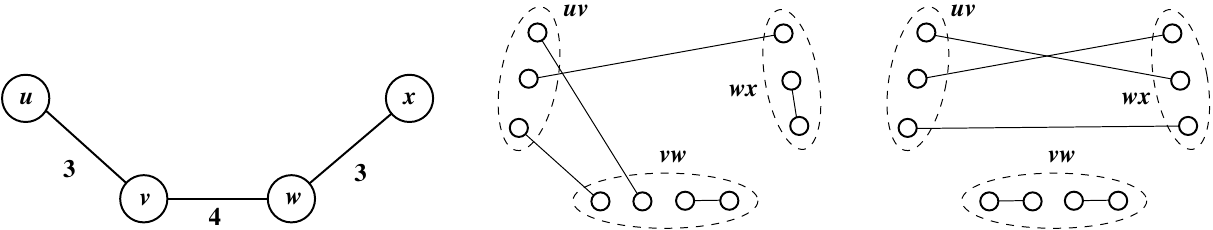} \]
\caption{$K(Y_{uv}^3,Y_{vw}^4,Y_{wx}^3)$ and two pairings.}
 \label{fig:cumexample}
 \end{figure}

In Table~\ref{table} we give enough cumulants to show that
\[
    -\dfrac12\eps_1 - \eps_2 
       + \sum_{r=1}^{r_0} \,\dfrac{1}{r!} \kappa_r(R_m(\Y))
    = \dfrac14 + O(n^{-1/3})
\]
for  $\varLambda = \Omega(n^{-5/9})$ (i.e., for
$\min\{d,n-d-1\} \geq n^{4/9}$), which is enough to verify again
that~\eqref{regconjecture} is valid for all~$d$.
In each case we give a representative graph, and then the
total contribution of all isomorphic graphs.

\def\contrib#1{{\large$#1$}${}+O(n^{-1/3})$}

\begin{table}[t]
\centering
\renewcommand{\arraystretch}{1.4}
\begin{tabular}{c|c||c|c}
 $K(Y_{12}^4)$ & \contrib{\frac{1-6\varLambda }{4\varLambda}} &
 $K(Y_{12}^4,Y_{13}^3,Y_{24}^3)$ &
   \contrib{- \frac{4}{\varLambda^2n}} \\
 $K(Y_{12}^6)$ & \contrib{- \frac{1}{12\varLambda^2n}} & 
 $K(Y_{12}^4,Y_{34}^3,Y_{15}^3)$ &
   \contrib{\frac{10}{\varLambda^2n}} \\
 $K(Y_{12}^3,Y_{13}^3)$ &
   \contrib{-\frac{7(1-4\varLambda)}{6\varLambda }} &
 $K(Y_{12}^4,Y_{34}^3,Y_{56}^3)$ &
  \contrib{- \frac{3}{2\varLambda^2n}} \\
 $K(Y_{12}^3,Y_{34}^3)$ & 
    \contrib{\frac{1-4\varLambda }{2\varLambda }} &
 $K(Y_{12}^3,Y_{13}^3,Y_{14}^3,Y_{15}^3)$ &
  \contrib{\frac{37}{2\varLambda^2n}} \\
 $K(Y_{12}^5,Y_{13}^3)$ &
   \contrib{\frac{4}{3\varLambda^2n}} &
 $K(Y_{12}^3,Y_{13}^3,Y_{14}^3,Y_{25}^3)$ &
   \contrib{\frac{30}{\varLambda^2n}} \\
 $K(Y_{12}^5,Y_{34}^3)$ &
    \contrib{- \frac{1}{2\varLambda^2n}} &
 $K(Y_{12}^3,Y_{13}^3,Y_{24}^3,Y_{35}^3)$ &
   \contrib{\frac{6}{\varLambda^2n}} \\
 $K(Y_{12}^4,Y_{13}^4)$ &
    \contrib{\frac{13}{24\varLambda^2n}} &
 $K(Y_{12}^3,Y_{13}^3,Y_{45}^3,Y_{16}^3)$ &
   \contrib{- \frac{30}{\varLambda^2n}} \\
 $K(Y_{12}^4,Y_{13}^3,Y_{14}^3)$ &
   \contrib{- \frac{37}{4\varLambda^2n}} &
 $K(Y_{12}^3,Y_{13}^3,Y_{45}^3,Y_{26}^3)$ &
   \contrib{- \frac{24}{\varLambda^2n} } \\
 $K(Y_{12}^3,Y_{13}^4,Y_{24}^3)$ &
   \contrib{- \frac{3}{\varLambda^2n}} &
  $K(Y_{12}^3,Y_{34}^3,Y_{56}^3,Y_{17}^3)$ &
   \contrib{\frac{18}{\varLambda^2n}} \\
  \multicolumn{2}{c||}{ } &
  $K(Y_{12}^3,Y_{34}^3,Y_{56}^3,Y_{78}^3)$ &
   \contrib{- \frac{2}{\varLambda^2n}} \\
\end{tabular}
\caption{Contributions to cumulants for $\varLambda=\Omega(n^{-5/9})$}
\label{table}
\end{table}

By considering the general form, we find that there are
polynomials $p_j(x)$ for $j\geq 2$, with $p_j$ having degree
$j$ for each~$j$,
such that Theorem~\ref{regularthm} holds.

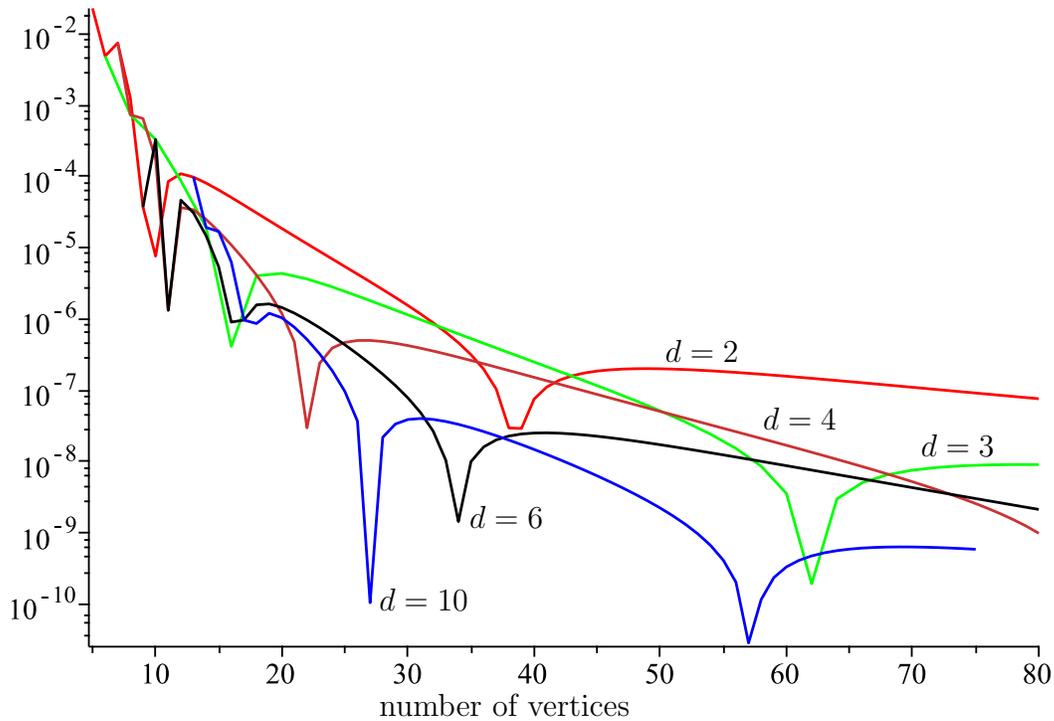
\begin{figure}[ht]
\setlength{\unitlength}{1.0cm}
\begin{picture}(12,9.3)
  \put(0.4,-0.3){\includegraphics[scale=0.75]{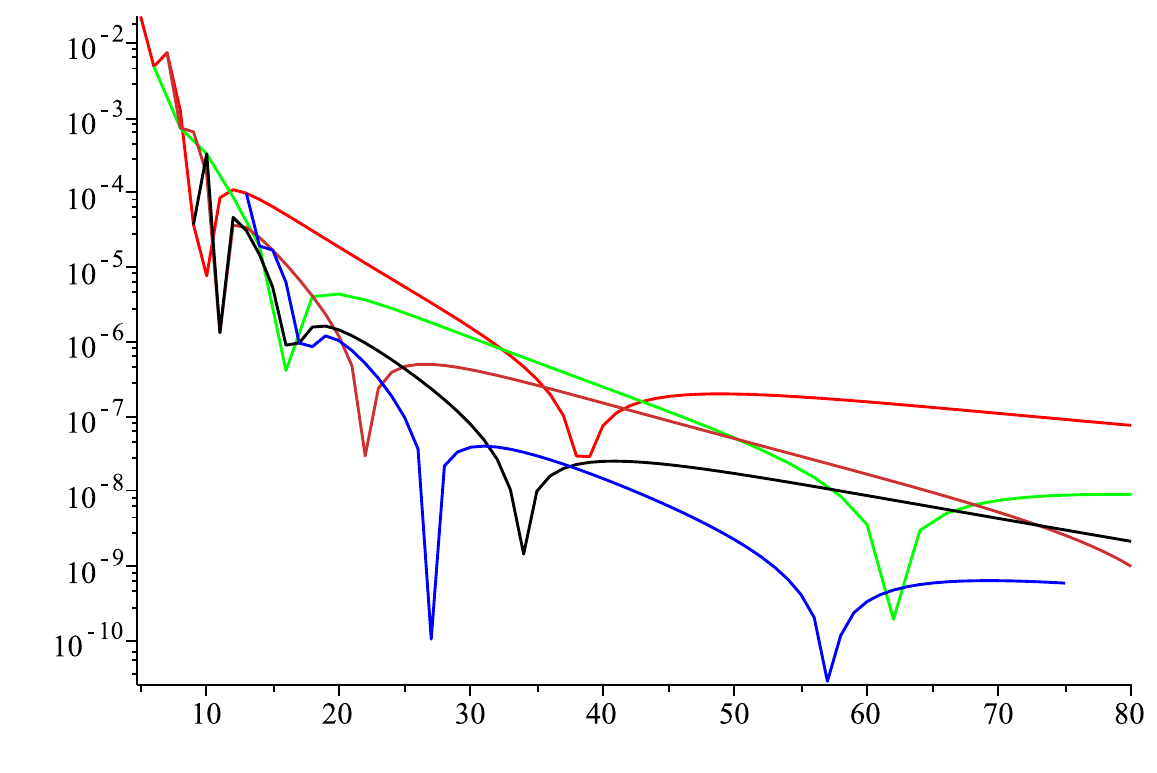}}
  \put(9.8,4.6){$d=2$}
  \put(13.2,3.35){$d=3$}
  \put(11.1,3.7){$d=4$}
  \put(7.2,2.4){$d=6$}
  \put(6.0,1.3){$d=10$}
  \put(6,-0.1){number of vertices}
\end{picture}
\caption{Absolute relative precision for low degree}
\label{plot}
\end{figure}

Using cumulants up to the 14th, we have
\begin{align*}
  p_1(x) &= \dfrac14 x, \\
  p_2(x) &= -\dfrac14 x^2, \\
  p_3(x) &= \dfrac{1}{24}(2-23x)x^2, \\
  p_4(x) &= \dfrac{1}{24}(22-129x)x^3, \displaybreak[1] \\
  p_5(x) &= -\dfrac{1}{12}(3-115x +483x^2)x^3, \displaybreak[1]\\
  p_6(x) &= -\dfrac{1}{60}(375-6615x+22097x^2)x^4, \\
  p_7(x) &= \dfrac{1}{720}(1046-87318x+1002900x^2-2791541x^3)x^4.
\end{align*}
This computation required 1,504,943 multigraphs with up to 14 edges.
The second author's graph isomorphism software \texttt{nauty}~\cite{PGI2}
was used to find the multigraphs and their automorphism groups,
as well as to assist in counting pairings when the multigraph has
low diameter. 

It is interesting to note that $p_j(x)$ has $x^{\lfloor j/2\rfloor+1}$ as a divisor for
$2\leq j\leq 7$.  If this is true for all larger~$j$, Theorem~\ref{regularthm}
with a decreasing error term may hold even for constant~$d$.
This is what happens experimentally, as illustrated in Figure~\ref{plot}
for~$k=7$.
However, we stress that we have not proved this observation.

De Panafieu found an asymptotic expansion of a different form
for constant degree, though it is probably valid for increasing
degree within some (undetermined) range~\cite{Panafieu}.
De Panafieu's expansion agrees with ours for as many terms as
we can compute for both, which provides further evidence that
our expansion works even for constant degree.
In fact, this enables us to conjecture the next few coefficients in
our expansion.
\begin{conj}\label{p89conj}
\begin{align*}
   p_8(x) &= \dfrac{1}{1680}(104594-3726282x+31805060x^2
         - 75882319x^3)x^5 \\
   p_9(x) &= -\dfrac{1}{180}(2235-329800x+7204710x^2-48922725x^3
         +102061471x^4)x^5 \qedhere
\end{align*}
\end{conj}

If $E(n,d)$ denotes the actual value of the term denoted $O(\varLambda^{-k-1}n^{-k})$
in~\eqref{regularexpansion} for the case $k=7$, we found that
$\abs{E(n,d)}<9\varLambda^{-4}n^{-7}$ for all values we were able to
compute for $n\geq 10$ and $1\leq d\leq n-2$.

For Figure~\ref{plot} we computed many exact values of $\RG(n,d)$ using
the method of~\cite{Exact} and a simple recursion.
They are available in machine-readable form at~\cite{numbers}.
%For example, we find that \xix
%{\def\s{,\allowbreak}
%$\RG(37\s18)=1\s623\s781\s578\s195\s608\s520\s
%499\s530\s872\s136\s589\s940\s152\s231\s286\s488\s859\s382\s919\s
%114\s991\s046\s229\s670\s472\s484\s
%494\s138\s991\s443\s462\s306\s423\s339\s861\s891\s134\s056\s
%872\s303\s885\s810\s636\s392\s392\s221\s
%138\s607\s568\s402\s692\s102\s209\s309\s408\s000$},
For example, the exact value of $\RG(37,18)$ is used 
Table~\ref{table3718}
to illustrate the precision of Theorem~\ref{regularthm} for increasing~$k$.

\begin{table}[ht]
\centering
\begin{tabular}{c|c|l}
$k$ & value of \eqref{regularexpansion} & rel.\,error \\
\hline
1 & $1.6355978242\times 10^{168}$ & $7.3\times 10^{-3}$ \\
2 & $1.6245837392\times 10^{168}$ & $4.9\times 10^{-4}$ \\
3 & $1.6238422249\times 10^{168}$ & $3.7\times 10^{-5}$ \\
4 & $1.6237874594\times 10^{168}$ & $3.6\times 10^{-6}$ \\
5 & $1.6237823335\times 10^{168}$ & $4.7\times 10^{-7}$ \\
6 & $1.6237816946\times 10^{168}$ & $7.2\times 10^{-8}$ \\
7 & $1.6237815979\times 10^{168}$ & $1.2\times 10^{-8}$ \\
8 & $1.6237815817\times 10^{168}$ & $2.2\times 10^{-9}$ \\
9 & $1.6237815788\times 10^{168}$ & $3.9\times 10^{-10}$ \\
\hline
exact & \multicolumn{1}{c}{$1.6237815782\times 10^{168}$}
\end{tabular}
\caption{Values of \eqref{regularexpansion} for $\RG(37,18)$
(assuming Conjecture~\ref{p89conj} for $k=8,9$)}
\label{table3718}
\end{table}

%\nicebreak

\nicebreak

\section{The probability of an edge}\label{s:observations}

Let $H^-,H^+$ be disjoint subgraphs of~$G$, with degree sequences
$\hvec^-,\hvec+$. 
Then the probability that a random $\dvec$-factor of~$G$
contains~$H^+$ and is edge-disjoint from~$H^-$ is
\[
      \frac{N(G-H^--H^+,\dvec-\hvec^+)}{N(G,\dvec)}.
\]

This opportunity to explore the structure of random factors will be
taken up in another paper.
In this paper we will be content with the probability that a pair of vertices
appears as an edge in a random $\dvec$-factor.

Precise bounds were given in~\cite{Msubgraphs} when
$\dmax(\dmax+\lmax)=o(\sum_{j=1}d_j)$, where 
$\dmax=\max_{j=1}^n d_j$ and $\lmax$ is the maximum degree
of the complement of~$G$.
The case of dense near-regular degree sequences, with $G$ almost a
complete graph, was done in~\cite{ranx}.
The intermediate range of average degree for near-regular degree
with $G=K_n$ was solved in~\cite{Liebenau}.
Although~\cite{Liebenau} does not explicitly engage with the
beta equations~\eqref{betaeqn}, the expression it gives for the
probability of an edge is compatible with the more
general Theorem~\ref{edgeprob}.

\nicebreak

\begin{proof}[Proof of Theorem~\ref{edgeprob}]
We have
\[
    \frac{\Prob(uv\in H)}{1-\Prob(uv\in H)} = \frac{N(G-uv,\dvec-\evec)}{N(G-uv,\dvec)},
\]
where $\evec$ is the degree sequence of~$uv$.
 
 Define $A:=A(G-uv,\betavec)$ as in Theorem~\ref{mainthm}.
 and let $\Y$ be a Gaussian  random vector with density
$(2\pi)^{-n/2} \abs{A}^{1/2} (\varLambda n)^{n/2}
   e^{-\frac12\varLambda n\sum_{j=1}^n \yvec\Trans6 A\yvec}$.
 We will apply Theorem~\ref{mainthm} with $p=1$. 
 First note that
  \[
     \frac{M(G-uv,\dvec-\evec)}{M(G-uv,\dvec)}
     = e^{\beta_u+\beta_v} = \frac{\lambda_{uv}}{1-\lambda_{uv}}.
 \]
 and that the two determinants are the same.
 Next, define
 \begin{align*}
   R'(\thetavec) &= i\(\deltavec+(1-\lambda_{uv})\evec\)\trans \thetavec
            + \sum_{\ell=3}^{\ell_0} \sum_{jk\in G-uv}c_\ell(\ljk)(\theta_j+\theta_k)^\ell \\[-1ex]
   R''(\thetavec) &= i\(\deltavec-\lambda_{uv}\evec\)\trans \thetavec 
            + \sum_{\ell=3}^{\ell_0} \sum_{jk\in G-uv}c_\ell(\ljk)(\theta_j+\theta_k)^\ell.
 \end{align*}
 Then Theorem~\ref{mainthm} gives
  \[
      \frac{\Prob(uv\in H)}{1-\Prob(uv\in H)} = \frac{\lambda_{uv}}{1-\lambda_{uv}}
      \exp\biggl(\, \sum_{r=1}^{r_0} \mu_r(R'(\Y)) -  \sum_{r=1}^{r_0} \mu_r(R''(\Y))
                   + O(n^{-1})\biggr).
 \]
 Note that $R'(\thetavec)$ and $R''(\thetavec)$ differ only by a small change
 in the linear term.  Applying Theorem~\ref{cumdiff} we have $\mu=0$
 since the two distributions are the same.  Also,
 $C_\ell=O(\varLambda^{-1/2}n^{-3/2})$ and $\mu_\ell=0$ for $\ell\geq 3$,
 and $C_2=\mu_2=0$.
 Finally, $\mu_1n^2C_1=O(\varLambda^{-1/2}n^{-1/2})$ and
 $\mu'_1n^3C_1^2 = O\( \varLambda^{-1}n^{-1}(\onenorm\deltavec+1)\)$.
 Considering all the possibilities, we have
 \[
      \frac{\Prob(uv\in H)}{1-\Prob(uv\in H)} = \frac{\lambda_{uv}}{1-\lambda_{uv}}
      \exp\( O(\varLambda^{-1}n^{-1}(\onenorm\deltavec+1))\),
 \]
 which proves the lemma.
\end{proof}
Clearly, Lemma~\ref{edgeprob} can be extended to larger subgraphs.

\begin{appendices}

\nicebreak
\section{Matrices and graphs}\label{AppendixA}

Here we will collect some  linear algebra and graph theory that
we need for our proofs.

\begin{thm}\label{matrixthm1}
  Let $M$ be an $n\times n$ real positive-definite symmetric
  matrix with minimum eigenvalue $\mumin>0$.
 Let $D$ be a real diagonal matrix with entries in $[\dmin,\dmax]$,
 where $\dmin>0$ and such that $\maxnorm{M-D} \leq r\dmin/n$
 for some~$r$.
 Then the following are true.
\begin{itemize}[topsep=1ex]
   \item[(a)]  $\displaystyle
    \infnorm{M^{-1}}\leq \frac{r\dmax+\mumin}{\mumin\dmin}$
    ~and~
     $\displaystyle\maxnorm{M^{-1}-D^{-1}}\leq 
      \frac{r(r\dmax+\mumin)}{\mumin\dmin n}$.
\end{itemize}
  Furthermore, there exists a real matrix $T$ such that
     $T\Trans4 M T=I$ and $T$ can be chosen such that
     the following hold.
\begin{itemize}[topsep=1ex,itemsep=0.2ex]
     \item[(b)]
        $\displaystyle\onenorm{T},\infnorm{T}
        \leq \frac{r\dmax^{1/2}+\mumin^{1/2}}{\mumin^{1/2}\dmin^{1/2}}$.
     \item[(c)]  $\displaystyle \onenorm{T^{-1}}, \infnorm{T^{-1}}
        \leq \frac{(r+1)(r\dmax^{1/2}+\mumin^{1/2})\dmax^{1/2}}
                     {\mumin^{1/2}}$.
     \item[(d)]  $\displaystyle \maxnorm{T-D^{-1/2}} 
        \leq \frac{r((r+1)\mumin^{1/2}+2r\dmax^{1/2})}
                      {2\mumin^{1/2}\,\dmin^{1/2}\,n}$.
     \item[(e)] $\displaystyle \maxnorm{T^{-1}-D^{1/2}} 
        \leq \frac{r((r+3)\mumin^{1/2}\dmax^{1/2}+4r\dmax)\dmin^{1/2}}{2\mumin^{1/2}\,n}$.
\end{itemize}
\end{thm}
\begin{proof}
This is a special case of \cite[Lem.~4.9]{mother}.
\end{proof}

% % % % % % % % % % % % % % % % % % % % % % % % % % % %

For an $n\times n$ real matrix $M$ and $J \subseteq [n]$,  denote by $M_J$
the matrix whose $j$-th columns are equal to
$\evec_j$ if $j\notin J$ and $\fvec_j$ otherwise,
where $\evec_1,\ldots,\evec_n$ is the standard basis and
$\fvec_1,\ldots,\fvec_n$ are the columns of~$M$.
In particular, $M_{\emptyset} = I$ and $M_{[n]} = M$. 

\begin{lemma}\label{l:Mtrans}
Suppose the matrix $M$ satisfies
\begin{equation}\label{M_property}
    \begin{aligned}
    	& \text{off-diagonal entries of $M$ and $M^{-1}$ are  uniformly $O(n^{-1})$,}\\
    		&\text{diagonal entries of $M$ and $M^{-1}$ are uniformly $\Theta(1)$.} 
    \end{aligned}
\end{equation} 
Then, for sufficiently small $\zeta$, if $\abs{J}\leq\zeta n$ or $\abs{J}\geq (1-\zeta)n$,
$M_J$ also satisfies~\eqref{M_property}.
\end{lemma}
 \begin{proof} 
  Obviously $M_J$ itself has the desired structure.
  Let $D_J$ denote the diagonal matrix whose diagonal entries are those of the
  matrix $M_J$.
  For $\abs{J}\leq\zeta n$, we have for sufficiently small $\zeta$ and any~$\xvec$ that
 	\[
 		\infnorm{M_J \xvec}\geq  \infnorm{D_J \xvec} - \infnorm{D_J- M_J} \infnorm{\xvec} 
 		 = \Theta(1) \infnorm{\xvec} - O(\zeta)\infnorm{\xvec}  = \Theta(1) \infnorm{\xvec}.
 	\]
 	Therefore, we have $\infnorm{M_J^{-1}} = O(1)$ and 
 	\[
 	 	\maxnorm{M_J^{-1} - D_J^{-1}}
 	 	 \leq \infnorm{M_J^{-1}}  \maxnorm{D_J - M_J} \onenorm{D_J^{-1}}= O(n^{-1}). 
 	 \]
  Thus $M_J^{-1}$ exists and has the required structure.
  
  Now consider $\abs{J}\geq(1-\zeta)n$. Let $M=(m_{jk})$.
  Define $\hat M=M\diag(\hat m_1,\ldots,\hat m_n)$, where
  $\hat m_j=m_{jj}^{-1}$ if $j\notin J$ and $\hat m_j=1$ otherwise.
  Note that the $j$-th column of $\hat M - M_J$ is 0 if $j\in J$ and
  is uniformly $O(n^{-1})$ otherwise.
  Since $\infnorm{\hat M^{-1}}=\Theta(1)$, we have for sufficiently small $\zeta$
  and any $\xvec$ that
  \[
     \infnorm{M_J\xvec} \geq \infnorm{\hat M\xvec} - \infnorm{(\hat M-M_J)\xvec}
     = \Omega(1)\infnorm\xvec -O(\zeta)\infnorm\xvec
     = \Omega(1)\infnorm\xvec.
 \]
 This shows that $\infnorm{M_J^{-1}}=O(1)$ and the rest of the argument is
 the same as in the previous case.
\end{proof}

\begin{lemma}\label{l:edges_sub}
 Let $G$ be a graph  with average degree $d$. Then, 
 for nonnegative integer  $h\leq  d$ there is a subgraph of $G$ with 
 at least $hd/2$ edges and with degrees bounded above by $h$.
\end{lemma}
\begin{proof}
  Let $n = \abs{V(G)}$. Let $H$ be a graph on the same vertex set as $G$ 
  with at least $h(n-1)/2$ edges such that its degrees are bounded above by $h$.
  Such graph $H$ exists as we can either take $h$-regular graph on $n$ vertices 
  or  $h$-regular graph on $n-1$ vertices. Let $H^\sigma$ 
  denote the graph isomorphic to $H$ under permutation $\sigma$ of its vertices. 
  Now take $\X $ to be a uniformly random permutation of $V(H)$. Then, 
  for any possible edge $e$
  \[ 
  	\text{Pr} (e \in E(H^{\X })) = \frac{2(n-2)!}{n!} \abs{E(H)}\geq  \frac{(n-2)! }{n!} h(n-1)  = h/n.
  \]
  Thus, we get that 
  \[
  	\E \,\abs{E(G) \cap E(H^{\X })}   \geq  \abs{E(G)}\, h/n = hd/2.
  \] 
   The last inequality implies the existence of a permutation 
   $\sigma$ such that graph $H^{\sigma}$ has at least $hd/2$ edges 
   in common with~$G$. These edges form the required subgraph of~$G$. 
\end{proof}

The following lemma sets out one important structural property of strongly nonbipartite graphs.
As usual, the length of a walk is the number of its edges.

\begin{thm}\label{oddpaths} Let $G$ be a graph with $n$
vertices such that $q(G)\geq \tauQ n$ for some fixed $\tauQ>0$.
Let $U$ be a nonempty set of vertices.
Define
\[
     L := 1 + 2 \,\lceil \log_{1+\tauQ}(2/\tauQ)\rceil.
\]
Then there at least $\dfrac{\tauQ n\abs{U}}{8L}$
edge-disjoint walks of odd length at most~$L$ that start
and finish in~$U$.
\end{thm}
\begin{proof}
  Let $u:=\abs{U}$. We call a walk \textit{good} if it has odd length at most $L$
  and starts and finishes in $U$. For a spanning subgraph $H$ of $G$ denote 
  \[
  	q_u(H) = \min_{\abs{A}+\abs{B}\geq u} \Dfrac{\xvec_{AB}\trans  Q(H) \xvec_{AB} }
	{\twonorm{\xvec_{AB}}^2} =
  	\min_{\abs{A}+\abs{B}\geq u} \Dfrac{4\abs{E(H[A])} 
	 +4 \abs{E(H[B])} + \abs{\partial_H (A\cup B)}}{\abs{A}+\abs{B}},
  \] 
 where the both minima are with respect to two disjoint sets of vertices $A,B$ such
  that $\abs{A\cup B} = \abs{A}+\abs{B}\geq u$; 
 the vector $\xvec_{AB}$ has components equal to $1$ and $-1$ for positions corresponding to $A$ and $B$, respectively, and other components equal to $0$;
 the matrix $Q(H)$ is the signless Laplacian of $H$;
 and  $\partial_H(A\cup B)$ denotes the set of edges of $H$ with one end in
 $A\cup B$ and the other end not in $A\cup B$. Note that
 \[
 	q_u(G) \geq q(G) \geq \tau_Q n.
 \]

 Suppose that $H$ is a spanning subgraph of $G$ with $q_u(H) \geq \tauQ n/2$.
 We show that there is a good walk in~$H$.
   For $k \geq 0$,
     let $N_k$ denote the set of vertices at distance at most $k$  from $U$ in the graph~$H$.
   Suppose first that $H[N_k]$ is bipartite for $k<(L-1)/2$,
   and let $A_k,B_k$ denote its parts.
    We have that 
     $\abs{A_k\cup B_k} \geq \abs{U}=u$ and $E(H[A_k]) = E(H[B_k]) = \emptyset$.
       Since $q_u(H) \geq \tauQ n/2$, we find that
     \[
     	\abs{\partial_{H} (N_k)} = \abs{\partial_H (A_k\cup B_k) }
     	\geq \dfrac12\tauQ n (\abs{A_k}+\abs{B_k}) = \dfrac12\tauQ n\abs{N_k}. 
     \]
     Therefore, we get that
     \[
     	\abs{N_{k+1}} \geq \abs{N_k} + \dfrac{\abs{\partial_{H} (N_k)} }{n}  
     	 \geq (1+ \tau_Q/2) \abs{N_k}.
     	\]
     	Similarly, since every vertex from $N_1$ is adjacent with at most $u$ vertices from $U$, we find that
    $
     		\abs{N_1} \geq \abs{U} + \tau_q n/2 > \tau_q n/2.
     $
     Therefore,
     \[
     	\abs{N_k} > \dfrac12 (1+ \tau_Q/2)^{k-1}  \tau_q n.
     \]
      Since $\abs{N_k}\leq n$, the last inequality is impossible when
       $k = (L-1)/2$, which shows that $H[N_{(L-1)/2}]$ is
       not bipartite.
       
   Since $H[N_{(L-1)/2}]$ has an odd cycle that is reachable from~$U$,
   there is an odd-length walk from~$U$ to~$U$.
   Consider a shortest such walk $P$ and suppose its length is
   at least $L+2$. By the definition of $H[N_{(L-1)/2}]$, there
   is a path of length at most $(L-1)/2$ from $U$ to a
   central vertex of~$P$, which implies there is an odd-length walk from~$U$ to~$U$
   that is shorter than~$P$. This shows that $P$ has length 
   at most~$L$. 
 
 Starting with $H := G$, we can find good walks in~$H$
 and remove them from $H$ until $q_u(H) < \tau_Q n/2 $.
 If we found less than 
$ \dfrac{\tauQ n u}{8L}$  good walks then we deleted at most 
  $ \dfrac{\tauQ nu}{8}$ edges. However, for any two disjoint $A,B$
  with $\abs{A}+\abs{B} \geq u$, we have
  \begin{align*}
  	 \Dfrac{4\abs{E(H[A])} +4 \abs{E(H[B])} + \abs{\partial_H (A\cup B)}}{\abs{A}+\abs{B}}
  	& \geq 
  	  \Dfrac{4\abs{E(G[A])} +4 \abs{E(G[B])} + \abs{\partial_G (A\cup B)} -    \tauQ nu/2}{\abs{A}+\abs{B}} 
  	  \\
  	  &\geq  q_u(G) -     \tauQ n/2 \geq  \tauQ n/2.
  \end{align*}
  Thus, $q_u(H) \geq \tauQ n/2$ which would contradict
  our stopping criterion.
\end{proof} 
% % % % % % % % % % % % % % % % % % % % % % % % % % % % % % %

\begin{thm}\label{shortpaths}
Let $G$ be a graph with $n$ vertices, minimum degree at least~$4r$,
and with $h(G )\geq m$.  Let $U_1,U_2$ be disjoint sets of vertices with
$\abs{U_1},\abs{U_2}\geq m$ and $\abs{U_1}+\abs{U_2}\geq n-r$.
Then there are at least $rm$ mutually
edge-disjoint paths of length at most~3
with one end in $U_1$ and the other end in $U_2$.
\end{thm}
\begin{proof}
  Define $R := V(G) \setminus (U_1\cup U_2)$.
  
  Suppose first that $\abs{U_1}\leq 2r$.  By the minimum
  degree condition and the fact that $\abs R\leq r$, each vertex
  of $U_1$ is adjacent at least $r$ vertices of~$U_2$.
  Since $\abs{U_1}\geq m$ by assumption, there are
  at least $rm$ edges from $U_1$ to~$U_2$.  Similarly if
  $\abs{U_2}\leq 2r$.
  
  The remaining case is when $\abs{U_1},\abs{U_2}>2r$.
  Divide $R$ into two parts: $R_1$ are those vertices at least
  half of whose neighbours in $U_1\cup U_2$ lie in~$U_1$,
  and $R_2$ are the rest.
  Write $E(U,W)$ for the set of edges between $U$ and~$W$.
  Consider three sets of edge-disjoint paths from $U_1$ to~$U_2$:
  \begin{itemize}[noitemsep,topsep=1ex]
    \item[(a)] $P_1$ consists of single edges from $U_1$ to $U_2$.
       Obviously $\abs{P_1}=\abs{E(U_1,U_2)}$.
    \item[(b)] $P_2$ consists of paths of two edges going from
      $U_1$ to $U_2$ via $R$.  By the definitions of $R_1$ and $R_2$,
      each edge in $E(U_1,R_2)\cup E(R_1,U_2)$ can be extended
      by an edge while remaining edge-disjoint, so we can take
      $\abs{P_2} = \abs{E(U_1,R_2)\cup E(R_1,U_2)}$.
    \item[(c)] $P_3$ consists of paths of three edges going from
      $U_1$ to $R_1$ to $R_2$ to $U_2$.  By the minimum degree
      condition, the definition of $R_1$, and the fact that $\abs{R_2}\leq r$,
      each vertex in $R_1$ is adjacent to more vertices in $U_1$ than in
      $R_2$,  and symmetrically for a vertex in~$R_2$.  Thus each
      edge in $E(R_1,R_2)$ can be extended by an edge at each end
      to make $\abs{P_3}=\abs{E(R_1,R_2)}$ edge-disjoint paths.
   \end{itemize}
   Now note that $P_1$ is edge-disjoint from both $P_2$ and $P_3$.
   Also $\abs{P_1\cup P_2}+\abs{P_1\cup P_3}
     \geq \abs{E(U_1\cup R_1,U_2\cup R_2)}
     \geq 2rm$, where the last inequality is from applying $h(G )\geq m$
     to the smaller of $U_1\cup R_1$ and $U_2\cup R_2$.
   Consequently, at least of one of $P_1\cup P_2$ and $P_1\cup P_3$
    has at least $rm$ edge-disjoint paths.
\end{proof}

\nicebreak
\section{Integral bound for periodic functions} \label{s:general_boxing}\label{AppendixB}

In this appendix we will prove a general bound on the integral of a complex
periodic function on~$\Reals^n$.  In order to facilitate future applications, this appendix is
notationally independent of the rest of the paper.

Our setting is that we have a periodic function $H:\Reals^n\to\Complexes$ whose
basic cell is a parallelepiped $B$ centered at the origin.  We also have a small cuboid
$B_0\subseteq B$ centered at the origin (not necessarily aligned with~$B$), and our aim is
to bound $\Abs{\int_{B-B_0} H(\xvec)\,d\xvec}$.
The reason for excluding $B_0$ is that in our applications $H$ is very large
in~$B_0$ and we need that $\int_{B-B_0} H(\xvec)\,d\xvec$ is small in comparison
with $\int_{B_0} H(\xvec)\,d\xvec$.
In prior applications (see~\cite{mother} for a list of about 30 examples),
$\int_{B-B_0} \abs{H(\xvec)}\,d\xvec$ was small enough, though proving it
wasn't always trivial.
For the present application, $\int_{B-B_0} \abs{H(\xvec)}\,d\xvec$ is much too
large so we need another approach.

Our method assumes that we can bound lower-dimensional integrals
of $H$ over domains of the form $\{x_1\}\times\cdots\times\{x_m\}
\times I_{m+1}\times\cdots\times I_n$, where $I_{m+1},\ldots,I_n$ are
intervals with $[-\frac13\rho,\frac13\rho]\subseteq I_j\subseteq [-\frac43\rho,\frac43\rho]$ for each~$j$.
We combine such domains to create disjoint regions which (in a sense that
involves the periodicity of~$H$) tile all of $B-B_0$ except for a part $\Babs$
in which we assume $H$ to be small enough to use the integral of~$\abs H$.

\smallskip

Let us proceed to the formal setup of our general bound on the integral of a complex periodic function. 
As usual, let  $U_n(t)=\{\xvec=(x_1,\ldots,x_n)\st  \abs{x_j}\leq t \text{ for }1\leq j\leq n\}$.
  For $J\subseteq[n]$    denote 
     \[
        	U_{J}(t) := \{(x_1,\ldots,x_n)\in U_n(t)\st x_j =0 \text{ if $j\in J$} \}.
    \]
Let vectors $\fvec_1, \ldots,\fvec_n$ form a basis in $\Reals^n$.
  Consider the lattice $\calL \subseteq \Reals^n$ consisting of the points 
    $k_1 \fvec_1 + \cdots + k_n \fvec_n$ where  $k_1,\ldots, k_n$ are integers.
   Let 
   \[
      B := \bigl\{t_1 \fvec_1 + \cdots + t_n \fvec_n \st  t_1,\ldots,t_n \in (-\tfrac12,\tfrac12\,]\bigr\}.
   \] 
   For any $\xvec\in \Reals^n$ let   $\tau_B(\xvec)$ be the unique point in $B$ such that
    $\xvec - \tau_B(\xvec) \in \calL $. For a   set $S \subseteq \Reals^n$,  we define 
   $\tau_B(S) = \bigcup_{\xvec\in S}\tau_B(\xvec)$.
Let 
\[
 \text{
$M$ be the matrix whose columns are vectors $\fvec_j$, $j=1,\ldots,n$.} 
\]
 Let $\evec_1, \ldots,\evec_n$ denote  the standard basis in $\Reals^n$.

\begin{thm}\label{MishaMagic}
  Assume that $M$ satisfies~\eqref{M_property}.
  Assume that $\zeta,\rho>0$ are sufficiently small.
 Let $H:\Reals^n \rightarrow \Complexes$ 
 be an absolutely integrable $\calL $-periodic function, that is,  $H(\xvec) = H(\pi_B(\xvec))$.
 Let $\Babs$ be such that  
 \begin{equation}\label{S_property}
 \Bigl\{(x_1,\ldots,x_n) \in B \st 
    \sum_{j =1}^n \abs{x_j} \One _{\abs{x_j}>\rho/3} > \dfrac13\zeta n\rho\Bigr\}
    \subseteq\Babs\subseteq B
\end{equation}
and $\Babs\cap S$ is measurable within $S$ for any affine subspace $S$
of~$\Reals^n$.
Suppose that there is a function $L:[n]\to\Reals_+$ such that
    \[
    		\biggl|\int_{W} H(\wvec + \yvec)\, d \yvec \biggr| \leq L(\abs J) 
    \]
 for every $J\subseteq [n]$ with $1\leq \abs{J} \leq n$, 
 every axis-aligned cuboid $W$ such that
 \begin{equation}\label{W_property}
 U_J(\tfrac13\rho) \subseteq 	W   \subseteq U_J(\tfrac43\rho).
 \end{equation}
  and every
\[  \wvec = (w_1, w_2,\ldots, w_n)\trans \in (B +  U_n(\tfrac13\rho)) \cap \Span\{\evec_j\st j\in J\}\] 
such that 
 \begin{itemize}
 \item[(i)] $\abs{w_j} > \Dfrac{\rho\, m_{jj}}{9 \max_{j\in[n]} m_{jj}}$  for all $j\in J$, where $m_{jj}$ are the diagonal entries of $M$,
  \item[(ii)]  $\tau_B(\wvec + U_J(2\rho)) \cap (B-\Babs ) \neq \emptyset$.
  \end{itemize}
    Then,
    \[
    		\biggl|\int_{B - U_n(\rho)} H(\xvec)\, d \xvec \biggr| \leq  
    		 	\int_{\Babs} \abs{H(\xvec)}\, d \xvec + \sum_{j=1}^{n}\, (n\infnorm{M})^j L(j).
    \]
 \end{thm}

Throughout this appendix, when we say that ``$\kappa$ is sufficiently small'', we
 mean that $\kappa= \kappa(n)$ and $\limsup_n \kappa(n)$ is bounded by a sufficiently small constant
 that depends only on the implicit constants in the bounds   $O(n^{-1})$  and $\Theta(1)$ for  elements of the matrices  $M$ and $M^{-1}$.
 Whenever we use  the notations $O(\,)$ and $\Theta(\,)$,  
 the implicit  constants also depend only on the implicit constants in the bounds    of \eqref{M_property}.

The proof of Theorem \ref{MishaMagic} relies on the six technical lemmas given in the next subsection. 

\nicebreak
\subsection{Six lemmas}\label{AppendixB1}

For a matrix $M$ and $J\subseteq [n]$, recall the definition of $M_J$ from
Appendix~\ref{AppendixA}.

\begin{lemma}\label{l:MJdet}
For any $M,J$, the absolute value of
$\abs{M_J}$ is at most $\infnorm{M}^{\abs{J}}$.
\end{lemma}  
 \begin{proof}
    Since $M_J$ is block triangular, $\abs{M_J}$ equals the determinant
    of a $\abs{J}\times\abs{J}$ submatrix of $M$.
    By Gershgorin's theorem, all the eigenvalues of the submatrix
    have absolute value at most $\infnorm{M}$. The bound follows.
 \end{proof}
 
   \begin{lemma}\label{l:pi}
        For a set $S\in \Reals^n$ define
        $\diam(S) = \sup_{\xvec,\yvec \in S} \infnorm{\xvec-\yvec}.$
      \begin{itemize}\itemsep=0pt
      			\item[(a)]  If  $\diam(S) < \dfrac{1}{\infnorm{M^{-1}}}$ then 
      			$\tau_B$ maps $S$ into $B$ injectively.
      		        \item[(b)]  If  $\diam(S) < \dfrac{1}{2\infnorm{M^{-1}}}$ 
		        and  $S \cap B \cap \Span\{\fvec_j \st j \in J\} \neq  \emptyset$ for some $J\in [n]$  then 
      		        for any $\xvec \in S$, we have 
      		        \[
      		        	\xvec  - \tau_B(\xvec) = \sum_{j\in J} k_j \fvec_j, \qquad \text{ where } \qquad k_j \in\{-1,0,1\}.
      		        \]
       	\item[(c)] Let $S_1,\ldots,  S_m \subseteq \Reals^n$ be   sets such that 
    $\diam(S_k) < \dfrac{1}{2 \infnorm{M^{-1}}}$
for $k =1,\ldots,m$.  Then, 
\[
	\bigcap_{k=1}^m \tau_B(S_k) 
	=\tau_B \biggl(\,\bigcap_{k=1}^m\, (\xivec_k + S_k)\biggr)
\]	
for some $\xivec_1,\ldots, \xivec_m\in \calL $, which are determined uniquely up 
to uniform translation. 
   \end{itemize}
   \end{lemma}
 \begin{proof}
 	Note that $\infnorm{\xivec} < \frac{1}{\infnorm{M^{-1}}}$
 	implies $\infnorm{M^{-1}\xivec} <  1$. 
 	If, in addition, $\xivec \in \calL $, which is 
 	equivalent to saying that  all components of $M^{-1}\xivec$  are integers, then we can conclude 
 	$\xivec =(0,\ldots,0)$. 
 	Part (a) follows immediately.
 	
 	For (b) we can find some  $\sum_{j\in J} t_j \fvec_j \in S$ that $t_j \in (-\frac12,\frac12]$ for $j\in J$.   
 	Observe  that   
 	\[
 	 \infNorm{M^{-1}(\xvec  - \sum_{j\in J} t_j \fvec_j))} \leq 
 	  \infnorm{M^{-1}} \diam(S)  <   \dfrac12.
 	 \] 
 	Therefore, coefficients of $\xvec$ in the basis $\{\fvec_j\}$ 
 	 lie in $(-1,1)$ for $j\in J$ and lie in $(-\frac12, \frac12)$ for $j \notin J$, which implies (b).  
 
 	 We proceed to (c). Consider any point $\avec$  from  
 	 $\bigcap_{k=1}^m \tau_B(S_k) $.  Then we can find $\xivec_k \in \calL $
 	 and $\xvec_k \in S_k$ such that  $\avec = \xvec_k + \xivec_k $ for all $k$.  
 	 It is sufficient to prove that  $ \xivec_j  -  \xivec_k $ does not
	 depend on $\avec$ for all $j$ and $k$. Indeed, 
 	 if $\avec' = \xvec_j' + \xivec_j'$ and  $\avec' = \xvec_k' + \xivec_k'$  for another point 
 	 $\avec'\in \bigcap_{j=1}^m \tau_B(S_j) $ then 
 	  \[
 	  		\infnorm{\xivec_j - \xivec_k  -  \xivec_j' + \xivec_k'}
 	  		\leq  \infnorm{\xvec_j - \xvec_k}  + \infnorm{\xvec_j' - \xvec_k'}
 	  		\leq  \frac{1}{\infnorm{M^{-1}}}.
 	  \]
 	  Since $\xivec_j - \xivec_k  -  \xivec_j' + \xivec_k' \in \calL $, we
 	   conclude $\xivec_j - \xivec_k   =   \xivec_j' - \xivec_k'$.  Part (c) follows by part~(a).
 \end{proof}

Let $\Babs$ be the region defined in \eqref{S_property}. 
Note that for any $J\subseteq [n]$ and $\xvec \in B-\Babs$, we have  
\begin{equation}\label{crit_bound}
\sum_{j\in J} \,\abs{x_j} \leq  \sum_{j \in J}\, \(\dfrac13\rho+\abs{x_j}\One _{\abs{x_j}>\rho/3}\) 
\leq  \dfrac13(\abs{J} + \zeta n) \rho.
\end{equation}

     For $J\subseteq [n]$ with $\abs{J} \leq \zeta n$,
   let 
   \[
    \rho_J^+=  \rho \Bigl(1 - \Dfrac{\abs{J}}{2\zeta n}\Bigr) 
       \qquad \text{ and  } \qquad 
    \rho_J^-=  \Dfrac{\rho}{6\max_{j\in [n]} m_{jj}}  \Bigl(1  + \Dfrac{\abs{J}}{2\zeta n}\Bigr)  
   \]
   where $m_{jj}$ are the diagonal entries of $M$.  
        Define  $\varOmega_J$ as the set of all points $\xvec$ satisfying  the following properties:
      \begin{itemize}\itemsep=0pt
   		\item[P1.] $\xvec = (x_1,\ldots, x_n)= \sum_{j \in J }t_j \fvec_j  \in B$, 
		          that is, $t_j \in (-\frac12,\frac12]$ for $j\in J$. Define $t_j=0$
   		for $j\notin J$.
  		\item[P2.] $\tau_B(\xvec + U_J(\rho_J^+))$ contains a point from  $B-\Babs $. 
   	   \item [P3.] At least  one of the following  holds:
   \begin{itemize}
   	\item[(i)] $\min_{j\in J} \abs{x_j} > \frac13\rho$ 
   	\item [(ii)]  $\min_{j \in J} \abs{t_j} >  \rho_J^-$ and  $\infnorm{\xvec} > \frac43\rho$.
   \end{itemize} 	
  \end{itemize}

   \begin{lemma}\label{l:oneU}
   	 Assume $M$ and $M^{-1}$ satisfy  \eqref{M_property}. Then 
   	 for any sufficiently small $\zeta, \rho$  the following  holds. 
       Take  some $J \subseteq [n]$ with  $\abs{J}\leq \zeta n$ and $\xvec=  \sum_{j \in J }t_j \fvec_j$. 
 	  \begin{itemize}\itemsep=0pt
 	       \item[(a)]  If  $\xvec$ satisfies P1, P2  for  $J$  then  $\sum_{j \in J} \abs{t_j} =  O(n) \zeta  \rho$
	       and $\abs{x_j}=O(\zeta\rho)$ for $j\notin J$.
  	       \item[(b)] If $\xvec \in \varOmega_J$ we have $\min_{j \in J} \abs{t_j} >  \rho_J^-$. 
  	        \item[(c)] If $\yvec= (y_1,\ldots,y_n) \in \tau_B( \varOmega_J + U_J(\rho_J^+))$  
	            then $ \abs{y_j} > \Dfrac{\rho\,m_{jj}}{9\max_{j\in [n]}m_{jj}}$ for all $j \in J$.
  	        \end{itemize}
  \end{lemma}
  \begin{proof}
   	By P2, we can find a point $\avec$ such that
   	\[
\avec = (a_1,\ldots,a_n)\trans  = \sum_{j=1}^n \medtilde{t}_j \fvec_j \in \tau_B(\xvec + U_J(\rho_J^+)) \cap (B-\Babs ).
\]  
  Note that $\Abs{\abs{t_j}-\abs{\medtilde t_j}}\leq \rho_J^+$.
  Since $(\medtilde t_1,\ldots,\medtilde t_n)\trans =M^{-1}\avec$,
  property P1, assumption \eqref{M_property} for $M^{-1}$,
  and the bound \eqref{crit_bound} for sets $J$ and $[n]$, imply  that
 \begin{align*}
  \sum_{j\in J} \abs{t_j}  &\leq    
  \sum_{j\in J} ( \abs{\medtilde{t}_j}   +   \rho_J^+ )  
  \\
  &\leq \sum_{j\in J} \( \Theta(1)\abs{a_j}+O(n^{-1})\onenorm{\avec} 
    + \rho_J^+\) = O(n)\zeta\rho.
\end{align*}
The bound on $\abs{x_j}$ for $j\notin J$ follows from Lemma~\ref{l:Mtrans}.
This completes the proof of (a).

For (b) note that if P3(ii) holds then we are finished. Otherwise, we may assume P3(i)  since $\xvec \in \varOmega_J$.
Note that $\xvec=M(t_1,\ldots,t_n)\trans $.
 If $\abs{t_j} \leq \rho_J^{-}$ for some $j\in J$  then, by assumption \eqref{M_property} for $M$ and  part (a),  we have
\[
  	\abs{x_j} \leq \abs{m_{jj} t_j}  +  O(n^{-1}) \sum_{j\in J}\, \abs{t_j} \leq  m_{jj} \rho_J^{-}+  O(1)\zeta \rho < \tfrac13\rho
\]
 for sufficiently small $\zeta$.  This  contradicts P3(i)  and thus proves (b).

 For (c) assume  $\xvec\in \varOmega_J$ and $\yvec \in \tau_B(\xvec +U_J(\rho_J^+))$, we have that 
$\yvec \in \xvec +\xivec +U_J(\rho_J^+)$ for some $\xivec  =  (\xi_1,\ldots,\xi_n) \in \calL $.
Then we have $y_j = x_j +  \xi_j$ for $j \in J$.
Note that
\[
      \diam\(\xvec +U_J(\rho_J^+)\) \leq 2\rho_J^+ 
      \leq 2\rho \leq \frac{1}{2\infnorm{M^{-1}}}
\]
for sufficiently small $\rho$.
Using Lemma \ref{l:pi}(b), we find that 
$\xivec = \sum_{j\in J} k_j \fvec_j $, where $k_j \in \{-1,0,1\}$.
Define $k_j=0$ for $j\notin J$.
For $j\in J$, if $k_j \neq 0$ then  
$\abs{t_j + k_j }\geq \frac12>\rho_J^-$ by P1, while if $k_j=0$,
$\abs{t_j + k_j }>\rho_J^-$ by part~(b).
Observe also that   
\[
      \sum_{j\in [n]}\,\abs{t_j + k_j } \leq \sum_{j\in J}\, ( \abs{t_j}  + O(\rho))  = O(n)\zeta \rho,
\]
where we have used the fact that, when $\rho$ is small, $k_j\ne 0$ implies that
$t_j=\pm\frac12+O(\rho)$.
Using assumption \eqref{M_property} and $(\xvec+\xivec)\trans  = M(t_1+k_1,\ldots,t_n+k_n)\trans $,
 we can bound 
\begin{align*}
	\abs{y_j}  = \abs{x_j+\xi_j}
	  &=\biggl| m_{jj} (t_j+k_j)  
	-  O(n^{-1}) \sum_{j\in [n]}	\abs{t_j + k_j }\biggr| \\
	&\geq
	  m_{jj}  \rho_J^-  -  O(\zeta\rho) > 
	\Dfrac{\rho\, m_{jj}}{9\max_{j \in [n]} m_{jj} }. 
\end{align*}	
  This completes the proof.
   \end{proof}

 We next derive  a useful property of the intersection of two regions of the form 
  $\xvec + U_J(\rho_J^+)$ in the  lemma stated below.
  
  \begin{lemma}\label{l:inter}
  	 Assume $M$ and $M^{-1}$ satisfy  \eqref{M_property}. 
    Then, for sufficiently small $\zeta, \rho$  the following  holds.
      Suppose $\abs{J}, \abs{K}\leq\zeta n$ and $\tau_B(\xvec + U_J(\rho_J^+))\cap \tau_B(\yvec + U_K(\rho_K^+)) \neq \emptyset$ 
      for some  $\xvec \in \varOmega_J$ and $\yvec \in \varOmega_K.$    	 
  	 Then there is $\zvec\in  \varOmega_{J\cap K}$  such that 
  	 \[
  	 		\tau_B(\xvec + U_J(\rho_J^+)) \subseteq \tau_B(\zvec + U_{J\cap K}(\rho_{J\cap K}^+)). 
  	 \]
  \end{lemma}
  
  \begin{proof}
  For $J=K$, the lemma holds trivially with $\zvec=\xvec$, so assume
  that $J\ne K$.
  
  	Take $\xivec  =(\xi_1,\ldots,\xi_n)\in \calL $ such that 
  	$\(\xvec + U_J(\rho_J^+) \)  \cap  \(\xivec +\yvec + U_K(\rho_K^+)\) \neq \emptyset$ 
  	which is actually unique by Lemma \ref{l:pi}(c).
  	 Let $\xvec =\sum_{j\in J} t_j \fvec_j $, $\yvec = \sum_{j\in K} t_j'\fvec_j$.
  	Observing  $\xivec +\yvec \in \xvec + U_J(\rho_J^+)+ U_K(\rho_K^+)$
  	and using Lemma \ref{l:pi}(b), we find that  $\xivec = \sum_{j\in J} k_j \fvec_j$, where
  	$k_j  \in \{-1,0,1\}$ for all $j \in J$.  
  	Similarly, $\xivec\in\Span\(\{\fvec_j\st j\in K\}\)$
  	and so $\xivec\in\Span\(\{\fvec_j\st j\in J\cap K\}\)$.
       Observe that 
       \begin{equation}\label{xxiy}
       	\xvec - \xivec - \yvec  
       	 \in U_{J}(\rho_J^+) + U_{K}(\rho_K^+) \subseteq U_{J\cap K}(\rho_J^++\rho_K^+).  
       \end{equation}
       
  Then,
       using \eqref{M_property} for $M^{-1}$, we get that $t_j = O(\rho)$ for $j\in J- K$
       and $t_j' = O(\rho)$ for $j\in K- J$.
       By Lemma~\ref{l:Mtrans}, we can expand
       \[ 
       \sum_{j\in J-K} t_j \fvec_j = \vvec 	+ \wvec
       \qquad\text{ and }\qquad 
          	\sum_{j\in K-J} t_j' \fvec_j = \vvec'
              	+ \wvec',
       \]
       where $\vvec,\vvec'\in\Span\(\{\fvec_j\st j\in J\cap K\}\)$
       and $\wvec = (w_1,\ldots,w_n)\trans $, $\wvec' = (w_1',\ldots,w_n')\trans $
           are  such that $w_j = w_j' =0$ for all $j\in J \cap K$.
Since $\xvec-\xivec-\yvec\in\Span(\{ \evec_j\st j\notin J\cap K\})$
by~\eqref{xxiy},
\begin{equation}\label{xyww}
      \xvec-\xivec-\yvec = \vvec + \wvec-\vvec' - \wvec'
        + \sum_{j\in J\cap K} (t_j-t'_j-k_j)\fvec_j 
        = \wvec - \wvec'.
\end{equation}
Using Lemma \ref{l:Mtrans} and assumption \eqref{M_property} for $M^{-1}$, we obtain that
      \[
      			 \abs{w_j} =  O( n^{-1}) \Bigl\| \sum_{j\in J-K} t_j \fvec_j\Bigr\|_1 
      					  =  O( n^{-1}) \rho\, \abs{J-K} \qquad \text{for  $j \notin J$.}
      \]
      Similarly, $\abs{w_j'} = O( n^{-1}) \rho\, \abs{K-J}$ for $j \notin K$.
      From~\eqref{xxiy} and~\eqref{xyww},
      $\wvec  - \wvec'  \in U_J(\rho_J^+) + U_K(\rho_K^+)$.  Therefore,
       \begin{equation*}
      			 \abs{w_j} \leq    \rho_K^+  + \abs{w_j'}
			 =\rho_K^+ + O( n^{-1}) \rho\, \abs{K-J}  \qquad \text{for $j \in J-K$.}
       \end{equation*}
       
      Define $\zvec = (z_1,\ldots,z_n)\trans  =\tau_B \(\vvec+\sum_{j\in J\cap K} t_j \fvec_j \).$ 
      We want to show that $\zvec \in \varOmega_{J\cap K}$. Property P1 holds immediately; that is,
      $\zvec = \sum_{j\in J\cap K} \hat{t}_j\fvec_j$
      for some $\hat t_j\in(-\frac12,\frac12]$, $j\in J\cap K$.
           Note  that 
             $\zvec = \xvec + \xivec' - \wvec$
            for some $\xivec' \in \Span\{\fvec_j\st j \in J\cap K\}$. Similarly to before (using $\xvec,\zvec \in B$ and Lemma \ref{l:pi}(b)) we get that 
            $\xivec'=\sum_{j\in J\cap K} k_j'\fvec_j$
            with $k_j' \in\{-1,0,1\}$.

 Next we show that $\zvec$ satisfies P2. 
      For sufficiently small $\zeta$,  we have  
   \begin{align*}
   		\rho_{J\cap K}^+ - \rho_{J}^+ = \rho \tfrac{\abs{J-K}}{2\zeta n}
        &\geq \max_{j\notin J} \,\abs{w_j},\\
        \rho_{J\cap K}^+ = \rho_{K}^+ +\rho \tfrac{\abs{K-J}}{2\zeta n}
                &\geq \max_{j\in J-K} \abs{w_j}.
   \end{align*}
  Therefore, 
  \[
  	\tau_B \(\xvec + U_J(\rho_J^+) \)\subseteq 
	\tau_B\( \xvec - \wvec + U_{J\cap K} (\rho_{J\cap K}^+)\) 
	= \tau_B\( \zvec + U_{J\cap K} (\rho_{J\cap K}^+)\),
  \]
  which establishes the final part of the lemma.
  In particular, we find that  property P2  for~$\zvec$ and the set $J\cap K$ follows from P2 for  $\xvec$ and the set $J$ (which holds since  $\xvec \in \varOmega_J$). 
 
 It remains to show that $\zvec$ satisfies P3.
  Using  Lemma~\ref{l:oneU}(b) for $\xvec$ we get $\min_{j\in J}\abs{t_j} > \rho_J^-$.
 Take any  $j\in J \cap K$ .
If   $k_j'\neq  0$ then $\abs{\hat{t}_j} = \frac12-O(\rho)> \rho_{J\cap K}^-$.   Otherwise,  recalling the bounds for the components of $\wvec$, we get 
\begin{align*}
	\abs{\hat{t}_j} &\geq  \abs{t_j}  -   O(n^{-1}) \rho\abs{J-K} \geq  \rho_J^-  - O(n^{-1}) \rho\abs{J-K}   \\&=
	\rho_{J\cap K}^{-} + \dfrac{\rho \abs{J-K}}{12 \zeta \max_{j\in [n]} \abs{m_{jj}} n }   -  O(n^{-1}) \rho\abs{J-K} 
	 > \rho_{J\cap K}^{-}
\end{align*}
for sufficiently small $\zeta$.
If $\xivec\ne 0$, then $\abs{\hat{t}_j}=\frac12-O(\rho)$  for some $j\in J-K$ and so, since $\infnorm{M^{-1}}=O(1)$, $\infnorm{\zvec} = \Omega(1)>\frac43\rho$,  which gives us the property P3(ii).

Next, consider the case  of trivial $\xivec'$.
Recall that $\xvec\in\varOmega_J$.
  If $\min_{j\in J} \abs{x_j} > \frac13\rho$  then 
  we have P3(i) already since  $z_j = x_j$ for all $j \in J\cap K$.
Therefore, we may assume  $\infnorm{\xvec} > \frac43\rho$.
Using Lemma~\ref{l:oneU}(a) for $\zvec$ we get  
 $\sum_{j\in J \cap K} \abs{\hat{t}_j} = O(n)\zeta \rho$ and 
    then, for any $j\notin J \cap K$, we can bound 
   \[
 	 \abs{x_j} \leq \abs{z_j} + \rho_{J\cap K}^+ \leq 
 	    O(n^{-1}) \sum_{j\in J\cap K} \abs{\hat{t}_j}  + \rho < \tfrac43\rho
 	    < \infnorm\xvec.
\]
   This  implies
  \[
   \infnorm{\zvec} \geq \max_{j\in J\cap K} \abs{z_j}  =  \max_{j\in J\cap K} \abs{x_j} = \infnorm{\xvec} > \tfrac43\rho.
 \]
 Thus, we get P3(ii) for $\zvec$ which completes the proof.
  \end{proof}

   We are also  interested in intersections of several regions 
  	 of the form    $\xvec+ U_J(\rho_J^+)$.
%%%%%%%%%%%%%%%%%%%%%%%%%%%%%%%%%%%%%%%%	 
	 
 \begin{lemma}\label{l:manyU} 
 Assume $M$ and $M^{-1}$ satisfy \eqref{M_property} and let $J_1,\ldots,J_m\subseteq[n]$
 be such that $\abs{J_k}\leq\zeta n$ for each $k$ and $\bigcap_{k=1}^m J_k=\emptyset$,
 where $\zeta$ is sufficiently small.  Define $J=\bigcup_{k=1}^m J_k$.
Suppose that $\xvec_k\in\varOmega_{J_k}$ for all~$k$. 
 \begin{itemize}\itemsep=0pt
  \item[(a)]  For each $k$, $\xvec_k$ is uniquely determined by its coordinates from~$J_k$.
  \item[(b)]  $ \bigcap_{k=1}^m \tau_B(\xvec_k+ U_{J_k}(\rho_{J_k}^+)) \neq \emptyset$
    if and only if $\xvec_1,\ldots,\xvec_m$ agree on the coordinates from~$J$.
   \item[(c)] Suppose 
   $\xvec_1,\ldots,\xvec_m$ agree on the coordinates from~$J$.
     Define $\wvec=(w_1,\ldots,w_n)\trans $ to agree with 
       $\xvec_1,\ldots,\xvec_m$ on the coordinates from~$J$ and zero otherwise.
      Then $\wvec\in B$.  Moreover,
     there are intervals $I_1=I_1(\wvec),\ldots,I_k=I_k(\wvec)$ such that 
     $[-\frac13\rho,\frac13\rho]\subseteq I_k\subseteq [-\frac43\rho,\frac43\rho]$ for each $k$ and
     \[
 	  		\bigcap_{k=1}^m  \tau_B(\xvec_k+ U_{J_k}(\rho_{J_k}^+))
 	  		=\wvec + W(\wvec),
    \]
    where
     \[
 	  W(\wvec) = \{ (v_1,\ldots,v_n)\st v_j =0 \text{ for } j \in J 
 	     	\text{ and } v_j \in I_j, \text{ otherwise}\}.
 \] 
 \end{itemize}
 \end{lemma}
  \begin{proof}
  Part (a) follows from property P1 for $\xvec_k,J_k$ and Lemma~\ref{l:Mtrans}.
  
  For part (b), first assume that $ \bigcap_{k=1}^m \tau_B(\xvec_k+ U_{J_k}(\rho_{J_k}^+)) \neq \emptyset$.
  Using Lemma \ref{l:inter} multiple times, we conclude that 
  			\[
 	  		\bigcap_{k=1}^m  \tau_B\(\xvec_k+ U_{J_k}(\rho_{J_k}^+)\)
 	  		\subseteq U_n(\rho)
 	  	\]
 	  	since $\rho_{\emptyset}^+ = \rho$, $\bigcap_{k=1}^m J_k=\emptyset$ and $\varOmega_{\emptyset} = \{\Zero\}$.
In particular, since the left hand side is nonempty by assumption,
 	  	\[
 	  		\xvec_k+ U_{J_k}(\rho_{J_k}^+)  \subseteq U_n (\rho + 2 \rho_{J_k}^+) \subseteq B.
 	  	\]
 	  	Since $\rho$ is sufficiently small,
 $\bigcap_{k=1}^m \tau_B\(\xvec_k   + U_{J_k}(\rho_{J_k}^+)\) = \bigcap_{k=1}^m \(\xvec_k   + U_{J_k}(\rho_{J_k}^+) \)$.
 In particular, the points of 
     $ \bigcap_{k=1}^m \(\xvec_k   + U_{J_k}(\rho_{J_k}^+) \)$
 	  	  agree in coordinate $j$ for all $j \in J$.

Now assume that $\xvec_1,\ldots,\xvec_m$ agree on coordinates from $J$ and
define $\wvec$ as in part~(c). 
  Let $\xvec_k = (v_{k,1},\ldots, v_{k,n})\trans  = \sum_{j\in J_k} t_{k,j} \fvec_j$. 
     From Lemma \ref{l:oneU}(a),  we  get  $\sum_{j \in J_k} \abs{t_{k,j}} = O(n) \zeta \rho$. 
   Then, by assumption \eqref{M_property} for $M^{-1}$, we obtain that 
   \[
   v_{k,j} = O(n^{-1}) \sum_{j \in J_k} \abs{t_{k,j}} =  O(1) \zeta \rho
     \qquad \text{for all $j\notin J_k$}.
   \]
   We conclude that the 
   components of vectors in $\xvec_k+U_{J_k}(\rho_{J_k}^+)$ corresponding to indices $j \notin J_k$  take values in the interval 
   \[I_{k,j} := [-\rho_{J_k}^+, \rho_{J_k}^+] + v_{k,j}
    = [-\rho_{J_k}^+, \rho_{J_k}^+] + O(1)\zeta \rho.\]  
   Since $\rho_{J_k}^+\in [\frac12\rho, \rho]$, we find that 
   $[-\frac13\rho,\frac13\rho] \subseteq I_{k,j} \subseteq  [-\frac43\rho,\frac43\rho]$. 
   Therefore, for $j\notin J$, the $j$th component of vectors in
   $\bigcap_{k=1}^m \(\xvec_k   + U_{J_k}(\rho_{J_k}^+)\)$
   lies in the interval $I_j=\bigcap_{k=1}^m I_{k,j}$
   and $[-\frac13\rho,\frac13\rho] \subseteq I_j\subseteq  [-\frac43\rho,\frac43\rho]$.
   By part (a), $\xvec_1,\ldots,\xvec_m$ are uniquely determined by
   $\wvec$ and $J_1,\ldots,J_k$ and therefore $I_j=I_j(\wvec)$.
   Since $\bigcap_{k=1}^m J_k=\emptyset$,
   \[
   \bigcap_{k=1}^m \(\xvec_k   + U_{J_k}(\rho_{J_k}^+)\) 
      = \wvec + W(\wvec) \subseteq U_n(\tfrac43\rho).
   \]
   In particular, $\wvec\in U_n(\frac43\rho)\subseteq B$.
   Therefore, for each $k$, 
   \[
       \xvec_k   + U_{J_k}(\rho_{J_k}^+) \subseteq U_n(\tfrac43\rho+2\rho_{J_k}^+) \subseteq B.
   \]
   This implies $\tau_B\(\xvec_k   + U_{J_k}(\rho_{J_k}^+)\)
     = \xvec_k   + U_{J_k}(\rho_{J_k}^+)$, which completes the proof of part~(c).
   Note also that $\wvec\in \bigcap_{k=1}^m \tau_B\(\xvec_k   + U_{J_k}(\rho_{J_k}^+)\)$,
   which completes the proof of part~(b).
    \end{proof}

For $J\subseteq [n]$, define
\begin{align*}
    S_J := \begin{cases}
            U_n(\rho)-U_n(\frac13\rho), & \text{ if $J=\emptyset$}; \\
            \tau_B\(\varOmega_J' + U_J(\rho_J^+)\), &\text{ otherwise},
           \end{cases}
\end{align*}
where $\varOmega_J'$ is defined by
\[
\varOmega_J'  := \varOmega_J - \bigcup_{K\subseteq J, K\neq  J} \tau_B\(\varOmega_K + U_K(\rho_K^+)\).
\] 
The following lemma establishes the important properties of $S_J$.
  \begin{lemma}\label{S-lemma} 
    Assume $M$ and $M^{-1}$ satisfy  \eqref{M_property}.
    Then the following holds if $\rho,\zeta$ are sufficiently small.
 	 \begin{itemize} 
 	  \item[(a)]  We have  $\displaystyle B - U_n(\tfrac13\rho) - \Babs  \subseteq
 	  		 \bigcup_{J\subseteq [n]\st \abs{J}\leq \zeta n} 
			 \!\! S_J \subseteq B - U_n(\tfrac13\rho)$.
 	  		
 		\item[(b)] If $J\cap K \neq \emptyset$
 	  		 and $\abs{J},  \abs{K}\leq\zeta n$ then $S_J$ and $S_K$ are disjoint.
                    
                   \item[(c)] Suppose  $J_1,\ldots,J_m \subseteq [n]$
                    with $\abs{J_1}, \ldots, \abs{J_m} \leq \zeta n$  such that   $\bigcap_{k=1}^m J_k = \emptyset $
                    and
                    $\bigcap_{k=1}^m S_{J_k} \neq \emptyset$. 
                    Assume $J = \bigcup_{k=1}^m J_k\ne\emptyset$.                    Then,
                    \[ 
                      \bigcap_{k=1}^m S_{J_k}
                        = \bigcup_{\wvec \in \varOmega} (\wvec + W(\wvec)) 
                      \] 
                    where $\varOmega\subseteq B \cap \Span\{\evec_j \st j \in J\}$ is
                    measurable in $\Span\{\evec_j \st j \in J\}$,
                    $W(\wvec)$ is defined in Lemma~\ref{l:manyU}(c), and 
                    the union is disjoint.
               \item[(d)] Under the conditions of part (c),
                  if $\wvec=(w_1,\ldots,w_n)\trans \in\varOmega$ then
                  $\abs{w_j}>\frac{\rho\,\abs{m_{jj}}}{9\max_{j\in [n]}\abs{m_{jj}}}$ for $j\in J$
                  and $\wvec+U_n(2\rho)$ contains a point from $B-\Babs$.
        
 	 \end{itemize}
 \end{lemma}
 \begin{proof}
 	   To prove (a) take any $\xvec \in \varOmega_J'$ with $0<\abs{J}\leq \zeta n$. 
 	   Note that  $\xvec + U_J(\rho_J^+)$ is disjoint from  $U_n(\frac13\rho)$
 	   by property P3  and $\rho_J^+ \leq \rho$. 
      If 
 	   $\tau_B(\xvec + U_J(\rho_J^+)) \neq  \xvec + U_J(\rho_J^+)$  
	   we get  $\infnorm{\xvec + \xivec} \geq \frac12-O(\rho)$ for any $\xivec \in \calL $ and then 
 	$\tau_B(\xvec + U_J(\rho_J^+)) \subseteq B - U_n(\frac13\rho)$.
 	Therefore, $S_J \subseteq B-U_n(\frac13\rho)$ for all $0<\abs{J}\leq \zeta n$, and $S_\emptyset\subseteq B-U_n(\frac13\rho)$ by definition, which implies the upper containment of part  (a). 

Now take any point $\yvec=(y_1,\ldots,y_n)\trans  \in B - U_n(\frac13\rho) - \Babs $. 	 
Define $J =J(\yvec)$ to be the set of indices $j\in[n]$  such that $\abs{\yvec_j}>\frac13\rho$. 
Since $\yvec \in B-  \Babs $, we have 
$\frac13\rho\abs{J} \leq 
\sum_{j\in J} \abs{y_j} \leq \zeta n \frac13\rho$, by the definition of $\Babs $, so $\abs{J}\leq \zeta n$. 
By Lemma~\ref{l:Mtrans}, we can expand
\[
	\yvec = \sum_{j \in J} t_j \fvec_j + \wvec,
\]
where $\wvec =(w_1,\ldots,w_n)\trans $ with $w_j = 0$ for $j \in J$.
First, we are going to show that  $\xvec =\tau_B(\sum_{j \in J} t_j \fvec_j) \in \varOmega_J$ .  Note that P1 holds.  

From \eqref{crit_bound}, we know that $\onenorm\yvec=O(n\rho)$
and $\sum_{j\in J} \abs{y_j}=O(n \zeta \rho)$.
From Lemma~\ref{l:Mtrans}, this implies that
\begin{equation}\label{sumtjbound}
  \sum_{j \in J} \,\abs{t_j} 
  = \sum_{j\in J} O\(\abs{y_j}+\onenorm\yvec/n\) = O(\zeta n \rho ).
\end{equation}
Using Lemma~\ref{l:Mtrans} again,  we have if  
$j \notin J$ then $\abs{w_j} \leq \abs{y_j} +  O(\zeta\rho) < \frac12\rho\leq \rho_J^+ $,
for sufficiently small $\zeta$.
This implies $\yvec \in \tau_B(\xvec +  U_J(\rho_J^+))$ and  so 
$\xvec$ satisfies~P2.

We have already observed that $\infnorm{\wvec} \leq \rho_J^+$. If 
$\xvec = \sum_{j \in J} t_j \fvec_j$, then $\abs{x_j} = \abs{y_j} >\frac13\rho $ for all
$j\in J$ so P3(i) holds. Otherwise, we have  that $ \infnorm{\xvec} = \frac12-O(\rho)>\frac43\rho$. 
Using assumption~\eqref{M_property} for $M$ and~\eqref{sumtjbound}, we bound  
\[
 	\tfrac13\rho < \abs{y_j} \leq \abs{m_{jj} t_j} +  O(n^{-1}) \sum_{j\in J} \,\abs{t_j} 
\]
for all $j\in J$, which implies $\abs{t_j } > \rho/(6m_{jj})> \rho_J^-$. 
Since $\yvec \in B$ and $\infnorm{\wvec}$ is small, if
$\xvec \ne \sum_{j \in J} t_j \fvec_j$ then $\abs{t_j} = \frac12 + O(\rho)$
for some $j\in J$. This implies P3(ii), completing the proof that
$\xvec \in \varOmega_J$.  

Next, if $\xvec \in \varOmega_J'$  (which includes $J=\emptyset$)
then $ \tau_B(\xvec + U_J(\rho_J^+))\subseteq S_J\cup U_n(\frac13\rho)$, which implies $\yvec\in S_J$.
Otherwise, if $\xvec \notin \varOmega_J'$ (which implies $J\ne\emptyset$),
 suppose that $K$ is the smallest subset of $J$ for which 
$\yvec\in \tau_B(\zvec + U_K(\rho_K^+))$ for some $\zvec\in\varOmega_K$.
If $\zvec\in\varOmega_K'$, then  $\yvec\in S_K$ for the same reason.
 Otherwise,  there is some $K'\subset K$ and
 $\xvec' \in \varOmega_{K'}$  such that $\zvec \in \tau_B(\xvec' + U_{K'}(\rho_{K'}^+))$.  
Obviously, $\zvec \in \tau_B (\zvec + U_K(\rho_K^+))$.  
Using Lemma \ref{l:inter}, we find that, for some $\zvec'\in\varOmega_{K'}$,
\[
\yvec \in \tau_B (\zvec + U_K(\rho_K^+)) \subseteq \tau_B(\zvec' + U_{K'}(\rho_{K'}^+)). 
\]
This contradicts the minimality of $K$, completing the proof of part~(a).
 	   
 For (b), assume by contradiction that $S_J \cap S_K \neq \emptyset$ for some 
 $J \cap K \neq \emptyset$ and  $\abs{J},\abs{K} \leq\zeta n$. Then, we can find some 
 $\xvec \in \varOmega_J'$ and $\yvec \in \varOmega_K'$ that 
 $\tau_B(\xvec + U_{J}(\rho_J)) \cap \tau_B (\yvec + U_{K}(\rho_K))\neq \emptyset$.	  
From Lemma \ref{l:inter}, we find some $\zvec \in\varOmega_{J\cap K}$ such that
\[
	\xvec \in \tau_B(\xvec + U_{J}(\rho_J^+))  \subseteq
	\tau_B(\zvec + U_{J\cap K}(\rho_{J\cap K}^+))
\]
This contradicts $\xvec\in\varOmega_J'$, proving part~(b).
 
   We proceed to (c). 
   Note that $S_\emptyset\cap S_K=\tau_B\(\varOmega_\emptyset'
        + U_{\emptyset}(\rho_{\emptyset}^+)\) \cap
       \tau_B\(\varOmega_K' + U_{K}(\rho_{K}^+)\)$ if $K\ne\emptyset$,
      since $S_K\cap U_n(\frac13\rho)=\emptyset$ by part(a). 
   Note also that for any $\xvec\neq \yvec$ 
       	with $\xvec,\yvec \in B\cap \Span \{\fvec_j\st j\in J_k\}$
        the sets  $\tau_B(\xvec + U_{J_k}(\rho_{J_k}^+))$ 
       and $\tau_B(\yvec + U_{J_k}(\rho_{J_k}^+))$ are disjoint,
        by Lemma~\ref{l:Mtrans}.
       Then we get that
      	\[
          \bigcap_{k =1}^m S_{J_k}  = 
        \bigcap_{k =1}^m \tau_B(\varOmega_{J_k}' + U_{J_k}(\rho_{J_k}^+)  
   =   \bigcup_{\xvec_1 \in \varOmega_1', \ldots, \xvec_m \in \varOmega_m'} 
    \bigcap_{k =1}^m  \tau_B(\xvec_k + U_{J_k}(\rho_{J_k}^+)),
      \]  
      	where  the  union  over $\xvec_1 \in \varOmega_1', \ldots, \xvec_m \in \varOmega_m'$ is disjoint.

Define $\varOmega$ to be the set of all vectors $\wvec$ for which there are
 $\xvec_1\in\varOmega_1',\ldots,\xvec_m\in\varOmega_m'$ such that
$\wvec$ agrees
with $\xvec_1,\ldots,\xvec_m$ on the coordinates from~$J$ and
are zero otherwise.
    By Lemma \ref{l:manyU}, 
    \[
      \bigcup_{\xvec_1 \in \varOmega_1', \ldots, \xvec_m \in \varOmega_m'}
      \bigcap_{k =1}^m  \tau_B(\xvec_k + U_{J_k}(\rho_{J_k}^+))
      = \bigcup_{\wvec\in\varOmega} ( \wvec+W(\wvec) ).
    \]
       This completes the proof of part~(c).
       
 For $\wvec\in\varOmega$
 there are $\xvec_1,\ldots,\xvec_m$ as in the proof of part~(c).
 Applying Lemma \ref{l:oneU}(c) to each $J_k$, we also find that 
         $
         	\abs{w_j}>\frac{\rho\,\abs{m_{jj}}}{9\max_{j\in [n]}\abs{m_{jj}}}
         $
         for all $j \in J$. 
         Since $\wvec \in \pi_B(\xvec_k + U_{J_k}(\rho_{J_k}^+))$  
          and  $\rho_{J_k}^+ \leq \rho$, we find that
         \[
         	\tau_B(\xvec_k + U_{J_k}(\rho_{J_k}^+)) \subseteq  \tau_B(\wvec  + U_n(2\rho)),
         \] 
          therefore $\tau_B(\wvec  + U_n(2\rho))
$ contains a point from $B-\Babs $
by property~P2.    This completes part~(d).      
 \end{proof}

 \subsection{Proof of Theorem \ref{MishaMagic}}\label{AppendixB2}
 Now we are ready to prove the main theorem in this Appendix.
  
 \begin{proof}
   Let $S:= \bigcup_{J\subseteq [n]\st \abs{J}\leq \zeta n} S_J$.
 We use the inclusion-exclusion formula to estimate
 \[
 	\biggl|\int_{S-S_\emptyset} H(\xvec)\, d \xvec \biggr| 
 	\leq \sum_{J_1,\ldots,J_m}  \,\biggl| \int_{\bigcap_{k=1}^m S_{J_k}} H(\xvec)\, d \xvec \biggr|,
 \]
 where the sum is over sets $\{J_1,\ldots,J_m\}$ such that 
 $\bigcup_{k=1}^m J_k\ne\emptyset$ and $\abs{J_k}\leq\zeta n$ for $k\in[m]$.
 
Take any $J\subseteq[n]$ with $0<\abs{J}\leq \zeta n$.
By periodicity of~$H$ and Lemma~\eqref{l:pi}(a), we get
 \[
 	\biggl|\int_{S_J} H(\xvec) \,d \xvec \biggr|
 	\leq \abs{M_J} \int_{\varOmega_J'}  \bigg| \int_{U_J(\rho_{J}^+)} H(\zvec + \yvec)\, d\yvec \bigg|\,   d\zvec. 
 \]	
If $\zvec = (z_1,\ldots,z_n)  = \sum_{j\in J} t_j \fvec_j$, 
then by Lemma~\ref{l:oneU}(a), $z_j=O(\zeta\rho)$ for $j\notin J$.
Define $\wvec=(w_1,\ldots,w_n)\trans $ by  taking 
$w_j = z_j$ if $j\in J$ and $w_j =0$ otherwise.
Then 
\[
  \wvec \in \zvec + U_J( O(\zeta \rho)) \subseteq  B + U_n(\tfrac13\rho)
\]
 and $\wvec \in \Span\{\evec_j\st j\in J\}$ and
that $\zvec +U_J(\rho_{J}^+) = \wvec + W_J$ for some 
$W_J$  satisfying~\eqref{W_property}.
Since $\zvec \in \varOmega_J$,  by Lemma~\ref{l:oneU}(c), we also find that 
         $\abs{w_j} = \abs{z_j}>\frac{\rho\,\abs{m_{jj}}}{9\max_{j\in [n]}\abs{m_{jj}}}$
         for all $j \in J$ so condition (ii) holds. 
        Since  $\rho_{J_k}^+ < \rho$, we also find that
         \[
         	\tau_B(\zvec + U_{J}(\rho_{J}^+)) \subseteq  \tau_B(\wvec  + U_n(2\rho)),
         \] 
          therefore $\tau_B(\wvec  + U_n(2\rho)) $ contains a point from $B-\Babs $.     
The determinant of the Jacobian $M_J$  is bounded by $\infnorm{M}^{\abs J}$
by Lemma \ref{l:MJdet}. The measure of  $\varOmega_J' \subseteq \{t_j \fvec_j \st j\in J, \abs{t_j}\leq\frac12\}$ is at most $1$.    
By our assumptions, we get that
\[	
	\biggl|\int_{S_J} H(\xvec) d \xvec \biggr| \leq \infnorm{M}^{\abs J} \,L(\abs J).
\]	
 
 Now consider intersections of the regions $S_{J_1},\ldots, S_{J_m}$,
 where $1\leq\abs{J_1},\ldots,\abs{J_m}\leq\zeta n$ and  $m\geq 2$,
 Define $J=\bigcup_{j=1}^m J_k$.
         If $J_1, \ldots J_m$  are not pairwise disjoint then     
     $\bigcap_{k=1}^m S_{J_k}$ is empty by Lemma \ref{S-lemma}(b).
         In the case that $\bigcap_{k=1}^m S_{J_k} \neq \emptyset$,
         we find from  Lemma \ref{S-lemma}(c)
         that there is a measurable set $\varOmega$ and a function
         $W(\wvec)$  satisfying~\eqref{W_property},
         such that
          \[ 
                      \bigcap_{k=1}^m S_k
                        = \bigcup_{\wvec \in \varOmega}  (\wvec + W(\wvec)).
          \] 
     By Lemma~\ref{S-lemma}(d), assumptions (i) and (ii) of the theorem are
     satisfied. 

         Since  $\varOmega \subseteq B \cap \Span\{\evec_j\st j\in J\}$
         and $B\subseteq U_n(\frac12\infnorm{M})$, we have  that 
         the measure of~$\varOmega$ is bounded above by $\infnorm{M}^{\abs{J}}$.
         Therefore the assumptions imply that 
 	    \[
 	    	 \biggl|\int_{\bigcap_{k=1}^m S_{J_k}} H(\xvec)\, d \xvec \biggr|  
		      \leq \infnorm{M}^{\abs{J}}\,L(\abs J).
 	    \]
 	    The number of choices for  disjoint  $J_1,\ldots J_m\in[n]$ 
 	    such that $|\bigcup_{k=1}^m {J_k}| = j$ is bounded above by~$n^j$ (choose
	    a subset of size $j$ then the cycle partition of a permutation on that subset).
	    This count includes the case $m=1$ which we bounded above.
	    Thus, we get
 	     \[
 	     	\biggl|\int_{S-S_\emptyset} H(\xvec)\, d \xvec \biggr| 
		  \leq \sum_{j=1}^{ n} \,(n\infnorm{M})^j\,L(j).
 	     \] 
	    Next, by Lemma~\ref{S-lemma}(a) and $S_\emptyset=U_n(\rho)-U_n(\frac13\rho)$,
	     we have a disjoint union
	    \[  B - U_n(\rho) = (S - S_\emptyset) \cup \(B - U_n(\tfrac13\rho) - S\),
	    \]
	    and $B - U_n(\frac13\rho) - S\subseteq\Babs$.
 	    Thus we get that 
 	     \[	
 	     		\biggl|\int_{B - U_n(\frac13\rho) - S} H(\xvec) \,d \xvec \biggr| 
 	     		\leq\int_{\Babs } \abs{H(\xvec)} \,d \xvec .
 	     \]
 	     This completes the proof. 
 \end{proof}
\end{appendices}

%: References

\nicebreak

\end{document}